\documentclass[a4paper,10pt]{article}
 
\usepackage[utf8]{inputenc}
\usepackage[T1]{fontenc}
\usepackage[english]{babel}
\usepackage{lmodern}
\usepackage{geometry}
\usepackage{url,xspace}
\usepackage{amsmath,amssymb,stmaryrd,spalign,bbm,mathtools,mathrsfs,dsfont}
\usepackage{xfrac}
\usepackage{comment}
\usepackage[thmmarks,amsmath,framed,thref,hyperref]{ntheorem}
\usepackage[colorlinks=true,citecolor=teal,linkcolor=blue]{hyperref}
\usepackage{upgreek}
\usepackage{array}
\usepackage{enumitem}
\usepackage{cite}
\usepackage{graphicx}
\usepackage[toc,page]{appendix}
\usepackage{etoolbox}
\usepackage{pgfplots}
\pgfplotsset{width=10cm,compat=1.9}

\usepackage{tikz, xcolor, float}
\usetikzlibrary{arrows.meta,patterns,calc}

\usepackage{makecell}
\usepackage{tikz-cd}

\usepgfplotslibrary{external}

\newtheorem{theorem}{Theorem}[section]
    \newtheorem{corollary}[theorem]{Corollary}
    
    \newtheorem{proposition}[theorem]{Proposition}

    \theorembodyfont{\upshape}
    \newtheorem{definition}[theorem]{Definition}
    \newtheorem{notation}[theorem]{Notation}
    \newtheorem{remark}[theorem]{Remark}

\theoremstyle{nonumberplain}
\theoremheaderfont{\itshape}
\theorembodyfont{\normalfont}
\theoremseparator{.\,—}
\theoremsymbol{}
\newtheorem{proof-wo}{Proof}
\theoremsymbol{\ensuremath{\blacksquare}}
\newtheorem{proof}{Proof}

\newcommand{\B}{\mathrm{B}}
\newcommand{\C}{\mathrm{C}}     \newcommand{\Ccinfty}{\mathrm{C}^\infty_c}     \newcommand{\Ccinftydiv}{\mathrm{C}^\infty_{c,\sigma}}
\newcommand{\D}{\mathrm{D}}

\renewcommand{\d}{\mathrm{d}}
\renewcommand{\H}{\mathrm{H}}
\newcommand{\I}{\mathrm{I}}
\renewcommand{\L}{\mathrm{L}}
\renewcommand{\P}{\mathrm{P}}
\newcommand{\N}{\mathrm{N}}
\newcommand{\T}{\mathrm{T}}
\newcommand{\K}{\mathrm{K}}
\newcommand{\R}{\mathrm{R}}
\newcommand{\W}{\mathrm{W}}
\newcommand{\X}{\mathrm{X}}
\newcommand{\Y}{\mathrm{Y}}
\newcommand{\Z}{\mathrm{Z}}

\DeclareMathAlphabet\gothic{U}{euf}{m}{n}

\renewcommand{\AA}{\mathbb{A}}
\newcommand{\RR}{\mathbb{R}}

\newcommand{\NN}{\mathbb{N}}

\newcommand{\CC}{\mathbb{C}}

\newcommand{\PP}{\mathbb{P}}

\newcommand{\ff}{\mathbf{f}}

\newcommand{\uu}{\mathbf{u}}
\newcommand{\vv}{\mathbf{v}}
\newcommand{\ww}{\mathbf{w}}

\let\div\undefined
\DeclareMathOperator{\div}{div\,}

\DeclareMathOperator{\supp}{supp\,}

\counterwithin{equation}{section}

\setlist[enumerate,1]{label={\textit{(\roman*)}}}% Required for inserting images

\title{The Boussinesq system in 3-dimensional bounded rough domains:
Well-posedness in critical spaces and long-time behavior\thanks{
\textbf{keywords:} \textit{Boussinesq system, $\L^1$-Maximal regularity, Bounded Lipschitz domains, Bounded H\"{o}lder continuous domains, Well-posedness, Critical spaces, Geophysical flows}}}
\author{
  Anatole \textsc{Gaudin}\thanks{Universität Duisburg-Essen, Fakultät für Mathematik, \textbf{e-mail:} anatole.gaudin@uni-due.de}
}
\date{}

\begin{document}

\maketitle

\begin{abstract}
We study the three-dimensional Boussinesq system in bounded rough domains, including bounded Lipschitz and $\mathrm{C}^{1,\alpha}$ domains, within a critical functional framework. We establish existence and uniqueness results that are global in time for small initial data and local in time for arbitrary initial data. Well-posedness in critical endpoint Besov spaces with third index equal to $\infty$ is obtained in domains with H\"older continuous boundaries, relying on $\mathrm{L}^2$-maximal regularity in time. We also prove well-posedness in critical Besov spaces with third index equal to $1$, using $\mathrm{L}^1$-maximal regularity. In this $\mathrm{L}^1$-in-time setting, the analysis applies to arbitrary bounded Lipschitz domains. In any case, we show that the fluid velocity stabilizes exponentially for large times and that the temperature converges to the initial averaged temperature of the fluid.

The linear theory --fitting the adapted product estimates and vice versa-- is properly established prior to the nonlinear analysis. With this fully prepared linear framework in hand, the nonlinear estimates that follow are then handled in the critical framework with a simplified treatment --especially in the case where the fluid velocity and the temperature belong to slightly larger spaces than $\mathrm{L}^2(\mathrm{W}^{1,3})$ and $\mathrm{L}^2(\mathrm{L}^{3/2})$ respectively-- when compared with previously known similar results in smooth domains. This approach relies on a robust linear theory and sharp product estimates based on operator-theoretic methods and Besov space techniques.  Finally, as part of the analysis, we establish several new results for the underlying linear operators, including refined characterizations for the domains of fractional powers of the Neumann Laplacian and of the Stokes operator in bounded Lipschitz domains.
\end{abstract}

\tableofcontents

%----------------------------------------------------
%------------------ Introduction --------------------
%----------------------------------------------------
\section{Introduction}

\subsection{State of the art, the main theorems, and functional analytic motivations}

Given a bounded \textit{Lipschitz} domain $\Omega$ of $\RR^3$, with outward unit normal $\mathbf{n}$, the goal of this article is to study the following Cauchy problem 
\begin{equation}\tag{BSQ}\label{eq:BoussinesqSystemIntro}
    \left\{\begin{aligned}
        \partial_t \uu - \nu\Delta \uu + \nabla \mathfrak{p}+ \div (\uu \otimes \uu)  & =  \theta\mathfrak{e}_3 & \quad & \text{in $(0,T)\times\Omega$}, \\
        \partial_t \theta - \kappa\Delta \theta +\div(\theta\mathbf{u})  & = 0 & \quad & \text{in $(0,T)\times\Omega$}, \\
        \div \uu & =  0 & \quad & \text{in $(0,T)\times\Omega$}, \\
        \uu & =0&  &  \text{on $(0,T)\times\partial \Omega$}, \\
        \partial_\mathbf{n}\theta & =0&  &  \text{on $(0,T)\times\partial \Omega$}, \\
        (\uu, \theta) (0) & = (\uu_0, \theta_0). &  &
    \end{aligned}\right.
\end{equation}
where $\nu,\kappa>0$ are given, it is called the \textit{incompressible} Boussinesq system. Here $\uu\,:\,(0,T)\times\Omega\longrightarrow\RR^3$ is the fluid velocity, $\mathfrak{p}\,:\,(0,T)\times\Omega\longrightarrow\RR$ the pressure of the fluid and $\theta\,:\,(0,T)\times\Omega\longrightarrow\RR$ is the ``temperature'' of the fluid. This is a simplified model coming from geophysical flows.  This system describes a heated fluid where it creates temperature variations $\theta$; these variations generate a buoyancy force\footnote{This is an arbitrarily chosen direction: one could choose any constant unit vector of $\RR^3$ instead of $\mathfrak{e}_3$.} $\theta \mathfrak{e}_3$, which sets the fluid into motion, and this motion transports the temperature. For the modeling, see for instance \cite[Chap.~10]{Friedlander80} and \cite[Chap.~2,~Sec.~16]{Salmon1998}, the latter concerning the case when viscosity and diffusion are neglected ($\nu=\kappa=0$).

\medbreak

In order to investigate the well-posedness of such a non-linear system, the notion of a \emph{scaling--invariant space} plays a fundamental role, because it captures the intrinsic balance between nonlinearity and dissipation (or the lack thereof). Each equation comes with a natural scaling law: if $\Omega=\RR^3$, $T=\infty$, and $(\uu(t,x),\theta(t,x))$ is a solution to \eqref{eq:BoussinesqSystemIntro}, then for any $\lambda>0$
\begin{align}\label{eq:scaling}
    \uu_\lambda(t,x)=\lambda \uu(\lambda^2 t,\lambda x),\qquad \theta_\lambda(t,x)=\lambda^3 \theta(\lambda^2 t,\lambda x)
\end{align}
also yields a solution. A function space $\X$ is called critical for \eqref{eq:BoussinesqSystemIntro} when $$\|(\uu_\lambda,\theta_\lambda)\|_{\X(\RR_+\times\RR^3)}=\|(\uu,\theta)\|_{\X(\RR_+\times\RR^3)},$$ meaning that the norm does not distinguish between different spatial scales. This is precisely the threshold at which the equation is neither dominated by dissipation nor by nonlinearity: in subcritical spaces dissipation prevails and regularity improves, while in supercritical spaces nonlinear effects may concentrate energy and generate singularities. Working in scaling--invariant spaces therefore allows one to formulate existence, uniqueness, and stability results that are intrinsic to the dynamics of the equation and independent of any arbitrary choice of scale, and to reach the optimal framework where the genuine analytical difficulty of the problem is revealed.

\medbreak

Within this regime, we can then consider a wealth of different function spaces for the unknown $(\uu,\theta)$ in order to reach well-posedness, such as
\begin{align}\label{eq:CritSclaingspacesIntro}
    \L^\infty(\L^3(\RR^3)^3\times\L^1(\RR^3)),\,\,\, \L^\infty(\dot{\B}^0_{3,q}(\RR^3)^3\times\dot{\B}^{-1}_{3/2,r}(\RR^3)),\,\,\, \L^\infty(\L^3(\RR^3))\times\L^2(\L^\frac{3}{2}(\RR^3)),\,\,\,\text{etc.}
\end{align}
and many other combinations with different integrability index and regularity indices are possible, some being more difficult to deal with than others in order to reach well-posedness results.

\medbreak

The system has been widely investigated mostly in the whole space or the $n$-dimensional torus: Cannon and DiBenedetto \cite{CannonDiBenedetto80} achieved the first investigation in $\L^q(\L^p(\RR^n))$, $2<p,q<\infty$ in critical spaces for the solution but non-critical for the initial data. In the work by Sawada and Taniuchi, \cite{SawadaTaniuchi2004}, the system has been investigated for non-decaying initial data in spaces like (and built upon) $\L^\infty(\RR^n)$. Danchin and Paicu \cite{DanchinPaicu2008} did investigate the case $\kappa=0$ for initial data in critical Lebesgue-Lorentz $\L^{n,\infty}(\RR^n)\times \L^{\frac{n}{3}}(\RR^n)$ and in Besov spaces $\dot{\B}^{0}_{n,1}(\RR^n)^n\times \dot{\B}^{0}_{n,1}(\RR^n)$ with solutions in appropriate critical Lebesgue-Besov spaces. After this, Deng and Cui pursued similar interests, \cite{DengCui2012}, considering data in less regular log-Besov spaces like $(\B^{-1}_{1,\infty}\cap \B^{-1,1}_{\infty\infty})(\RR^n)^n\times\B^{-1}_{p,r}(\RR^n)$. More recently, Brandolese and He \cite{BrandoleseHe2020BSQ} constructed, with uniqueness, mild solutions in $\C^{0}(\L^3(\RR^3)\times\L^1(\RR^3))$ for  data  $(\uu_0,\theta_0)\in\L^3(\RR^3)\times\L^1(\RR^3)$. Pursuing this analysis, Brandolese and Monniaux, \cite{BrandoleseMonniaux2021}, obtained existence and uniqueness in $\C^{0}(\L^3(\RR^3))\times\L^2(\L^{\frac{3}{2}}(\RR^3))$, a critical space, provided in $(\uu_0,\theta_0)\in\L^3(\RR^3)\times\dot{\B}^{-1}_{3/2,2}(\RR^3)$, for either small initial data and globally-in-time, or finite times and arbitrary data they also proved. Very recently, in a similar spirit, Fern\'{a}ndez-Dalgo and Jarr\'{i}n \cite{FernandezDalgoJarrin2025} performed a similar analysis in the $\L^2$-Sobolev spaces $\H^{r}(\RR^3)^3\times\dot{\H}^{s}(\RR^3)$, $s\in[-1/2,0]$, $r\geqslant 1/2$, but only for finite times.

For domains with boundary,  it has been obtained by Fife and Joseph, \cite{FifeJoseph69}, the well--posedness for smooth domains $\Omega$ within the Hilbertian setting. Hishida, \cite{Hishida91}, did reach well-posedness in the space $\C^0(\L^p(\Omega)^n\times\L^q(\Omega))$, $\Omega$ to be bounded with smooth boundary, in the non-critical regime, then Hishida and Yamada, \cite{HishidaYamada92}, obtained the corresponding results in the case of exterior domains. In \cite{Kagei93}, Kagei discusses, for $\Omega$ to be smooth bounded or exterior, the existence of weak solutions for initial data $\uu_0\in\L^2(\Omega)$ and $\theta_0\in\L^1(\Omega)$, and their uniqueness in some critical $\L^p$ classes. In \cite{Morimoto92}, Morimoto investigated existence and uniqueness for weak solutions in the variational Hilbertian setting where the temperature is subject to mixed, non-zero prescribed, boundary conditions on the boundary.

When it comes to the analysis in bounded domains, results in critical spaces seem to be largely missing from the literature. In this work, we not only aim at filling this gap, but also at addressing the case of rough geometries, where elliptic operators typically exhibit weakened smoothing effects. Compared with the sole Navier--Stokes system, the temperature equation introduces several additional obstructions. First, due to the scaling associated with the coupling, see~\eqref{eq:scaling}, there is necessarily a gap of two derivatives, in the sense of Sobolev indices, between the functional space for the temperature and that for the velocity, as their respective scaling differ by $\lambda^2$. Moreover, the same scaling~\eqref{eq:scaling} suggests that the product rule involving the velocity $\uu$ and the temperature $\theta$ is more naturally formulated in function spaces of non-positive regularity. Indeed, if $\X^{s,p}$ denotes a scaling-invariant space with Sobolev index $s-3/p$, requiring $\L^q(\X^{s,p})$ to be scaling invariant for the product $\uu\theta$ leads to the constraint
\begin{align*}
    4+s = \frac{2}{q}+\frac{3}{p}.
\end{align*}

All these issues become particularly challenging in the presence of boundaries, even for smooth domains, since the two-derivative gap forces us to work almost exclusively with non-positive regularity spaces for the temperature. Even allowing for shifts in the regularity indices would then impose compatibility conditions that effectively render the system overdetermined. A possible workaround would be to introduce a large separation between the underlying Lebesgue scales, which should still preserve this gap of $2$ exact derivatives thanks to Sobolev embeddings; however, in rough domains this strategy is not available either, due to the limited solvability of the Neumann and Stokes problems, see Section~\ref{Sec:Laplacians}. These difficulties are absent in the sole Navier--Stokes system, although part of the analytical strategy for the fluid component of the Boussinesq system remains similar.

The main contributions of this paper are the two main theorems that read as follows:
\begin{theorem}\label{thm:mainThmBSQLp}Let $\Omega\subset\RR^3$ be a bounded $\C^{1,\alpha}$ domain with $\alpha>1/3$. Let $(\uu_0,\theta_0)\in\B^{0}_{3,2}(\Omega)^3\times\B^{-1}_{3/2,2,0}(\Omega)$ such that $\div \uu_0 =0$ and $\uu_0\cdot\mathbf{n}_{|_{\partial\Omega}}=0$. One sets $\Tilde{\theta}_0:= \frac{1}{|\Omega|}\int_\Omega \theta_0\in\RR$. Then, either,
\begin{itemize}
    \item $T=\infty$, we assume $(\uu_0,\theta_0-\Tilde{\theta}_0)$ has  sufficiently small norm; or
    \item we assume $(\uu_0,\theta_0)$ is arbitrary, so there exists $T>0$ such that ;
\end{itemize}
the system \eqref{eq:BoussinesqSystemIntro} admits a unique weak solution
\begin{align*}
    (\uu,\theta-\Tilde{\theta}_0)\in\C^0([0,T];\B^{0}_{3,2}(\Omega)^3\times \B^{-1}_{3/2,2}(\Omega))\cap\L^2(0,T;\B^{1}_{3,\infty}(\Omega)^3\times\B^{0}_{3/2,\infty}(\Omega)).
\end{align*}
Furthermore, when $T=\infty$, there exists ${\lambda}>0$ such that
\begin{align*}
    e^{\lambda t}\lVert (\uu(t),\theta(t))- (0,\Tilde{\theta}_0) \rVert_{\B^{0}_{3,2}(\Omega)\times \B^{-1}_{3/2,2}(\Omega)} \xrightarrow[t\rightarrow+\infty]{} 0.
\end{align*}
\end{theorem}

\begin{theorem}\label{thm:mainThmBSQ}Let $\Omega\subset\RR^3$ be a bounded Lipschitz domain. Let $(\uu_0,\theta_0)\in\B^{0}_{3,1}(\Omega)^3\times\B^{-1}_{3/2,1,0}(\Omega)$ such that $\div \uu_0 =0$ and $\uu_0\cdot\mathbf{n}_{|_{\partial\Omega}}=0$. One sets $\Tilde{\theta}_0:= \frac{1}{|\Omega|}\int_\Omega \theta_0\in\RR$. Then, either,
\begin{itemize}
    \item $T=\infty$, we assume $(\uu_0,\theta_0-\Tilde{\theta}_0)$ has  sufficiently small norm; or
    \item we assume $(\uu_0,\theta_0)$ is arbitrary, so there exists $T>0$ such that ;
\end{itemize}
the system \eqref{eq:BoussinesqSystemIntro} admits a unique mild solution
\begin{align*}
    (\uu,\theta-\tilde{\theta}_0)\in\C^0([0,T];&\B^{0}_{3,1}(\Omega)^3\times\B^{-1}_{3/2,1,0}(\Omega))\cap\L^1(0,T;\D^{0}_{3,1}(\AA_\mathcal{D},\Omega)\times \B^{1}_{3/2,1}(\Omega)).
\end{align*}
If additionally $\Omega$ has $\C^{1,\alpha}$ boundary with $\alpha>1/3$, then this solution is unique within the class of weak solutions.
\medbreak
\noindent Furthermore, when $T=\infty$,  there exists ${\lambda}>0$  such that
\begin{align*}
    e^{\lambda t}\lVert (\uu(t),\theta(t))- (0,\Tilde{\theta}_0) \rVert_{\B^{0}_{3,1}(\Omega)\times \B^{-1}_{3/2,1}(\Omega)} \xrightarrow[t\rightarrow+\infty]{} 0.
\end{align*}
\end{theorem}

Note that Theorem~\ref{thm:mainThmBSQ} also contributes to the folklore of analysis of the Navier--Stokes equations, for which no comparable results were previously available in non-smooth domains. In the case $\theta=0$, however, the Navier--Stokes result provided by Theorem~\ref{thm:mainThmBSQ} is not optimal: for $p$ in a small interval around $3$, it can be readily improved by working in the critical Besov space $\B^{3/p-1}_{p,1}$. Moreover, within the $\L^p$ framework for the pure Navier--Stokes equations, stronger results for mild solutions are known in arbitrary bounded Lipschitz domains; see, for instance \cite{Tolksdorf2018-1}. In contrast, the temperature variable $\theta$ and its evolution equation constitute the genuine limiting factors in the present analysis. In particular, the restriction to $\C^{1,\alpha}$-domains stems from the need for refined product estimates in combination with the study of the solution operator, both relying on the regularity properties of the Neumann Laplacian.

\medbreak

Further comments are in order. We highlight the following:
\begin{itemize}
    \item $\D^{0}_{3,1}(\AA_\mathcal{D},\Omega)$ is the domain of  the Stokes operator $\AA_{\mathcal{D}}$ in the Besov spaces $\B^{0}_{3,1}(\Omega)^3$, whenever the domain is smooth one can replace it by $\B^{2}_{3,1}(\Omega)^3$. This choice is necessary, and supported by comments on the lack of regularity for the Stokes operator in bounded rough domains, below;
    \item Concerning the function spaces exhibited in Theorems~\ref{thm:mainThmBSQLp}~and~\ref{thm:mainThmBSQ}, they are all critical, as those exposed in \eqref{eq:CritSclaingspacesIntro};
    \item Introducing  better boundary regularity with $\C^{1,\alpha}$-domains, $\alpha>1/3$, is necessary for the existence and uniqueness in Theorem~\ref{thm:mainThmBSQLp}, and only necessary for the uniqueness of weak solutions in Theorem~\ref{thm:mainThmBSQ}. The obstruction comes from the necessity to use a regularity estimate for the Neumann Laplacian (Proposition~\ref{prop:RieszTransfC1alphaDomainNeumannLap}), which is known to  hold only in this framework. While the author thinks it should be possible to reach $\C^{1,\alpha}$ for any $\alpha>0$, the question to know of whether the inability to reach boundary of minimal smoothness --such as $\C^1$ or Lipschitz-- is a purely technical artifact or not, remains unclear;
    \item Both results lie at the interface of the many previous ones mentioned in our short and non-exhaustive review above: We consider here initial data in  $\B^{0}_{3,2}(\Omega)^3\times\B^{-1}_{3/2,2}(\Omega)$ and  $\B^{0}_{3,1}(\Omega)^3\times\B^{-1}_{3/2,1}(\Omega)$ respectively. Hence, we allow distributional temperatures $\theta_0$. This space here is slightly smaller than the one exhibited by Monniaux and Brandolese in $\RR^3$, corresponding to $\L^3(\Omega)^3\times\B^{-1}_{3/2,2}(\Omega)$. Finally, the space for initial data  and the solution space, especially for the temperature, are larger than the ones considered by Danchin and Paicu in \cite[Thm.~1.3]{DanchinPaicu2008}. 
\end{itemize}

\noindent From the technical point of view, the hardest part for proving results such as Theorems~\ref{thm:mainThmBSQLp}~and~\ref{thm:mainThmBSQ} is the knowledge of a sufficiently robust linear theory that allows to close the non-linear estimates.
\begin{itemize}
    \item The standard approach is to use a fixed point argument and $\L^q$-maximal regularity theory. Yet, closing properly the product rules, especially in the critical setting, appears to be extremely tedious, see for instance \cite{DanchinPaicu2008,BrandoleseMonniaux2021}. However, Besov spaces with third index $1$, $\B^{s}_{p,1}$, $s\geqslant0$, $p\in[1,\infty]$, are known to have a better behavior with respect to product rules, see for instance their use in \cite{DanchinPaicu2008}, or results such as Proposition~\ref{prop:Productq=1Besov} in the Appendix. But using them as ground spaces for our analysis necessarily forces us to consider $\L^1$-in-time maximal regularity, see  \cite{RiFarwig2022} and \cite[Thm.~4.5]{BechtelBuiKunstmann2024} for necessary and sufficient conditions, and see below for a short presentation of $\L^1$-maximal regularity. As we shall see, this $\L^1$-framework is rather a blessing than a burden, yielding quite simple proofs.
    \item The roughness of the boundary: $\Omega$ is assumed here to only have Lipschitz boundary, or at most $\C^{1,\alpha}$ boundary, for $\alpha$ away from $1$. Therefore, one might not expect full elliptic regularity for the Neumann Laplacian and the Stokes operator, \textit{i.e.} no gain of two derivatives in the appropriate class. For instance, it might happen that
    \begin{align*}
        &\D_{2}(\Delta_\mathcal{N})=\{\eta\in\H^{1}(\Omega)\,|\, \Delta \eta \in\L^2(\Omega),\text{ and }\partial_{\mathbf{n}}\eta_{|_{\partial\Omega}}=0 \text{ in }\H^{-1/2,2}(\partial\Omega)\}\not\subset \W^{2,2}(\Omega),\\
        &\D_{2}(\AA_\mathcal{D})\,=\{\vv\in\H^{1}_0(\Omega)^3\,|\, \div \vv =0,\text{ and }\exists \mathfrak{q}\in\L^2(\Omega), -\Delta\vv+\nabla \mathfrak{q}\in\L^2_{\mathfrak{n},\sigma}(\Omega)\} \not\subset \W^{2,2}(\Omega),
    \end{align*}
    where $\L^p_{\mathfrak{n},\sigma}(\Omega):=\{ \ff\in\L^p(\Omega)^3 \,|\,\div \ff =0\text{ and } \ff\cdot\mathbf{n}_{|_{\partial\Omega}}=0\}$. The same might obviously happen to other function spaces, $\L^p$ and $\W^{2,p}$, $\B^{0}_{p,q}$ and $\B^{2}_{p,q}$, $p,q\in[1,\infty]$, etc. See for instance \cite[Point~{(i)},~p.~IV-5]{JerisonKenig1989}, for the Neumann Laplacian.

    \medbreak

    This is the reason why, in Theorem~\ref{thm:mainThmBSQ}, we ask for membership in $\L^1(0,T;\D_{3,1}^{0}(\AA_\mathcal{D},\Omega))$ instead of $\L^1(0,T;\B^{2}_{3,1}(\Omega)^3)$, where, here $\D_{p,q}^{s}(\AA_\mathcal{D},\Omega)$ denotes the domain of the Stokes operator in $\B^{s}_{p,q}(\Omega)^3$, see below for a definition. The latter choice with full elliptic regularity is possible whenever $\Omega$ has at least $\B^{2-1/r}_{r,1}$-boundary\footnote{\textit{i.e.} the local charts that describe the boundary $\partial\Omega$ locally belong to $\B^{2-1/r}_{r,1}(\RR^2)$, the latter being contained in $\C^{1}_{0}(\RR^2)$}, provided $r>3$, see for instance \cite[Chap.~2~\&~6]{BreitGaudin2025}.

    \medbreak

    Hence, due to the possible lack of regularity, we are \textit{a priori} restricted with not that much room left to maneuver in order to exploit product rules in concordance with the previous point. As we shall see, everything will remain under our reach, with the nearly exact amount of regularity in order to obtain applicability of the product rules, with actually almost no margin.
\end{itemize}

\subsection{Some notations and background related to semigroup theory and maximal regularity}\label{Sec:SectOp}

We claimed several times about the use of $\L^q$-maximal regularity as a goal to reach in order to perform standard well-posedness arguments from the linear theory perspective. In the framework of Navier-Stokes equations, and more precisely for the Boussinesq system with non-zero diffusion and viscosity, as  already mentioned above \cite{CannonDiBenedetto80,Hishida91,HishidaYamada92,Kagei93,DanchinPaicu2008,BrandoleseMonniaux2021,FernandezDalgoJarrin2025}, the use of $\L^q$-maximal regularity has been  central to obtain well-posedness results.

\medbreak

We introduce the following subset of the complex plane
\begin{align*}
    \Sigma_\mu &:=\{ \,z\in\mathbb{C}^\ast\,|\,\lvert\mathrm{arg}(z)\rvert<\mu\,\}\text{, if } \mu\in(0,\pi)\text{, }
\end{align*}
we also define $\Sigma_0 := (0,+\infty)$, and we are going to consider the closure $\overline{\Sigma}_\mu$.

\medbreak

For the reader's convenience, we recall the concept of sectorial operators and $\L^q$-maximal regularity here, especially the one of Da Prato--Grisvard type. Let $\X$ be a Banach space, and $(\D(A),A)$ be a closed operator in $\X$.

In virtue of the result \cite[Thm.~3.7.11]{ArendtBattyHieberNeubranker2011}, the two following assertions are equivalent:
\begin{enumerate}
    \item The operator $(\D(A),A)$ is such that its spectrum satisfies $\sigma(A)\subset \overline{\Sigma}_\omega$ for some $\omega \in[0,\frac{\pi}{2})$, and for all $\mu\in(\omega,\pi)$, there exists $C_\mu>0$, such that the following holds
    \begin{align}\label{eq:SectCondtn}
        \lVert \lambda (\lambda\I-A)^{-1} x\rVert_{\X} \leqslant C_\mu \lVert x\rVert_{\X},\qquad \forall x\in\X, \quad\forall \lambda \in\CC\setminus\overline{\Sigma}_\mu;
    \end{align}
    \item For all $\theta\in(0,\frac{\pi}{2}-\omega)$, $-A$ generates a bounded holomorphic $\mathrm{C}_0$-semigroup in $\X$, denoted by $(e^{-zA})_{z\in\Sigma_\theta\cup\{0\}}$.
\end{enumerate}

A closed operator satisfying condition \textit{(i)}, especially \eqref{eq:SectCondtn}, is called a \textbf{sectorial operator} (of angle $\omega$). Having such an operator $(\D(A),A)$ on $\X$ in hand, for $0<T\leqslant \infty$, we consider for $f\in\L^q(0,T;\X)$,  $1\leqslant q\leqslant \infty$, the abstract Cauchy problem
\begin{align}\tag{ACP}\label{ACP}
    \left\{\begin{array}{rl}
            \partial_t u(t) +Au(t)  &= f(t) \text{, }\quad 0<t<T \text{, }\\
            u(0) &= u_0\text{.}
    \end{array}
    \right.
\end{align}

It turns out, by \cite[Prop.~3.1.16]{ArendtBattyHieberNeubranker2011}, that the integral solution $u\in \mathrm{C}^0_{ub}([0,T];\X)$ for \eqref{ACP} is unique, also called the \textbf{mild solution} of \eqref{ACP} and given by
\begin{align*}
    u(t)= e^{-tA}u_0 +\int_{0}^{t} e^{-(t-s)A}f(s)\,{\mathrm{d}s} \text{, } \quad 0\leqslant t<T\text{.}
\end{align*}
When $u_0=0$, we may also write $u=(\partial_t + A)^{-1}f$.

We say that $(\D(A),A)$ (or in short $A$) has the $\boldsymbol{\L^q}$\textbf{-maximal regularity property} (on $(0,T)$, in $\X$) if for almost every $t\in(0,T)$, $u(t)\in\D(A)$, and $A u\in\L^q(0,T;\X)$, for any $f\in\L^q(0,T;\X)$, provided $u_0=0$. By the closed graph theorem and linearity, this is equivalent to the existence of some $C>0$,  such that
\begin{align}\tag{MR${}_q^T$}\label{eq:MaxRegLq}
    \lVert \partial_tu, Au\rVert_{\L^q(0,T;\X)} \leqslant C \lVert f\rVert_{\L^q(0,T;\X)},\qquad\text{ for any }f\in\L^q(0,T;\X).
\end{align}
We write $(\text{MR}_q)$ for $(\text{MR}_q^\infty)$, noting that $(\text{MR}_q)$ implies \eqref{eq:MaxRegLq} for all $T>0$. Conversely, if \eqref{eq:MaxRegLq} holds for all $T>0$, with a constant uniform with respect to $T>0$, then $(\text{MR}_q)$ holds.

Written differently, and by linearity of \eqref{ACP}, \eqref{eq:MaxRegLq} reduces to the knowledge of the boundedness of the operator $A(\partial_t+A)^{-1}$ on $\L^q(0,T;\X)$.

\medbreak

It is classical, \cite{CoulhonLamberton1986}, that \hyperref[eq:MaxRegLq]{(MR${}_q$)} for an operator $A$ on $\X$ implies that $A$ is sectorial of angle $\omega$ for some $\omega\in[0,\frac{\pi}{2})$, \textit{i.e.} it implies \eqref{eq:SectCondtn}. Unfortunately, when $\X$ is not a Hilbert space, to prove that an operator $A$ satisfies \hyperref[eq:MaxRegLq]{(MR${}_q$)} is in general a very tedious and hard question, although such kind of results and the corresponding kind of proofs to produce tends to be quite well identified by now\cite{bookDenkHieberPruss2003,KunstmannWeis2004,PrussSimonett2016,KunstmannWeis2017,BreitGaudin2025}. And this occurs even for the most standard examples, say $A$ to be the reader's favorite uniformly elliptic operator (such as, \textit{e.g.}, a Laplacian, a Stokes operator, etc.) with smooth enough coefficients, on a quite standard and simple function space such as $\X=\L^p(\Omega)$, $1<p<\infty$. Indeed, to prove the property \hyperref[eq:MaxRegLq]{(MR${}_q$)} often reduces to deep operator theoretic considerations, including geometry of Banach spaces (via stability under unconditional summation, the so called $\mathcal{R}$-boundedness), upon which is built operator-valued Fourier multiplier theory. The latter turned out to be a cornerstone for the characterization of such a property in the most common cases \cite{Weis2001}. For more details on $\L^q$-maximal regularity in \emph{natural frameworks}, we refer to \cite{bookDenkHieberPruss2003,KunstmannWeis2004,Monniaux2009MaxReg,PrussSimonett2016,BechtelBuiKunstmann2024} and the references therein, see also the short review by the author in \cite[Sec.~2]{Gaudin2023}. The connection to functional analytic properties of sectorial operators has been outlined in Appendix~\ref{App:OperatorTheory}, see Theorem~\ref{thm:LqMaxRegUMD}. Furthermore, notice that the tools designed in aforementioned works restricts $q$ to the range $q\in(1,\infty)$.

\medbreak

However, given $(\D(A),A)$ on $\X$ to be sectorial of angle strictly less than $\pi/2$, proving \hyperref[eq:MaxRegLq]{(MR${}_q$)} for $A$ becomes significantly simpler if one allows themselves to replace the ground space $\X$ for the $\L^q$-maximal regularity property by the \emph{real interpolation space}
\begin{align}\label{eq:realInterpDomain}
    (\X,\D(A))_{\theta,q},\qquad 0<\theta<1,
\end{align}
allowing in this case $q\in[1,\infty]$. This result is known as the \textit{Da Prato--Grisvard Theorem} \cite[Thms~4.7~\&~4.15]{DaPratoGrisvard1975}, \cite[Thm.~2.7,~Cor.~2.8]{DanchinTolksdorf2023}, \cite[Cor.~4.6,~Cor~4.8,~Thm.~4.9]{BechtelBuiKunstmann2024} and reads as follows
\begin{theorem}[ \textbf{Da Prato--Grisvard} ]\label{thm:DaPratoGrisvard} Let $r\in[1,\infty]$, $\theta\in(0,1)$ and $q\in(1,\infty)\cup\{r\}$. Consider $(\D(A),A)$ an $\omega$-sectorial operator on a Banach space $\X$, with $\omega\in[0,\frac{\pi}{2})$. Let $T\in(0,\infty)$ and set $\theta_q := 1+\theta - {1}/{q}$.

\medbreak

Then, provided $f\in\L^{q}(0,T; (\X,\D(A))_{\theta,r})$ and $u_0\in (\X,\D(A))_{\theta_q,q}$, the problem \eqref{ACP} admits a unique mild solution $u\in \C^{0}_{ub}([0,T];  (\X,\D(A))_{\theta_q,q})$, such that $$u\in\W^{1,q}(0,T;(\X,\D(A))_{\theta,r})\quad\text{ and }\quad Au \in \L^{q}(0,T;  (\X,\D(A))_{\theta,r}),$$
with the estimates
\begin{align*}
    \lVert u\rVert_{\L^{\infty}(0,T; (\X,\D(A))_{\theta_q,q})}\qquad\qquad &\\\lesssim_{\theta,r,q,\X,A,T} \lVert (u,\partial_t u, Au) &\rVert_{\L^q(0,T; (\X,\D(A))_{\theta,r})} \\&\lesssim_{\theta,r,q,\X,A,T} \lVert f \rVert_{\L^q(0,T; (\X,\D(A))_{\theta,r})}+ \lVert u_0\rVert_{ (\X,\D(A))_{\theta_q,q}}.
\end{align*}
If additionally  $0\in\rho(A)$, \textit{i.e.} if $A$ is invertible, then the implicit continuity constants are independent of $T$, and the result remains valid for $T=\infty$. Furthermore, there exists $c>0$, such that
\begin{align*}
    \lVert e^{c t}u\rVert_{\L^{\infty}(\RR_+; (\X,\D(A))_{\theta_q,q})}\qquad\qquad &\\\lesssim_{\theta,r,q,\X,A} \lVert e^{ct}(u,\partial_t u, Au) &\rVert_{\L^q(\RR_+; (\X,\D(A))_{\theta,r})} \\&\lesssim_{\theta,r,q,\X,A} \lVert f \rVert_{\L^q(\RR_+; (\X,\D(A))_{\theta,r})}+ \lVert u_0\rVert_{ (\X,\D(A))_{\theta_q,q}}.
\end{align*}
\end{theorem}
Therefore, in this framework, to obtain meaningful maximal regularity results, one is ``just due to compute'' real interpolation spaces. 

\begin{remark}
%Three comments:
% \begin{itemize}
%     \item The Da Prato-Grisvard is the \textbf{only result} that allows to consider $\L^1$-maximal regularity: it can not be otherwise. Indeed, conversely, if a closed operator $A$ has \hyperref[eq:MaxRegLq]{(MR${}_1$)} on a Banach space $\X$, then there exists two Banach spaces $\X_{-}$ dans $\X_{+}$ such that $\X\hookrightarrow \X_-+\X_+$,  $A$ admits a consistent sectorial extension of angle strictly less than $\pi/2$ on both, and \emph{necessarily}
%     \begin{align*}
%         \X = (\X_-;\X_+)_{\frac{1}{2},1}.
%     \end{align*}
%     See \cite[Thm.~4.5]{BechtelBuiKunstmann2024} for more details.
%     \item 
We did a slight abuse in the statement of Theorem~\ref{thm:DaPratoGrisvard}, for simplicity, that we have to correct here. When $\theta\geqslant 1/q$, one has $\theta_q \geqslant 1$ which obviously does not belong to the interval $(0,1)$ in order to define properly the real interpolation space \eqref{eq:realInterpDomain}. However, see for instance \cite[Prop.~6.4]{bookLunardiInterpTheory}, one has the following coincidence with equivalence of norms
    \begin{align*}
        (\X,\D(A))_{\eta,q} = (\X,\D(A^2))_{\frac{\eta}{2},q},\quad \eta\in(0,1),
    \end{align*}
    where the right hand-side remains meaningful for all $\eta\in(0,2)$. Thus, when $\theta_q\in[1,2)$ in Theorem~\ref{thm:DaPratoGrisvard}, one has to understand
    \begin{align*}
        (\X,\D(A))_{\theta_q,q} := (\X,\D(A^2))_{\frac{\theta_q}{2},q},\quad \theta\in(0,1),\quad q\in[1,\infty],\, \quad \theta q\geqslant 1.
    \end{align*}
    Furthermore, for any $m\in\NN^\ast$, provided $0\in\rho(A)$, one has the equivalence of norms
    \begin{align}\label{eq:AdaptedBesovSpaces}
        \lVert x\rVert_{(\X,\D(A^m))_{{\theta},q}}\sim_{m,\theta,q} \left(\int_{0}^{+\infty} \lVert t^{m(1-\theta)}A^m e^{-tA}x\rVert_{\X}^q \frac{\d t}{t}\right)^{1/q} ,\quad \theta\in(0,1),
    \end{align}
    with the usual change when $q=\infty$.
%     \item \textbf{A rule of thumb:} We recall that Besov spaces are stable by real interpolation for fixed integrability index $p\in[1,\infty]$. Therefore, for $s\in I$, $I$ an open interval with non empty interior, if $A$ is a reasonable differential operator with consistent resolvent estimates for some $\mu_s<{\pi}/{2}$
% \begin{align*}
%         \lVert \lambda (\lambda\I-A)^{-1} f\rVert_{\B^{s}_{p,\infty}(\Omega)} \leqslant C_{p,s,\mu_s} \lVert f\rVert_{\B^{s}_{p,1}(\Omega)},\qquad \forall f\in\B^{s}_{p,1}(\Omega), \quad\forall \lambda \in\CC\setminus\overline{\Sigma}_{\mu_s}\quad\forall s \in I,
%     \end{align*}
% then this is \textbf{equivalent} to $\L^q(\B^{s}_{p,q}(\Omega))$-maximal regularity for all $q\in[1,\infty]$, all $s\in I$.

% \medbreak

% Abridged and truncated version\footnote{each point should be an interior point of the interpolation scale, etc.}: having an operator which turns out to be sectorial (of angle less than $\pi/2$) on several Besov spaces is equivalent to its maximal regularity  on all the corresponding Besov spaces!
% \end{itemize}
\end{remark}

For more details about the Da Prato--Grisvard Theory and its uses in fluid mechanics see for instance \cite{DanchinTolksdorf2023,DanchinHieberMuchaTolk2020}. For more details about real (or complex) interpolation of function spaces in the scope of PDEs, with possibly prescribed boundary or algebraic (\textit{e.g.}, divergence free) condition(s), see for instance --and this is far from being exhaustive--\cite{BerghLofstrom1976,bookTriebel1978,Guidetti1991Interp,KaltonMayborodaMitrea2007,MitreaMonniaux2008,bookLunardiInterpTheory,DanchinTolksdorf2023,DanchinHieberMuchaTolk2020,HieberKozonoMonniauxSohr2025,BreitGaudin2025} and the references therein.

\subsection{Main strategy and road map of the present work}

The starting idea is standard, up to applying the Leray projection $\PP_\Omega$ onto solenoidal vector fields, we rewrite the system \eqref{eq:BoussinesqSystemIntro} as an abstract Cauchy problem
\begin{equation}\tag{A-BSQ}\label{eq:AbsBoussinesqSystemIntro}
    \left\{\begin{aligned}
        \partial_t \uu + \nu\AA_\mathcal{D} \uu  & =   -  \PP_{\Omega}\div (\uu \otimes \uu) +  \PP_{\Omega}(\theta\mathfrak{e}_3) & \quad & \text{in $(0,T)$}, \\
        \partial_t \theta - \kappa\Delta_{\mathcal{N}} \theta  & = - \div(\theta\mathbf{u}) & \quad & \text{in $(0,T)$}, \\
        (\uu, \theta) (0) & = (\uu_0, \theta_0). &  &
    \end{aligned}\right.
\end{equation}
where  $-\Delta_{\mathcal{N}}$ is the (negative) Neumann Laplacian operator and  $\AA_\mathcal{D}$ is the Stokes operator supplemented with Dirichlet/no-slip boundary conditions.

From this point, one can rewrite the system \eqref{eq:AbsBoussinesqSystemIntro} using Duhamel's formula, we obtain then the following system of integral equations for all $t\in[0,T)$:
\begin{align*}
    \uu(t) &= e^{-\nu t\AA_\mathcal{D}}\uu_0 - \int_{0}^{t} e^{-\nu(t-s)\AA_\mathcal{D}}\PP_{\Omega}[\div(\uu(s)\otimes\uu(s))]\,\d s + \int_{0}^{t} e^{-(t-s)\nu\AA_\mathcal{D}}\PP_{\Omega}[\theta(s)\mathfrak{e}_3]\,\d s,\\
    \theta(t) &= e^{ \kappa t\Delta_\mathcal{N}}\theta_0 - \int_{0}^{t} e^{\kappa(t-s)\Delta_\mathcal{N}}[\div(\theta(s)\uu(s))]\,\d s.
\end{align*}
Again, we may reach one more level of abstraction writing
\begin{align}
    \begin{array}{rl}
         \uu &= \mathbf{v}_0 + \B_{\nu}(\uu,\uu) + \L_{\nu}(\theta),\\
    \theta &= \boldsymbol{\Theta}_0 + \C_{\kappa}(\theta,\uu),
    \end{array}\label{eq:AbstractBoussinesq}
\end{align}
where  we write $(\mathbf{v}_0,\,\boldsymbol{\Theta}_0) := (e^{-\nu(\cdot)\AA_\mathcal{D}}\uu_0,e^{\kappa(\cdot)\Delta_\mathcal{N}}\theta_0)$ with the formal linear and bilinear operators
    \begin{align}\label{eq:BOpBiLin}
    \B_\nu \,:\, (\vv,\ww) \longmapsto \left[t\mapsto -\int_{0}^{t} e^{-\nu(t-s)\AA_\mathcal{D}}\PP_{\Omega}\div (\vv(s)\otimes\ww(s))\,\d s\right],
\end{align}
from the Navier--Stokes fluid part, with its temperature contribution:
 \begin{align}\label{eq:LOpLin}
    \L_\nu \,:\, \Theta \longmapsto \left[t\mapsto \int_{0}^{t} e^{-\nu(t-s)\AA_\mathcal{D}}\PP_{\Omega}(\Theta(s) \mathfrak{e}_3)\,\d s\right],
\end{align}
and then the bilinear operator arising from transported heat by the flow:
\begin{align}\label{eq:COpBiLin}
    \C_{\kappa} \,:\, (\vv,\Theta) \longmapsto \left[t\mapsto -\int_{0}^{t} e^{\kappa(t-s)\Delta_\mathcal{N}}\div(\Theta(s)\vv(s))\,\d s\right].
\end{align}
Before we continue, note that \eqref{eq:AbstractBoussinesq} can be rephrased as bilinear problem: in a way similar to \cite[Intro.,~p.~234]{BrandoleseHe2020BSQ}, plugging the second equation in the first one, it becomes
\begin{align}\label{eq:ModifAbstractBoussinesq}
    \begin{array}{rl}
         \uu &= \mathbf{v}_0 + \L_{\nu}(\boldsymbol{\Theta}_0) + \B_{\nu}(\uu,\uu)  + \L_{\nu}(\C_{\kappa}(\theta,\uu)),\\
    \theta & = \boldsymbol{\Theta}_0 + \C_{\kappa}(\theta,\uu),
    \end{array}
\end{align}
so that for $\mathbf{x}=\prescript{t}{}{(\uu,\theta)}$, $\mathbf{x}_0=\prescript{t}{}{(\mathbf{v}_0 + \L_{\nu}(\boldsymbol{\Theta}_0),\boldsymbol{\Theta}_0)}$, one can write \eqref{eq:AbstractBoussinesq} in its final and pure abstract form
\begin{align*}
    \mathbf{x} = \mathbf{x}_0 + \mathcal{B}(\mathbf{x},\mathbf{x}),
\end{align*}
where the arising appropriate bilinear map from \eqref{eq:ModifAbstractBoussinesq} is
\begin{align}\label{eq:ultimateBilinearmap}
    \mathcal{B}\,:\, \left(\left(\begin{array}{c}
    \vv\\
    \eta
    \end{array}\right), \left(\begin{array}{c}
    \ww\\
    \zeta
    \end{array}\right)\right)\,\longmapsto \frac{1}{2}\left(\begin{array}{c}
         2\B_{\nu}(\vv,\ww)  + \L_{\nu}(\C_{\kappa}(\zeta,\vv))+\L_{\nu}(\C_{\kappa}(\eta,\ww))\\
         \C_{\kappa}(\zeta,\vv)+\C_{\kappa}(\eta,\ww)
    \end{array}\right).
\end{align}

The solutions $(\uu,\theta)$ of the functional equations \eqref{eq:ModifAbstractBoussinesq} are the so called \emph{mild solutions} of the Boussinesq equations \eqref{eq:BoussinesqSystemIntro}.

Thus, in order to construct global-in-time solutions, we aim to apply the following very well-known fixed point result:
\begin{proposition}\label{prop:FixedPointBilin} Let $\X$ be a Banach space. Let  $\mathcal{B}\,:\X\times\X\longrightarrow\X$ be a well-defined continuous bilinear map. For all $x_0\in\X$ such that $4\lVert \mathcal{B}\rVert_{\X^2\rightarrow\X}\lVert x_0\rVert_{\X}<1$, there exists a unique $x\in\X$ such that
\begin{enumerate}
    \item $4\lVert \mathcal{B}\rVert_{\X^2\rightarrow\X}\lVert x\rVert_{\X}<1$;
    \item The functional equation
    \begin{equation*}
        x = x_0  + \mathcal{B}(x,x) \text{,} \quad \text{ in }\X,
    \end{equation*}
    is satisfied.
\end{enumerate}
Additionally, it holds that $\lVert x\rVert_{\X}\leqslant 2 \lVert x_0\rVert_{\X}$. 
\end{proposition}

For local-in-time solutions, we rather want to use the following fixed point result:
\begin{proposition}\label{prop:FixedPoint2Bil+Lin} Let $\X$ be a Banach space. Let  $\mathcal{B}\,:\X\times\X\longrightarrow\X$ be a well-defined continuous bilinear map, and  $\mathcal{L}\,:\,\X\longrightarrow\X$ be a continuous linear map such that $\lVert \mathcal{L}\rVert_{\X\rightarrow\X}<1$. For all $x_0\in\X$ such that $4\lVert \mathcal{B}\rVert_{\X^2\rightarrow\X}\lVert x_0\rVert_{\X}<(1-\lVert \mathcal{L}\rVert_{\X\rightarrow\X})^2$, there exists a unique $x\in\X$ such that
\begin{enumerate}
    \item $2\lVert \mathcal{B}\rVert_{\X^2\rightarrow\X}\lVert x\rVert_{\X}<(1-\lVert \mathcal{L}\rVert_{\X\rightarrow\X})$;
    \item The functional equation
    \begin{equation*}
        x = x_0  + \mathcal{L}(x)+\mathcal{B}(x,x) \text{,} \quad \text{ in }\X,
    \end{equation*}
    is satisfied.
\end{enumerate}
Additionally, it holds that $\lVert x\rVert_{\X}\leqslant 2(1-\lVert \mathcal{L}\rVert_{\X\rightarrow\X})^{-1} \lVert x_0\rVert_{\X}$. 
\end{proposition}

This procedure via semigroup approach will allow us to build the so called \emph{mild solutions}. If one specifies it, such solutions are usually even strong in a suitable sense. In any case, in particular, they will satisfy for reasonable test functions $\boldsymbol{\varphi}$, smooth divergence free vector field vanishing on $\partial\Omega$, and $\eta$ smooth up to the boundary\footnote{Variations have to be considered depending on the function space to ``see the boundary in an appropriate way'', which also change a little bit the weak formulation depending on that. The equation on $\theta$ and the test function $\eta$ are particularly concerned by this remark.},
    \begin{align*}
        \int_{(0,T)\times\Omega} -\uu\cdot \partial_t \boldsymbol{\varphi} +\nu\nabla \uu: \nabla \boldsymbol{\varphi} - (\uu\otimes\uu) : \nabla \boldsymbol{\varphi} \,\d x \,\d t    &= \int_{(0,T)\times\Omega} \theta \boldsymbol{\varphi}_3\, \d x \d t  + \int_{\Omega} \uu_0\cdot\boldsymbol{\varphi}(0) \,\d x\\
        \text{ and }\qquad \int_{(0,T)\times\Omega} -\theta\cdot \partial_t \eta +\kappa \nabla \theta\cdot  \nabla\eta -  \theta \, \uu \cdot \nabla \eta \,\d x \,\d t    &=  \int_{\Omega} \theta_0\,\eta(0)\, \d x.
    \end{align*}
    Such solutions are the usual weak solutions for which we will prove uniqueness in a slightly larger class than the one exhibited in Theorems~\ref{thm:mainThmBSQLp}~and~\ref{thm:mainThmBSQ}, but only when the domain has some H\"{o}lder-type continuous boundary. When the domain has only Lipschitz boundary, one is only able to prove uniqueness within the class of mild solutions.

    \medbreak

The paper is organized as follows:
\begin{itemize}
    \item Section~\ref{Sec:Laplacians} gathers known results and some of their non-trivial refinements for the Neumann Laplacian and Stokes--Dirichlet operator in $3$-dimensional Lipschitz domains, their fractional powers and their behavior in Sobolev and Besov spaces.

    \item Section~\ref{Sec:MaxGeg} contains the derivation of various $\L^q$-maximal regularity results according to the regularity results gathered in  Section~\ref{Sec:Laplacians} for the Stokes and the Neumann Laplacian operators.

    \item Section~\ref{Sec:BiLiN} gather the appropriate boundedness results for the linear and bilinear solution operators \eqref{eq:BOpBiLin}--\eqref{eq:COpBiLin} in Lebesgue-Sobolev and Lebesgue-Besov spaces, built upon Section~\ref{Sec:MaxGeg}, via maximal regularity result such as Theorems~\ref{thm:DaPratoGrisvard}~and~\ref{thm:LqMaxRegUMD}.

    \item Section~\ref{Sec:WPLebesgueSobolev} is then dedicated to the proof for the well-posedness part of Theorems~\ref{thm:mainThmBSQLp}~and~\ref{thm:mainThmBSQ}. Section~\ref{Sec:WPBSQExistence} concerns the existence part. Section~\ref{Sec:WPBSQUniqueness} concerns the uniqueness part.

    \item Section~\ref{Sec:Asymptotics} gives the proof for the asymptotic behavior with exponential decay claimed in both Theorems~\ref{thm:mainThmBSQLp}~and~\ref{thm:mainThmBSQ}.

    \item Appendix~\ref{App:ProductRules} contains (para-)product results in Besov spaces, necessary to complete the analysis performed in Section~\ref{Sec:BiLiN}.

    \item for the reader's convenience, Appendix~\ref{App:OperatorTheory} contains some additional knowledge on operator theory such as \emph{Bounded Imaginary Powers}, $\mathbf{H}^{\infty}$-\emph{functional calculus} and their link with standard theory of $\L^q$-maximal regularity.
\end{itemize}

\subsection{Clarifications on notations and base knowledge for the current work.}\label{sec:IntroNotations}

For $n\in\NN^\ast$, $\B_r(x)\subset\RR^n$ stands for the ball of radius $r>0$, centered in $x\in\RR^n$. 
\medbreak

For $\X$ be a Banach space, $\Omega\subset\RR^n$ an open set, $\Ccinfty(\Omega;\X)$ is the space of smooth  $\X$-valued functions compactly supported in $\Omega$. The space $\mathcal{D}'(\Omega;\X):=\mathcal{L}(\Ccinfty(\Omega);\X)$ stands for the space of $\X$-valued distributions on $\Omega$, with natural duality bracket $\langle\cdot,\cdot\rangle_{\Omega}$. We can introduce similarly $\mathcal{S}(\RR^n;\X)$ and $\mathcal{S}'(\RR^n;\X)$, the class of Schwartz functions and tempered distributions. For any $p\in[1,\infty]$, $k\in\NN$, $\W^{k,p}(\RR^n;\X)$ stands for the Bochner-Sobolev space with norm $\lVert u\rVert_{\W^{k,p}(\RR^{n};\X)} := \sum_{m=0}^k \lVert \nabla^m u\rVert_{\L^p(\RR^n;\X)}.$ For all $s\in\RR$, we set the Sobolev-Bessel potential space over $\L^p$ to be  $\H^{s,p}(\RR^n;\X):=(\I-\Delta)^{-s/2}\L^p(\RR^{n};\X)$, with the Bessel potential  $(\I-\Delta)^\frac{s}{2}:=\mathcal{F}^{-1}(1+|\xi|^2)^\frac{s}{2}\mathcal{F}$ and $\mathcal{F}$ is the Fourier transform. By definition, one has  $\L^p(\RR^n;\X) = \W^{0,p}(\RR^n;\X) = \H^{0,p}(\RR^n;\X)$, with equalities of norms.

Considering $\varphi\in\Ccinfty(\B_{4/3}(0),\RR)$, such that $\varphi_{|_{\B_{3/4}}(0)}=1$,  we set $\phi_{-1}=\varphi$, for all $j\geqslant0$, $\phi_{j}:=\varphi(2^{-(j+1)}\cdot) - \varphi(2^{-j}\cdot)$ and $\Delta_{j}:=\mathcal{F}^{-1}\phi_j\mathcal{F}$,  so that $\I = \sum_{j\geqslant -1} \Delta_j$ on $\mathcal{S}'(\RR^n;\X)$. For $s\in\RR$, $1\leqslant p,q\leqslant\infty$, we define the $\X$-valued Besov space $\B^{s}_{p,q}(\RR^n;\X)$ as the vector space made of elements $u\in\mathcal{S}'(\RR^n;\X)$ such that
\begin{align*}
    \lVert u\rVert_{\B^{s}_{p,q}(\RR^{n};\X)}:= \big(\sum_{j\geqslant-1} \lVert 2^{js}\Delta_j u\rVert_{\L^p(\RR^n;\X)}^q\big)^{1/q}<\infty
\end{align*}
with the standard change when $q=\infty$. One also sets:
\begin{align*}
    \mathcal{B}^{s}_{p,\infty}(\RR^n;\X) := \overline{\B^{s}_{p,\infty}\cap\C^\infty_{ub}(\RR^n;\X)}^{\lVert\cdot\rVert_{\B^{s}_{p,\infty}(\RR^{n};\X)}}.
\end{align*}
For all $s_0<s_1$, writing $s_0=k_0$, $s_1=k_1$ if $s_0,s_1\in\NN$, for all $p,q,q_0,q_1\in[1,\infty]$, all $\theta\in(0,1)$, it holds
\begin{align}
    \big(\W^{k_0,p}(\RR^{n};\X),\W^{k_1,p}(\RR^{n};\X)\big)_{\theta,q} = &\big(\H^{s_0,p}(\RR^{n};\X),\H^{s_1,p}(\RR^{n};\X)\big)_{\theta,q}\nonumber \\
    =&\big(\B^{s_0}_{p,q_0}(\RR^{n};\X),\B^{s_1}_{p,q_1}(\RR^{n};\X)\big)_{\theta,q}=\B^{(1-\theta)s_0+\theta s_1}_{p,q}(\RR^{n};\X)\label{eq:RealInterpbesovbasicIntro}
\end{align}
with equivalence of norms, where $(\cdot,\cdot)_{\theta,q}$ is the real interpolation method. A similar result holds for the spaces $\mathcal{B}^{\bullet}_{p,\infty}$. If $\Omega\subset\RR^n$ is a Lipschitz domain, $\R_\Omega$ is the restriction operator from $\mathcal{D}'(\RR^n;\X)$ to  $\mathcal{D}'(\Omega;\X)$, in which case we define
\begin{align*}
    \W^{k,p}(\Omega;\X):= \R_{\Omega}&[\W^{k,p}(\RR^{n};\X)],\quad\B^{s}_{p,q}(\Omega;\X):= \R_{\Omega}[\B^{s}_{p,q}(\RR^{n};\X)],\\ &\,\text{ and }\,\H^{s,p}(\Omega;\X):= \R_{\Omega}[\H^{s,p}(\RR^{n};\X)],
\end{align*}
with their respective induced norms, as well as
\begin{align*}
    \W^{k,p}_0(\Omega;\X):=\{u\in \W^{k,p}(\RR^{n};\X)\,|\,\supp u \subset\overline{\Omega}\}
\end{align*}
and similarly with $\B^{s}_{p,q,0}(\Omega;\X)$ and $\H^{s,p}_0(\Omega;\X)$, and $\Ccinfty(\Omega,\X)$ is, for each of them, strongly dense for all $s\in\RR$ as soon as  $1\leqslant p,q<\infty$, $p=1$ being excluded for Bessel potential spaces. Since one has an appropriate extension operator, thanks to the Lipschitz boundary of the domain, the real interpolation result \eqref{eq:RealInterpbesovbasicIntro} remains valid these variations of function spaces in domains.

\medbreak

When $\X$ is finite dimensional, we just write without any distinction the norms as $\lVert \cdot\rVert_{\W^{k,p}(\Omega)}$,  $\lVert \cdot\rVert_{\H^{s,p}(\Omega)}$ and $\lVert \cdot\rVert_{\B^{s}_{p,q}(\Omega)}$, and the omission means that the considered space is \emph{necessarily} finite dimensional-valued. When the function space is directly concerned, not just its norm, we mean $\W^{k,p}(\Omega,\CC)=:\W^{k,p}(\Omega)$ and similarly for other function spaces. Here, $(\Y)'$ denotes the (anti-)dual of the Banach space $\Y$. For all $1<p<\infty$, $k\in\NN$, one has $\H^{k,p}(\Omega)=\W^{k,p}(\Omega)$ with equivalence of norms. For all $s\in\RR$,  $p,q\in(1,\infty]$, $p<\infty$, one has $\B^{s}_{p,q}(\Omega) = (\B^{-s}_{p',q',0}(\Omega))'$, $\B^{s}_{p,q,0}(\Omega) = (\B^{-s}_{p',q'}(\Omega))'$, $\B^{s}_{p,1}(\Omega) =( \mathcal{B}^{-s}_{p',\infty,0}(\Omega))'$, $\B^{s}_{p,1,0}(\Omega) =( \mathcal{B}^{-s}_{p',\infty}(\Omega))'$, $\H^{s,p}(\Omega)= (\H^{-s,p'}_0(\Omega))'$,  and $\H^{s,p}_0(\Omega)=(\H^{-s,p'}(\Omega))'$.  When $-1+1/p<s<1/p$, one has the canonical identifications $\H^{s,p}(\Omega)= \H^{s,p}_0(\Omega)$ and $\B^{s}_{p,q}(\Omega)=\B^{s}_{p,q,0}(\Omega)$, obtained through the canonical extension by $0$. When $1/p<s<1+1/p$, one has the bounded trace operator $\H^{s,p}(\Omega)\longrightarrow\B^{s-1/p}_{p,p}(\partial\Omega)$ and $\B^{s}_{p,q}(\Omega)\longrightarrow\B^{s-1/p}_{p,q}(\partial\Omega)$, with the corresponding null-spaces respectively  identified as $\H^{s,p}_0(\Omega)$ and $\B^{s}_{p,q,0}(\Omega)$ through the extension to whole space by $0$. Additionally, for all $1<p_0,p_1<\infty$, $s_0,s_1\in\RR$, $0<\theta<1$, for $s_\theta := (1-\theta)s_0+\theta s_1$, $1/p_\theta:=(1-\theta)/p_0+\theta/p_1$, one has the following:
\begin{align*}
    [\H^{s_0,p_0}_0(\Omega),\H^{s_1,p_1}_0(\Omega)]_{\theta} = \H^{s_\theta,p_\theta}_0(\Omega),\,\text{ and }\,[\H^{s_0,p_0}(\Omega),\H^{s_1,p_1}(\Omega)]_{\theta} = \H^{s_\theta,p_\theta}(\Omega),
\end{align*}
where  $[\cdot,\cdot]_\theta$ is the complex interpolation method.

When  $\Omega\subset\RR^n$ is a bounded Lipschitz domain, the characteristic function of $\Omega$ satisfies $\mathds{1}_{\Omega}\in\H^{s,p}_0(\Omega),\B^{s}_{p,q,0}(\Omega)$, for all $s<1/p$. Therefore, for all $u\in \H^{s,p}(\Omega),\in\B^{s}_{p,q}(\Omega)$, $s>-1+1/p$, one can make sense  of $\int_{\Omega} u := \langle u,\,\mathds{1}_{\Omega} \rangle_{\Omega}$. In this case, we define the (orthogonal) projection onto mean-free distributions as $\P_\circ:=\I-|\Omega|^{-1}  \langle \cdot ,\,\mathds{1}_{\Omega} \rangle_{\Omega}\mathbf{1}$, and the mean-free function spaces be $\H^{s,p}_\circ(\Omega) := \P_\circ\H^{s,p}(\Omega)$ and $\B^{s}_{p,q,\circ}(\Omega):=\P_\circ \B^{s}_{p,q}(\Omega)$, $s>-1+1/p$. When $s< 1/p$, one has  $\mathbf{1}\in\H^{-s,p'}(\Omega),\,\B^{-s}_{p',q'}(\Omega)$, identified with $\mathds{1}_\Omega$, so that one can extend the scale of mean-free function spaces to arbitrary negative regularity by setting  $\H^{s,p}_\circ(\Omega) := \P_\circ\H^{s,p}_0(\Omega)$ and $\B^{s}_{p,q,\circ}(\Omega):=\P_\circ \B^{s}_{p,q,0}(\Omega)$. The definitions are consistent on the overlap $-1+1/p<s<1/p$. Both families are again interpolation scales.

Finally, for an unbounded linear operator $A$ that can be realized simultaneously on different function spaces, we write $\D_p(A)=\D_p^0(A)$ its domain in $\L^p$, $\D_{p}^{s}(A)$ its domain in $\H^{s,p}$, $\D^{s}_{p,q}(A)$ its domain in $\B^{s}_{p,q}$.

When $(\D(A),A)$ is an invertible sectorial operator on Banach space $\X$, \textit{i.e.} satisfying \eqref{eq:SectCondtn}, following \cite[Chap.~6,~Sec~6.3]{bookHaase2006}, one can define the \emph{universal vector space adapted to A}
\begin{align*}
    \mathcal{U}_{A}:=\bigcup_{n\in\NN} (\I+A)^{2n}A^{-n}\X
\end{align*}
on which $A$ extends naturally and consistently. For $\alpha\in\RR$, this allows to define the abstract Banach spaces
\begin{align*}
    \{\, x\in\mathcal{U}_A\,|\, A^{\alpha} x \in\X\,\},\quad 
    \text{ with norm }\quad x\longmapsto\lVert A^{\alpha} x\rVert_{\X}.
\end{align*}
The latter can be canonically identified as $\D(A^\alpha)$ when $\alpha \geqslant0$, due to invertibility of $A$. When $\alpha<0$, we denote the corresponding space by  $A^{-\alpha}\X$. In particular, for all $0\leqslant \alpha \leqslant 1$, $A$ acts has an isomorphism from $\D(A^{1-\alpha})$ to $A^{\alpha}\X$. If additionally $\X=\H$ is a Hilbert space, then the norms of $A^{\alpha}\H$, $\alpha\in\mathbb{R}$, can be chosen so that the latter is still a Hilbert space and $A^{\alpha}\,:\,\H\longrightarrow A^{\alpha}\H$ to be a bijective isometry.

%----------------------------------------------------
%--------------- Functional Setting -----------------
%----------------------------------------------------
\section{Functional Setting: Linear theory and maximal regularity estimates}

\subsection{Linear theory for the  Neumann Laplacian and the Stokes operator.}\label{Sec:Laplacians}

\paragraph{The Neumann Laplacian}

We start dealing with refined properties of the Neumann Laplacian in bounded Lipschitz domains.

\begin{definition}For $\Omega\subset \RR^n$, to be a Lipschitz domain, the (negative) Neumann Laplacian is the unbounded $\L^2$-realization of the bounded operator $-\Delta\,:\,\H^{1,2}(\Omega)\longrightarrow (\H^{1,2}(\Omega))'= \H^{-1,2}_{0}(\Omega)$, denoted by $(\D_2(-\Delta_\mathcal{N}),-\Delta_\mathcal{N})$, and induced by the sesquilinear form
\begin{align*}
    \mathfrak{a}_{\mathcal{N}}\,:\,\D(\mathfrak{a}_{\mathcal{N}})\times \D(\mathfrak{a}_{\mathcal{N}}) &\longrightarrow \CC\\
    (u ,v)\quad\quad&\longmapsto \langle \nabla u , \nabla  v\rangle_{\Omega} = \int_{\Omega} \nabla u (x) \cdot \overline{\nabla v(x)}\,\d x,
\end{align*}
with form domain $\D(\mathfrak{a}_{\mathcal{N}}):=\H^{1,2}(\Omega)$.
\end{definition}
In particular, the Hilbert space setting $(\H^{1,2}(\Omega),\L^2(\Omega), (\H^{1,2}(\Omega))')$ allows to solve uniquely (up to a constant) the Laplace equation supplemented with a Neumann boundary condition : for any $f \in\L^2(\Omega)$, there exists a $u\in\H^{1,2}(\Omega)$ unique up to a constant, such that
\begin{equation*}
    \left\{ \begin{array}{rllr}
         - \Delta u  &= f \text{, }&&\text{ in } \Omega\text{,}\\
        \partial_\mathbf{n} u_{|_{\partial\Omega}} &=0\text{, } &&\text{ on } \partial\Omega\text{.}
    \end{array}
    \right.
\end{equation*}
Since the Neumann Laplacian defined this way is given by a reasonable symmetric sesquilinear form, it is very well-known, \cite[Chap.~1--6]{bookOuhabaz2005}, we can define the operator in $\L^p$ and we can even go a bit further.

\begin{proposition}\label{prop:NeumannHinfty} Let $\Omega$ be a bounded Lipschitz domain of $\RR^n$. For all $p\in(1,\infty)$, the (negative) Neumann Laplacian $-\Delta_\mathcal{N}$ is sectorial of angle $0$ in $\L^p(\Omega)$. Furthermore, 
its restriction to $\L^p_{\circ}(\Omega)$ is invertible, remains sectorial of angle $0$ and has bounded $\mathbf{H}^\infty(\Sigma_\theta)$-functional calculus for all $\theta\in(0,\pi)$.
\end{proposition}

Although a part of this result is very well-known, a proper transition from $\L^p$-theory to average-free $\L^p_\circ$-theory from the  operator theoretic point of view seems to be missing in the literature. Also, no known reference seems to address the boundedness of the $\mathbf{H}^{\infty}$-functional calculus of the Neumann Laplacian itself --and not a shift by a positive multiple of the identity-- in the $\L^p$ setting. Hence, we give a detailed scheme of proof here, for the reader's convenience.

\begin{proof} 
\textbf{Step 1:} We explain the very well known fact: why $-\Delta_\mathcal{N}$  is sectorial of angle $0$ in $\L^2(\Omega)$ with bounded holomorphic functional calculus, and that the semigroup extrapolates to $\L^p(\Omega)$, $1\leqslant p\leqslant \infty$, holomorphic when $1<p<\infty$.

Briefly, the $\L^2$-theory and the $\L^p$-extrapolation of the semigroup follows from the standard theory of sesquilinear form, see, \textit{e.g.}, \cite[Prop.~1.51,~Thm.~1.52,~Thm.~3.13]{bookOuhabaz2005}, applied to the real symmetric and coercive sesquilinear form $(\D(\mathfrak{a}_{\mathcal{N}}),\mathfrak{a}_{\mathcal{N}})$. This yields two results:
\begin{itemize}
    \item Combined with self-adjointness, we obtain a uniformly bounded holomorphic semigroup $(e^{z\Delta_\mathcal{N}})_{z\in\CC_+}$ in $\L^2(\Omega)$, with the estimate
    \begin{align*}
        \lVert e^{z\Delta_\mathcal{N}} f\rVert_{\L^2(\Omega)}\leqslant \lVert f\rVert_{\L^2(\Omega)},\quad \forall f\in\L^2(\Omega),\quad \forall z\in\CC_+; 
    \end{align*}
    Furthermore, as a self-adjoint operator of angle $0$, it enjoys the boundedness of its $\mathbf{H}^\infty(\Sigma_\theta)$-calculus, for all $\theta\in(0,\pi)$, see \cite{McIntosh1986}.
    \item And the semigroup $(e^{t\Delta_{\mathcal{N}}})_{t\geqslant 0}$ extends as such in $\L^p(\Omega)$ for all $p\in[1,\infty]$, and one has uniform bound
    \begin{align*}
        \lVert e^{t\Delta_\mathcal{N}} f\rVert_{\L^p(\Omega)}\leqslant \lVert f\rVert_{\L^p(\Omega)},\quad \forall f\in\L^p(\Omega),\quad \forall t\geqslant 0.
    \end{align*}
    Moreover, for all $1<p<\infty$, given  $\omega_p:= \arctan\big(\frac{1}{2}\sqrt{|p'-2||p-2|}\big)<\pi/2$, the semigroup is even holomorphic in $\Sigma_{\pi/2-\omega_p}$, with the same bound. Additionally, the semigroup is strongly continuous in $\L^p$ up to time $0$ for all $1\leqslant p<\infty$. 
\end{itemize}
Additionally, one can reproduce the analysis for $\varepsilon \langle\cdot,\cdot\rangle_\Omega+\mathfrak{a}_\mathcal{N}$, so that one obtains that $\varepsilon\I-\Delta_\mathcal{N}$ is an invertible $0$-sectorial operator in $\L^2(\Omega)$, with bound
\begin{align*}
        \lVert e^{-z(\varepsilon\I-\Delta_\mathcal{N})} f\rVert_{\L^p(\Omega)}\leqslant e^{-\Re(z)\varepsilon} \lVert f\rVert_{\L^p(\Omega)},\quad  \forall f\in\L^p(\Omega),\quad \forall z\in\Sigma_{\pi/2-\omega_p},\quad 1<p<\infty.
\end{align*}
When $p=1,\infty$, the semigroup $(e^{-t(\varepsilon\I-\Delta_\mathcal{N})})_{t\geqslant 0}$ has the bound $e^{-t\varepsilon}\leqslant 1$.    
    
Thus, \eqref{eq:SectCondtn} holds for $\varepsilon\I-\Delta_\mathcal{N}$, in $\L^p$ :  for $\omega_p= \arctan\big(\frac{1}{2}\sqrt{|p'-2||p-2|}\big)<\pi/2$, for all $\varepsilon\geqslant 0$, all $\mu\in(0,\pi-\omega_p)$, one obtains the following resolvent estimate for all $f\in\L^p(\Omega)$, for all $\lambda\in\CC$ such that $|\arg(\lambda)|<\mu<\omega_p$, 
\begin{align}\label{eq:ResolEstShiftNeumannLp}
    (\varepsilon+|\lambda|)\lVert (\lambda\I + \varepsilon\I-\Delta_\mathcal{N})^{-1}f\rVert_{\L^p(\Omega)} \lesssim_{\mu,p}\lVert f\rVert_{\L^p(\Omega)}.
\end{align}

\textbf{Step 2:} The semigroup in the average free $\L^p$ spaces, $\L^p_\circ(\Omega)$: we derive basic properties for the Neumann Laplacian acting on mean-free functions, and prove the semigroup is well-defined and exponentially stable. This will imply some useful spectral properties that will be of use to extrapolate $\mathbf{H}^\infty$-functional calculus of angle $0$ in Step 3.

It turns out that the Neumann Laplacian in $\L^2(\Omega)$ has exactly null-space $\N_2(\Delta_\mathcal{N})= \CC$. Since, we have a canonical orthogonal decomposition  $\L^2(\Omega) = \L^2_\circ(\Omega)\oplus \CC$ given by the orthogonal projection $\P_\circ =\mathbf{I}-|\Omega|^{-1}\langle \mathbf{1},\, \cdot\rangle_{\Omega}\mathbf{1}$, we check that $(\D_{2}(\Delta_\mathcal{N}),-\Delta_{\mathcal{N}})$  restricts naturally as an invertible $0$ sectorial operator in $\L^2_\circ(\Omega)$ with form domain $\H^{1,2}_\circ(\Omega)$, and that the semigroup decays exponentially when restricted to $\L^p_{\circ}(\Omega)$, for all $1<p<\infty$.

\textbf{Step 2.1:} We aim to show that for all $f\in\L^2(\Omega)$, all $\lambda\in\Sigma_{\pi}$, one has $\P_\circ(\lambda\I-\Delta_\mathcal{N})^{-1}f \in\D_{2}(\Delta_\mathcal{N})$ and that
\begin{align}\label{eqref:NeumannResolvDivFree}
    \P_\circ(\lambda\I-\Delta_\mathcal{N})^{-1}f=(\lambda\I-\Delta_\mathcal{N})^{-1}\P_\circ f.
\end{align}

\medbreak

Fix $\lambda\in\Sigma_\pi$ and $f\in\L^2(\Omega)$, there exists a unique $u\in\H^{1}(\Omega)$ such that
\begin{align*}
    \lambda \langle u,\,v \rangle_{\Omega} + \mathfrak{a}_\mathcal{N}(u,v) = \langle  f,\,v \rangle_{\Omega}.
\end{align*}
Note that $\P_\circ$ maps $\L^2$ to $\L^2_\circ$, is orthogonal, hence self-adjoint, and that $\P_\circ$ maps $\H^{1,2}(\Omega)$ to $\H^{1,2}_\circ(\Omega)$, and that $\nabla \P_\circ = \nabla$ in $\L^2(\Omega)$. Therefore, for all $v\in\H^{1,2}(\Omega)$,
\begin{align*}
    \lambda \langle \P_\circ u,\,v \rangle_{\Omega} + \mathfrak{a}_\mathcal{N}( \P_\circ u,v) &= \lambda \langle \P_\circ u,\,v \rangle_{\Omega} + \mathfrak{a}_\mathcal{N}(  u,v)\\
    &= \lambda \langle u,\, \P_\circ v \rangle_{\Omega} + \mathfrak{a}_\mathcal{N}(  u,\P_\circ v)\\
    &= \langle  f,\,\P_\circ v \rangle_{\Omega}\\
    &= \langle  \P_\circ  f,\,v \rangle_{\Omega}.
\end{align*}
Since $\P_\circ f \in \L^2(\Omega)$ and $\P_\circ u\in\H^{1,2}(\Omega)$, by uniqueness, we deduce $\P_\circ u=(\lambda\I-\Delta_\mathcal{N})^{-1}\P_\circ f \in\D_2(\Delta_\mathcal{N})$, which means that we also have obtained \eqref{eqref:NeumannResolvDivFree} .

\textbf{Step 2.2:} The semigroup in $\L^p_\circ$, $1\leqslant p\leqslant \infty$ is uniformly bounded and exponentially stable when $1<p<\infty$, implying that the Neumann Laplacian is an invertible sectorial operator in $\L^p_\circ(\Omega)$.

By the Poincaré-Wirtinger inequality, the sesquilinear form $\mathfrak{a}_\mathcal{N}$ restricted with form domain $\P_\circ \D(\mathfrak{a}_\mathcal{N})=\H^{1,2}_\circ(\Omega)$  with ground space $\L^2_\circ$, satisfies
\begin{align*}
    \mathfrak{a}_\mathcal{N}(u,u) = \lVert \nabla u\rVert^2_{\L^2(\Omega)} \geqslant \frac{1}{C_\Omega^2}\lVert u \rVert_{\L^2(\Omega)}^2.
\end{align*}
Therefore, the sesquilinear form being now strictly coercive, there exists a unique self-adjoint invertible $0$-sectorial operator $(\D_2(\Delta_\mathcal{N}^\circ),-\Delta_\mathcal{N}^{\circ})$ in $\L^{2}_{\circ}(\Omega)$ such that it realizes the sesquilinear form $(\P_\circ \D(\mathfrak{a}_\mathcal{N}),\mathfrak{a}_\mathcal{N})$. By Step 2.1, for all $f\in\L^2(\Omega)$, for all $\lambda\in\Sigma_\pi$, it holds
\begin{align}\label{eq:ResolvIdentityMeanFreeNeumann}
    (\lambda\I-\Delta_\mathcal{N}^{\circ})^{-1}\P_\circ f = (\lambda\I-\Delta_\mathcal{N})^{-1}\P_\circ f.
\end{align}
By the Cauchy Integral formula, it transfers as a corresponding identity for the respective semigroups for all $f\in\L^2(\Omega)$, all $z\in\Sigma_{\pi/2}$, it holds
\begin{align}\label{eq:SemigroupIdentityMeanFreeNeumann}
    e^{z\Delta_\mathcal{N}^\circ}\P_\circ f = e^{z\Delta_\mathcal{N}}\P_\circ f,\quad\text{ in }\L^2(\Omega).
\end{align}

Since the resolvent is compact in $\L^2$, there exists $(\lambda_k)_{k\in\NN^\ast}\subset(0,\infty)$, such that
\begin{align*}
    \sigma(-\Delta_\mathcal{N}^\circ)=\{ \lambda_k,\, k\in\NN^\ast\},\quad\text{ and }\quad\sigma(-\Delta_\mathcal{N})=\sigma(-\Delta_\mathcal{N}^\circ)\cup\{0\}.
\end{align*}
By \eqref{eq:SemigroupIdentityMeanFreeNeumann}, since $-\Delta_{\mathcal{N}}^\circ$ is invertible and positive in $\L^2_\circ(\Omega)$, one deduces
\begin{align*}
    \lVert e^{z\Delta_\mathcal{N}}\P_\circ f\rVert_{\L^2(\Omega)}=\lVert e^{z\Delta_\mathcal{N}^\circ}\P_\circ f\rVert_{\L^2(\Omega)}\leqslant e^{-\Re(z)\lambda_1}\lVert \P_\circ f\rVert_{\L^2(\Omega)}\leqslant e^{-\Re(z)\lambda_1}\lVert f\rVert_{\L^2(\Omega)}.
\end{align*}
And one also has in $\L^p(\Omega)$, 
\begin{align*}
    \lVert e^{z\Delta_\mathcal{N}}\P_\circ f\rVert_{\L^p(\Omega)}=\lVert e^{z\Delta_\mathcal{N}^\circ}\P_\circ f\rVert_{\L^p(\Omega)}\leqslant\lVert \P_\circ f\rVert_{\L^p(\Omega)}\leqslant 2\lVert f\rVert_{\L^p(\Omega)}.
\end{align*}
So that $(e^{t\Delta_{\mathcal{N}}})_{t\geqslant 0}$ restricts to $\L^p_{\circ}(\Omega)$ as the uniformly bounded semigroup $(e^{t\Delta_{\mathcal{N}}^\circ})_{t\geqslant 0}$ , for all $p\in[1,\infty]$, and is holomorphic in $\Sigma_{\pi/2-\omega_p}$ for all $1<p<\infty$. By complex interpolation, one obtains for all $1<p<\infty$, $z\in\Sigma_{\pi/2-\omega_p}$, and all $f\in\L^p_\circ(\Omega)$
\begin{align*}
\lVert e^{z\Delta_\mathcal{N}}f\rVert_{\L^p(\Omega)}\leqslant 2 e^{-2\Re(z)\lambda_1/\max(p,p')}\lVert f\rVert_{\L^p(\Omega)}.
\end{align*}
Therefore, for all $1<p<\infty$, $-\Delta_\mathcal{N}-2\lambda_1/\max(p,p')$  is sectorial of angle $\omega_p$, and in particular $\Delta_\mathcal{N}$ is an invertible sectorial operator in $\L^p_\circ(\Omega)$, \textit{i.e.} there exists $r_p>0$ such that
\begin{align}\label{eq:InvertibleneumannMeanFreeResolv}
    \B_{r_p}(0)\subset \rho_p^\circ(\Delta_{\mathcal{N}}).
\end{align}

Additionally, we derive  from \eqref{eq:ResolvIdentityMeanFreeNeumann} and the Cauchy integral formula, the following functional calculus identity valid for all $\theta\in(0,\pi)$, for all $\varphi\in\mathbf{H}^\infty(\Sigma_\theta)$ and all $\varepsilon\geqslant 0$: 
\begin{align}\label{eq:FuncCalcIdentityMeanFreeNeumann}
    \varphi(\varepsilon\I-\Delta_\mathcal{N}^\circ)\P_\circ f = \varphi(\varepsilon\I-\Delta_\mathcal{N})\P_\circ f, \quad \text{ in }\L^2(\Omega).
\end{align}
    
\textbf{Step 3:} Now, we prove that both: the Neumann Laplacian is sectorial of angle $0$ in $\L^p(\Omega)$ and it has $\mathbf{H}^{\infty}$-functional calculus of the same angle.

\textbf{Step 3.1:} For any $\varepsilon\in(0,1)$, we prove first the $\mathbf{H}^\infty$-functional calculus of angle $0$ for $\varepsilon\I-\Delta_\mathcal{N}$ in $\L^p(\Omega)$ for all $1<p<\infty$. The goal is to apply \cite[Theorem~2.2]{BlunckKuntsmann2003}.

First, since the following Nash inequality is valid
\begin{align*}
    \lVert v\rVert_{\L^2(\Omega)} \lesssim_{\Omega,n} \Big(\frac{1}{\varepsilon}\Big)^\frac{2n}{n+2} \lVert v\rVert_{\L^1(\Omega)}^{1-\frac{n}{n+2}}\Big( \varepsilon\lVert v\rVert_{\L^2(\Omega)}^2+\lVert \nabla v\rVert_{\L^2(\Omega)}^2\Big)^{\frac{1}{2} \frac{n}{n+2}},\quad \forall v\in\H^{1,2}(\Omega),
\end{align*}
writing $\mathfrak{a}_{\mathcal{N}}^\varepsilon = \varepsilon \langle \cdot,\cdot\rangle_{\Omega} + \mathfrak{a}_{\mathcal{N}}$, we tautologically have
\begin{align*}
    \lVert u\rVert_{\L^2(\Omega)} \lesssim_{\Omega,n} \Big(\frac{1}{\varepsilon}\Big)^\frac{2n}{n+2} \lVert u\rVert_{\L^1(\Omega)}^{1-\frac{n}{n+2}}\Big( \mathfrak{a}_{\mathcal{N}}^\varepsilon(u,u)\Big)^{\frac{1}{2} \frac{n}{n+2}},\quad \forall u\in\H^{1,2}(\Omega)=\D(\mathfrak{a}_\mathcal{N}^\varepsilon).
\end{align*}
Thus, by \cite[Thm~6.3]{bookOuhabaz2005}, self-adjointness, duality and complex interpolation, we deduce the bound
\begin{align}\label{eq:LpLqDecay}
    \lVert e^{-t(\varepsilon\I-\Delta_{\mathcal{N}})} u\rVert_{\L^q(\Omega)} \lesssim_{p,q,n,\Omega} \frac{1}{\varepsilon^{2n(\frac{1}{p}-\frac{1}{q})}}\frac{1}{t^{\frac{n}{2}(\frac{1}{p}-\frac{1}{q})}}\lVert  u\rVert_{\L^p(\Omega)}, 
\end{align}
which holds for all $u\in\L^p(\Omega)$ and all $t>0$, provided $1\leqslant p\leqslant q\leqslant \infty$.
Now, notice that as in the proof of \cite[Theorem~8]{Davies1995}, one can deduce there  for all $1$-Lipschitz real valued function $\psi$, there exists $c>0$ such that for all $u\in\L^2(\Omega)$, all $t>0$, all $\varrho>0$,
\begin{align}\label{eq:L2offdiag}
    \lVert e^{\varrho\psi} e^{-t(\varepsilon\I-\Delta_{\mathcal{N}})}[e^{-\varrho\psi}u]\rVert_{\L^2(\Omega)} \lesssim_{n,\Omega} e^{-c \varrho^2 t}\lVert  u\rVert_{\L^2(\Omega)}.
\end{align}

With \eqref{eq:LpLqDecay} and \eqref{eq:L2offdiag} at hand, we can apply \cite[Theorem~2.2]{BlunckKuntsmann2003}, and deduce that for all $\varepsilon\in(0,1)$, $\varepsilon\I-\Delta_{\mathcal{N}}$ is sectorial of angle $0$ in $\L^p(\Omega)$ for all $p\in(1,\infty)$, which means that \eqref{eq:ResolEstShiftNeumannLp} remains valid for all $\lambda\in\Sigma_\mu\cup\{0\}$, for all $\mu\in(0,\pi)$, but more importantly it admits bounded $\mathbf{H}^{\infty}$-functional calculus of angle $0$.

\textbf{Step 3.2:} We prove that $-\Delta_\mathcal{N}$ is sectorial of angle $0$ in $\L^p(\Omega)$, for all $1<p<\infty$.

Since  $\varepsilon\I-\Delta_{\mathcal{N}}$ is $0$ sectorial in $\L^p(\Omega)$ for all $\varepsilon\in(0,1)$, choosing $\varepsilon>0$ small enough, it suffices to show that, for $\mu\in(0,\pi)$, for some $r>0$, for all $\lambda\in\Sigma_\mu\cap\B_{r}(0)$, one has
\begin{align*}
    |\lambda|\lVert (\lambda\I-\Delta_{\mathcal{N}})^{-1}f\rVert_{\L^p(\Omega)} \lesssim_{\mu,p,r} \lVert f\rVert_{\L^p(\Omega)}.
\end{align*}

We choose $r_p>0$ (note that $r_p \leqslant r_2$) from \eqref{eq:InvertibleneumannMeanFreeResolv}, and we consider $f\in\L^2\cap\L^p(\Omega)$, for all  $\lambda\in\Sigma_\mu\cap \B_{r_p}(0)$, one has in $\L^2(\Omega)$, hence almost everywhere,
\begin{align*}
    (\lambda \I-\Delta_{\mathcal{N}})^{-1}f = (\lambda \I-\Delta_{\mathcal{N}})^{-1}\P_\circ f + \frac{1}{\lambda}(f)_\Omega
\end{align*}
where $(f)_\Omega = 1/|\Omega|\int_{\Omega}f$. By the end of Step 2.2, \eqref{eq:InvertibleneumannMeanFreeResolv}, one has
\begin{align*}
    (1+|\lambda|)\lVert (\lambda \I-\Delta_{\mathcal{N}})^{-1}\P_{\circ}f  \rVert_{\L^p(\Omega)}\lesssim_{p,r_p,\mu} \lVert f\rVert_{\L^p(\Omega)},
\end{align*}
and then
\begin{align*}
    |\lambda|\lVert(\lambda \I-\Delta_{\mathcal{N}})^{-1}f \rVert_{\L^p(\Omega)} &\leqslant  |\lambda|\lVert (\lambda \I-\Delta_{\mathcal{N}})^{-1}\P_{\circ}f  \rVert_{\L^p(\Omega)} + |\lambda|\Big\lVert \frac{1}{|\lambda|} (f)_\Omega \Big\rVert_{\L^p(\Omega)}\\
    &\lesssim_{p,r_p,\mu} \lVert \P_{\circ}f  \rVert_{\L^p(\Omega)} + \lVert  (f)_\Omega \rVert_{\L^p(\Omega)}\\
    &\lesssim_{p,r_p,\mu} \lVert f  \rVert_{\L^p(\Omega)} .
\end{align*}
Since $\L^2\cap\L^p$ is dense in $\L^p$, $p$ is finite, the result holds for all $f\in\L^p(\Omega)$. In particular, thanks to the identity \eqref{eq:ResolvIdentityMeanFreeNeumann} valid in $\L^2\cap\L^p(\Omega)$, then in $\L^p$ by density, one also obtains that the Neumann Laplacian restrict as an invertible $0$ sectorial operator in $\L^p_\circ(\Omega)$.

Additionally, thanks to \eqref{eq:FuncCalcIdentityMeanFreeNeumann}, for all $\varepsilon\in(0,1)$, $\varepsilon\I-\Delta_\mathcal{N}$ has bounded holomorphic functional calculus of angle $0$ in $\L^p_\circ(\Omega)$.

\textbf{Step 3.3:} We prove that $-\Delta_\mathcal{N}$ has bounded holomorphic functional calculus of angle $0$ in $\L^p_\circ(\Omega)$ for all $1<p<\infty$. 

Let $p\in(1,\infty)$. By the previous Step 3.2, choosing $\varepsilon_{0}\in(0,r_p)$ to be fixed, one has
\begin{enumerate}
    \item $\varepsilon_0\I-\Delta_\mathcal{N}$ is invertible and sectorial of angle $0$ and has a bounded holomorphic functional calculus in $\L^p_\circ(\Omega)$;
    \item $\sup_{t\geqslant 1} \lVert t (-\varepsilon_0)(t\I + \varepsilon_0\I-\Delta_{\mathcal{N}})^{-1} f\rVert_{\L^p(\Omega)} \lesssim_{p,\Omega} \varepsilon_0 \lVert f\rVert_{\L^p(\Omega)}$ for all $f\in\L^p_\circ(\Omega)$;
    \item $-\Delta_{\mathcal{N}} = (\varepsilon_0\I-\Delta_{\mathcal{N}}) + (- \varepsilon_0\I)$ is an invertible sectorial operator of angle $0$ in $\L^p_\circ(\Omega)$.
\end{enumerate}

By \cite[Cor.~5.5.5]{bookHaase2006}, $-\Delta_\mathcal{N}$ has bounded $\mathbf{H}^{\infty}(\Sigma_\theta)$-functional calculus in $\L^p_\circ(\Omega)$, for all $\theta\in(0,\pi)$ and all $1<p<\infty$, which ends the proof.
\end{proof}

Now, we can characterize the behavior for the fractional powers for the Neuamnn Laplacian in $\L^p_\circ(\Omega)$, we are then able to derive regularity theory for negative indices. Although, Mendez and Mitrea seems to have a similar claim \cite[Thm~1.1]{MendezMitrea2001}, we go a bit further giving more details, and allowing explicitly negative regularity function spaces with a stronger emphasis between the invertible and the non-invertible case.

\begin{proposition}\label{prop:NeumannRegularity}Let $\Omega$ be a bounded Lipschitz domain of $\RR^n$, then there exists $\varepsilon_{{}_\Omega}>0$, such that for all $1<p<\infty$, $s\in(-2+1/p,1+1/p)$ satisfying either
\begin{enumerate}
        \item $s\in(\sfrac{3}{p}-3-\varepsilon_{{}_\Omega},1+\sfrac{1}{p})$, if $p\in(1,\frac{2}{1+\varepsilon_{{}_\Omega}}]$; or
        \item $s\in(-2+\sfrac{1}{p},1+\sfrac{1}{p})$, if $p\in[\frac{2}{1+\varepsilon_{{}_\Omega}},\frac{2}{1-\varepsilon_{{}_\Omega}}]$; or
        \item $s\in(-2+\sfrac{1}{p},\sfrac{3}{p}+\varepsilon_{{}_\Omega})$, if $p\in[\frac{2}{1-\varepsilon_{{}_\Omega}},\infty)$.
\end{enumerate}
\begin{figure}[H]
\centering
\begin{tikzpicture}[yscale=1,xscale=8]

\def\littleparam{0.2}

  \draw[->] (-0.1,0) -- (1.1,0) node[right,yshift=-2mm] {$1/p$};
  \draw[->] (0,-2) -- (0,2.6) node[above] {$s$};
 \fill[red!30,opacity=0.5] (0,\littleparam) -- plot[domain=0:{(1-\littleparam)/2}](\x,3*\x+\littleparam)-- plot[domain={(1-\littleparam)/2}:1](\x,1+\x)  -- (1,-\littleparam)-- ({(1+\littleparam)/2},{-1-(1-\littleparam)/2})  -- (0,-2)-- cycle;

  \draw[domain=0:(5/6-\littleparam/3)),dashed,variable=\x,teal] plot (\x,3*\x+\littleparam) node[right,yshift=2mm] {$s=3/p+\varepsilon_{{}_\Omega}$};
  \draw[domain=(1/6+\littleparam/3):1,dashed,variable=\x,teal] plot (\x,3*\x-3-\littleparam);
  \node[anchor=north east] at (0.7,-2.05) {\color{teal}$s=3/p-3-\varepsilon_{{}_\Omega}$};
  
  \draw[domain=0:1,smooth,variable=\x,blue] plot ({\x},{-1+\x}) node[right,yshift=2mm] {$s=-1+1/p$};
  \draw[domain=0:1,smooth,variable=\x,blue] plot ({\x},{-2+\x}) node[right] {$s=-2+1/p$};
  \draw[domain=0:1,smooth,variable=\x,blue] plot ({\x},{1+\x}) node[right,yshift=2mm] {$s=1+1/p$};
  \draw[domain=0:1,smooth,variable=\x,blue] plot ({\x},{\x}) node[right]{$s=1/p$};

  \draw[dashed] (1,1) -- (0,1)  node[left] {$s=1$};
  \draw[dashed] (1,2) -- (0,2)  node[left] {$s=2$};
  \draw[dashed] (1,-1) -- (0,-1)  node[left] {$s=-1$};
  \node[circle,fill,inner sep=1.5pt,label=below:$\mathrm{L}^2$] at (0.5,0) {};
  \node[circle,fill,inner sep=1.5pt,label=above:${\mathrm{H}}^1$] at (0.5,1) {};
  \draw[circle,fill,inner sep=1pt] (0,0) node[below left] {$0$};
  \draw[circle,fill,inner sep=1pt] (1,0) node[below right] {$1$};
\end{tikzpicture}
\caption{Representation of parameters $(s,\sfrac{1}{p})$}
\label{Fig:RegularityFracNeuLap}
\end{figure}
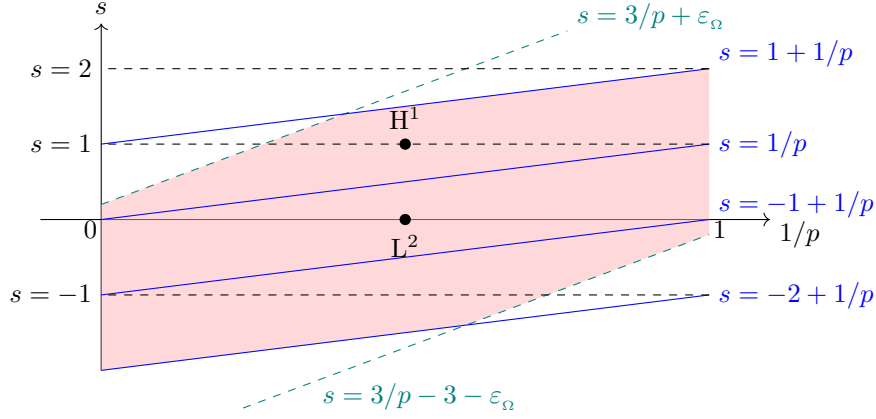
it holds that the mapping
\begin{align*}
    (-\Delta_\mathcal{N})^{s/2}\,:\,\H^{s,p}_\circ(\Omega)\longrightarrow\L^p_\circ(\Omega)
\end{align*}
is an isomorphism. In particular, the Neumann Laplacian is a sectorial operator of angle $0$ in $\H^{s,p}_0(\Omega)$ whenever $s\in(-2+1/p,-1+1/p)$ satisfies $s>3/p-3-\varepsilon_{{}_{\Omega}}$, and it restricts as an invertible sectorial operator of angle $0$ in $\H^{s,p}_\circ(\Omega)$ and with bounded $\mathbf{H}^\infty$-functional calculus of the same angle.
\end{proposition}

\begin{proof}Before actually starting the proof, we recall that, by \cite[Thm~9.2]{FabesMendezMitrea1998}, there exists $\varepsilon_{{}_\Omega}>0$ such that for all $p\in(1,\infty)$, $s\in(-1+1/p,1/p)$ satisfying either
\begin{enumerate}
        \item $s\in(\sfrac{3}{p}-2-\varepsilon_{{}_\Omega},\sfrac{1}{p})$, if $p\in(1,\frac{2}{1+\varepsilon_{{}_\Omega}}]$; or
        \item $s\in(-2+\sfrac{1}{p},\sfrac{1}{p})$, if $p\in[\frac{2}{1+\varepsilon_{{}_\Omega}},\frac{2}{1-\varepsilon_{{}_\Omega}}]$; or
        \item $s\in(-1+\sfrac{1}{p},\sfrac{3}{p}-1+\varepsilon_{{}_\Omega})$, if $p\in[\frac{2}{1-\varepsilon_{{}_\Omega}},\infty)$.
\end{enumerate}
\begin{figure}[H]
\centering
\begin{tikzpicture}[yscale=1,xscale=8]

\def\littleparam{0.2}

  \draw[->] (-0.1,0) -- (1.1,0) node[right,yshift=-2mm] {$1/p$};
  \draw[->] (0,-1.6) -- (0,1.6) node[above] {$s$};

  \draw[domain=0:(5/6-\littleparam/3)),dashed,variable=\x,teal] plot (\x,3*\x-1+\littleparam) node[right,yshift=2mm] {$s=3/p-1+\varepsilon_{{}_\Omega}$};
  \draw[domain=(1/6+\littleparam/3):1,dashed,variable=\x,teal] plot (\x,3*\x-2-\littleparam);
  \node[anchor=north east] at (0.7,-1.05) {\color{teal}$s=3/p-2-\varepsilon_{{}_\Omega}$};
  
  \draw[domain=0:1,smooth,variable=\x,blue] plot ({\x},{-1+\x}) node[right,yshift=2mm] {$s=-1+1/p$};
  \draw[domain=0:1,smooth,variable=\x,blue] plot ({\x},{\x}) node[right]{$s=1/p$};
 \fill[green!30,opacity=0.5] (0,-1+\littleparam) -- plot[domain=0:{(1-\littleparam)/2}](\x,3*\x-1+\littleparam)-- plot[domain={(1-\littleparam)/2}:1](\x,\x)  -- (1,1-\littleparam)-- ({(1+\littleparam)/2},{-(1-\littleparam)/2})  -- (0,-1)-- cycle;

  \draw[dashed] (1,1) -- (0,1)  node[left] {$s=1$};
  \node[circle,fill,inner sep=1.5pt,label=below:$\mathrm{L}^2$] at (0.5,0) {};
  \node[circle,fill,inner sep=1.5pt,label=above:${\mathrm{H}}^1$] at (0.5,1) {};
  \draw[circle,fill,inner sep=1pt] (0,0) node[below left] {$0$};
  \draw[circle,fill,inner sep=1pt] (1,0) node[below right] {$1$};
\end{tikzpicture}
\caption{Representation of $(s,\sfrac{1}{p})$.}
\label{Fig:RegularityNeuLapLipDomains}
\end{figure}
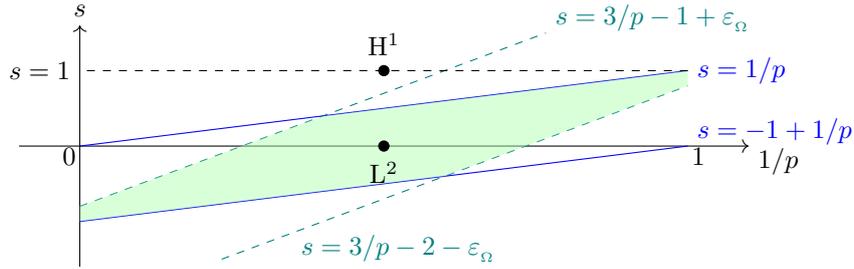
then for all $f\in\H^{s-1,p}_\circ(\Omega)$ (\textit{i.e.} $f\in\H^{s-1,p}_{0}(\Omega)$ such that $\langle f ,\,\mathbf{1}\rangle_{\Omega}=0$), there exists a unique $u\in \H^{s+1,p}_\circ(\Omega)$ (\textit{i.e.} $u\in\H^{s+1,p}(\Omega)$ such that $\langle u ,\,\mathds{1}_{\Omega}\rangle_{\Omega}=0$), satisfying
\begin{align*}
    -\Delta u =f,\quad\text{ in }\Omega,\quad\text{and}\quad\partial_\mathbf{n}u_{|_{\partial\Omega}} =0,
\end{align*}
with the estimate
\begin{align*}
    \lVert \nabla u \rVert_{\H^{s,p}(\Omega)} \lesssim_{p,s,n,\Omega}\lVert f \rVert_{\H^{s-1,p}(\Omega)}.
\end{align*}

\textbf{Step 1:} The case $p=2$, $s\in(-3/2,3/2)$. By the bounded holomorphic functional calculus in $\L^2$, in particular BIP, see Theorem~\ref{thm:BIP}, for all $\alpha\in[0,1]$, one has
\begin{align*}
    \D_{2,\circ}((-\Delta_\mathcal{N})^{\alpha/2}) = [\L^2_\circ(\Omega),\D_{2,\circ}((-\Delta_\mathcal{N})^{1/2}) ]_{\alpha} = [\L^2_\circ(\Omega),\H^{1,2}_\circ(\Omega) ]_{\alpha} = \H^{\alpha,2}_\circ(\Omega),
\end{align*}
with an isomorphism
\begin{align*}
    (-\Delta_\mathcal{N})^{\alpha/2}\,:\,\H^{\alpha,2}_\circ(\Omega)\longrightarrow\L^2_\circ(\Omega).
\end{align*}
By self-adjointness in $\L^2$ and duality, it remains valid for $\alpha\in[-1,0]$ \footnote{We recall for the reader's clarity that with our notations $(\H^{1,2}_\circ(\Omega))' =(\P_\circ\H^{1,2}(\Omega))' = \P_\circ(\H^{1,2}(\Omega))' = \P_\circ\H^{-1,2}_0(\Omega) = \H^{-1,2}_\circ(\Omega)$.}. In particular, one can export the bounded holomorphic function calculus to $\H^{\alpha,2}_\circ(\Omega)$ for all $\alpha\in[-1,1]$. However, by The Fabes Mendez and Mitrea result recalled above, for all $\alpha\in(1/2,3/2)$
\begin{align*}
    -\Delta_{\mathcal{N}}\,:\,\H^{\alpha,2}_\circ(\Omega)\longrightarrow\H^{\alpha-2,2}_\circ(\Omega)
\end{align*}
is an isomorphism, thus for all $s\in[-1,-1/2)$, $(-\Delta_\mathcal{N})^{-s/2}=(-\Delta_\mathcal{N})^{1-s/2}(-\Delta_\mathcal{N})$ is an isomorphism from $\H^{s+2,2}_\circ(\Omega)$ to $\L^{2}_{\circ}(\Omega)$. Thus, we obtain the isomorphism $(-\Delta_{\mathcal{N}})^{\alpha/2}\,:\,\H^{\alpha,2}_\circ(\Omega)\longrightarrow\L^{2}_{\circ}(\Omega)$, $\alpha\in[-1,3/2)$. By duality, we recover the full interval $s\in(-3/2,3/2)$.

\textbf{Step 2:} Now, the $\L^p$-case. From Proposition~\ref{prop:NeumannHinfty}, one has 
\begin{enumerate}
    \item For $\Re (z)=0$, $(-\Delta_\mathcal{N})^{\sfrac{z}{2}}$ acts as an isomorphism on $\L^p_\circ(\Omega)$, for all $p\in(1,\infty)$,
    \item For $\Re (z)=s$, $(-\Delta_\mathcal{N})^{\sfrac{z}{2}}$ is an isomorphism mapping $\H^{s,2}_\circ(\Omega)$ to $\L^2_\circ(\Omega)$;
\end{enumerate}
By complex interpolation, and duality, for all $\alpha\in(-3/2,3/2)$, such that $2|\frac{1}{p}-\frac{1}{2}|+\frac{2}{3}|\alpha|<1$, we obtain
\begin{align*}
    \D_{p,\circ}((-\Delta_\mathcal{N})^{{\alpha}/{2}}) = \H^{\alpha,p}_\circ(\Omega)
\end{align*}
accompanied with the underlying isomorphism. Now, choosing $p$ close to $1$, for all $s\in(-1+1/p-\varepsilon_{{}_\Omega},0)$ and such that $2|\frac{1}{p}-\frac{1}{2}|+\frac{2}{3}|s|<1$, since $-\Delta_{\mathcal{N}}\,:\,\H^{s+2,p}_\circ(\Omega)\longrightarrow\H^{s,p}_{\circ}(\Omega)$ is an isomorphism, so is $(-\Delta_{\mathcal{N}})^{1+s/2}=(-\Delta_{\mathcal{N}})^{s/2}(-\Delta_{\mathcal{N}})\,:\,\H^{s+2,p}_\circ(\Omega)\longrightarrow\L^{p}_{\circ}(\Omega)$. Now applying complex interpolation again to $(-\Delta)^{z/2}$, this yields the desired results for all $1<p\leqslant 2$, $s\in(-1+1/p,2+1/p)$, with isomorphism $(-\Delta_{\mathcal{N}})^{s/2}\,:\,\H^{s,p}_\circ(\Omega)\longrightarrow\L^{p}_{\circ}(\Omega)$, and by duality this yields the case $2\leqslant p<\infty$,  $s\in(-2+1/p,1/p)$.

Now let $p\in[\frac{2}{1-\varepsilon_{{}_\Omega}},\infty)$, for all $1/p<s<3/p-1+\varepsilon_{{}_\Omega}$, $\Delta_{\mathcal{N}}\,:\,\H^{s,p}_\circ(\Omega)\longrightarrow\H^{s-2,p}_{\circ}(\Omega)$ is an isomorphism, so is $(-\Delta_{\mathcal{N}})^{s/2}=(-\Delta_{\mathcal{N}})^{s/2-1}(-\Delta_{\mathcal{N}})\,:\,\H^{s,p}_\circ(\Omega)\longrightarrow\L^{p}_{\circ}(\Omega)$, complex interpolation yields the case $s=1/p$. Finally, duality and complex interpolation allows to fulfill the figure which ends the proof. In particular, the Neumann Laplacian is sectorial of angle $0$ in spaces $\H^{s,p}_\circ(\Omega)$ for all such indices, and it has bounded $\mathbf{H}^\infty$-functional calculus of angle $0$.

\textbf{Step 4:} Sectoriality in Sobolev spaces: We give the argument for sectoriality of angle $0$ in the bigger spaces $\H^{s,p}_0(\Omega)$, $s<-1+1/p$.

Let $\mu\in(0,\pi)$, $\lambda\in\Sigma_\mu$ and $f\in\H^{s,p}_{0}(\Omega)\cap\L^2(\Omega)$, then it holds
\begin{align*}
    (\lambda \I-\Delta_{\mathcal{N}})^{-1}f = (\lambda \I-\Delta_{\mathcal{N}})^{-1}\P_\circ f + \frac{1}{\lambda} (f)_{\Omega}\mathds{1}_{\Omega}, \quad \text{ in }\L^2(\Omega).
\end{align*}
By the case of mean-free function spaces, we deduce
\begin{align*}
    \lVert\lambda(\lambda \I-\Delta_{\mathcal{N}})^{-1}f \rVert_{\H^{s,p}(\Omega)} &\lesssim_{p,s,n,\Omega,\mu} \lVert\lambda(\lambda \I-\Delta_{\mathcal{N}})^{-1} \P_\circ f \rVert_{\H^{s,p}(\Omega)} +\lVert (f)_{\Omega}\mathds{1}_{\Omega} \rVert_{\H^{s,p}(\Omega)} \\ &\lesssim_{p,s,n,\Omega,\mu} \lVert \P_\circ f \rVert_{\H^{s,p}(\Omega)} +\lVert (f)_{\Omega}\mathds{1}_{\Omega} \rVert_{\H^{s,p}(\Omega)} \\&\lesssim_{p,s,n,\Omega,\mu} \lVert  f \rVert_{\H^{s,p}(\Omega)}.
\end{align*}
the result follows by strong density of $\H^{s,p}_{0}(\Omega)\cap\L^2(\Omega)$ in $\H^{s,p}_{0}(\Omega)$.
\end{proof}

We now state the regularity properties for the Neumann Laplacian in Besov spaces.

\begin{proposition}\label{prop:NeumannBesov}Let $\Omega$ be a bounded Lipschitz domain of $\RR^n$. Let $\varepsilon_{{}_\Omega}>0$ be given by Proposition~\ref{prop:NeumannRegularity}. Let $p\in(1,\infty)$, $s\in(-2+1/p,-1+1/p)$ such that $s>3/p-3+\varepsilon_{{}_\Omega}$. Let $q\in[1,\infty]$. Then the following assertions hold
\begin{enumerate}
    \item The Neumann Laplacian is sectorial of angle $0$ in $\B^{s}_{p,q,0}(\Omega)$, and its restriction to $\B^{s}_{p,q,\circ}(\Omega)$ is invertible and remains sectorial of angle $0$, and both are densely defined whenever $q<\infty$;
    \item Provided $s+2< 3/p+\varepsilon_{{}_\Omega}$, its domains are respectively $\B^{s+2}_{p,q}(\Omega)$ and $\B^{s+2}_{p,q,\circ}(\Omega)$;
    \item For all $\theta\in(0,1)$, without any restriction on $s$, asking for $s+2\theta< 3/p+\varepsilon_{{}_\Omega}$, for all $r\in[1,\infty]$, it holds
    \begin{align*}
        (\B^{s}_{p,q,0}(\Omega),\D^{s}_{p,q}(\Delta_{\mathcal{N}}))_{\theta,r} = \begin{cases}
            \B^{s+2\theta}_{p,r,0}(\Omega),\quad\text{ if } s+2\theta <1/p;\\
            \B^{s+2\theta}_{p,r}(\Omega),\quad\text{ if } s+2\theta  >-1+1/p;\\
    \end{cases}
    \end{align*}
    and similarly
   \begin{align*}
        (\B^{s}_{p,q,\circ}(\Omega),\D^{s}_{p,q,\circ}(\Delta_{\mathcal{N}}))_{\theta,r} = \B^{s+2\theta}_{p,r,\circ}(\Omega);
    \end{align*}
    \item Setting $I_p:=(-2+1/p,1+1/p)\cap(3/p-3-\varepsilon_{{}_\Omega},3/p+\varepsilon_{{}_\Omega})$, for all $\tau,\beta\in I_p$ such that $\tau+\beta \in I_p$,
    \begin{align*}
        (-\Delta_\mathcal{N})^{\beta/2}\,:\,\B^{\tau+\beta}_{p,q,\circ}(\Omega)\longrightarrow\B^{\tau}_{p,q,\circ}(\Omega)
    \end{align*}
    is an isomorphism.
\end{enumerate}
\end{proposition}

\begin{proof} Everything is a direct consequence of real interpolation applied to Proposition~\ref{prop:NeumannRegularity}.
\end{proof}

In the case of the $\L^p$-theory for the Boussinesq system, we will need, see the proof of Proposition~\ref{Prop:SolutionOperatorBoussinesqLp} below, the boundedness of the following linear operator when $\Omega\subset\RR^3$,
\begin{align*}
    (-\Delta_{\mathcal{N}})^{-1/2}\div\,:\, \B^{-1/2}_{6/5,\infty,0}(\Omega,\CC^3)\longrightarrow\B^{-1/2}_{6/5,\infty,\circ}(\Omega),
\end{align*}
but by duality and the mapping properties, it can be checked that this is equivalent to the isomorphism $\Delta_{\mathcal{N}}\,:\,\B^{3/2}_{6,1,\circ,\mathcal{N}}(\Omega)\longrightarrow\B^{-1/2}_{6,1,\circ}(\Omega)$, where $$\B^{3/2}_{6,1,\circ,\mathcal{N}}(\Omega):= \{u\in \B^{3/2}_{6,1,\circ}(\Omega)\,|\,\partial_\mathbf{n} u {}_{|_{\partial\Omega}}=0\,\}.$$ This property cannot be reached in an arbitrary bounded Lipschitz domain.

\begin{proposition}\label{prop:RieszTransfC1alphaDomainNeumannLap}Let $\Omega\subset \mathbb{R}^3$ be a bounded $\C^{1,\alpha}$ domain, with $\alpha>1/3$. Then, one has an isomorphism
\begin{align}\label{eq:refHigherRgneumannC1aplha}
    -\Delta_{\mathcal{N}}\,:\,\B^{3/2}_{6,1,\circ,\mathcal{N}}(\Omega)\longrightarrow\B^{-1/2}_{6,1,\circ}(\Omega).
\end{align}
In particular,
\begin{align}\label{NeumannRieszTransf}
     (-\Delta_{\mathcal{N}})^{-1/2}\div\,:\, \B^{-1/2}_{6/5,\infty,0}(\Omega,\CC^3)\longrightarrow\B^{-1/2}_{6/5,\infty,\circ}(\Omega)
\end{align}
is well-defined and bounded.
\end{proposition}

\begin{proof}  We just exhibit place in the literature for the arguments for \eqref{eq:refHigherRgneumannC1aplha} to hold, with the location in the literature to find the different parts of the arguments to prove such a statement, in order to not substantially increase the length of the present work\footnote{Otherwise, this would require us to introduce tremendous amount of concepts, heavy notations and so on, while not giving new information.}. By Proposition~\ref{prop:NeumannRegularity}, consider $f\in\B^{-1/2}_{6,1,\circ}(\Omega)\subset\H^{-1,2}_{\circ}(\Omega)$ and $u\in\H^{1,2}_{\circ}(\Omega)$ the solution to 
\begin{align*}
    -\Delta u =f,\quad\text{ in }\Omega,\quad\text{and}\quad\partial_\mathbf{n}u_{|_{\partial\Omega}} =0.
\end{align*}
The idea is to prove that furthermore $u\in\B^{3/2}_{6,1}(\Omega)$, one can localize and reduce to many finitely problems in the half-space following \cite[Proof of Thm. B.1]{BreitSchwarzacher2025} but applying the localization estimates arguments in negative Sobolev spaces from the \cite[Proof of Thm~6.2]{BreitGaudin2025}\footnote{The proof simplifies then by a lot in the present case of the Neumann Laplacian. It actually becomes  quite close in spirit to the one already performed for \cite[Thm. B.1]{BreitSchwarzacher2025}}. What allows us to consider $\C^{1,1/3+\varepsilon}$ instead of  $\C^{1,1/2+\varepsilon}$ is the use of Maz'ya and Shaposhnikova's  Sobolev multiplier theory \cite{MazyaShaposhnikova2009} in  \cite{BreitSchwarzacher2025,BreitGaudin2025}, refined by the author and Dominic Breit for their use in (some) negative regularity Sobolev spaces in \cite[Chap.~2,~Prop.~2.38]{BreitGaudin2025}. One could even relax the result to a rougher, but more fancy, boundary regularity, see for instance \cite[Prop.~2.41]{BreitGaudin2025}. More generally, one can obtain the exact Neumann Laplacian-counterpart of \cite[Thm.~6.5]{BreitGaudin2025} by just substituting the appropriate notations.
\end{proof}

\paragraph{The Stokes operator}

Let $\Omega\subset\RR^n$, $n\geqslant 2$,  be a bounded Lipschitz domain. Before, introducing the Stokes operator and the Stokes equations with their regularity properties, we introduce the wealth of solenoidal function spaces that accommodate the appropriate features. The space of the smooth divergence free vector fields is introduced as:
\begin{align*}
    \Ccinftydiv(\Omega) := \{ \vv\in\Ccinfty(\Omega,\CC^n)\,|\, \div \vv =0 \}.
\end{align*}
The adapted Sobolev and Besov spaces for the Stokes operator are introduced as
\begin{align*}
    \H^{s,p}_{0,\sigma}(\Omega)=\overline{\Ccinftydiv(\Omega)}^{\lVert \cdot\rVert_{\H^{s,p}(\Omega)}},\quad\text{ and } \B^{s,\sigma}_{p,q,0}(\Omega)=\overline{\Ccinftydiv(\Omega)}^{\lVert \cdot\rVert_{\B^{s}_{p,q}(\Omega)}}
\end{align*}
for all $s\in\mathbb{R}$, $p\in[1,\infty]$, $q\in[1,\infty)$, the case $q=\infty$ could be defined by real interpolation. It is known, see for instance \cite[Cor.~2.11]{MitreaMonniaux2008} and \cite[Thm~4.17,~Prop.~4.22]{BreitGaudin2025}, whenever $1<p<\infty$ and $s\neq1/p$, that we can canonically identify
\begin{align*}
    \H^{s,p}_{0,\sigma}(\Omega)=\begin{cases}
        \{\vv\in\H^{s,p}(\Omega,\CC^n)\,|\,\div \vv = 0,\,\text{ and }\, \vv\cdot\mathbf{n}_{|_{\partial\Omega}}=0\},\quad &\text{ if } -1+1/p<s<1/p,\\
        \{\vv\in\H^{s,p}(\Omega,\CC^n)\,|\,\div \vv = 0,\,\text{ and }\, \vv_{|_{\partial\Omega}}=0\}, &\text{ if } \quad\qquad1/p<s<1+1/p,
    \end{cases}
\end{align*}
and similarly in the case of Besov spaces $\B^{s,\sigma}_{p,q,0}(\Omega)$. Both are stables complex and real interpolation scales, except at the threshold cases $s=-1+1/p,1/p$, see \cite[Discussion following Prop.~2.10,~Thm.~2.12]{MitreaMonniaux2008}, \cite[Thm.~4.19,~Prop.~4.23]{BreitGaudin2025}.
\medbreak

For simplicity, we set $\L^p_{\mathfrak{n},\sigma}(\Omega):= \H^{0,p}_{0,\sigma}(\Omega)$, and $\W^{k,p}_{0,\sigma}(\Omega):=\H^{k,p}_{0,\sigma}(\Omega)$, $1<p<\infty$, $k\in\NN^{\ast}$. When $p=2$, it well known that there exists a unique orthogonal projection
\begin{align*}
    \PP_\Omega\,:\,\L^2(\Omega,\CC^n) \longrightarrow \L^2_{\mathfrak{n},\sigma}(\Omega)
\end{align*}
called the \emph{Helmholtz--Leray projection} and inducing the topological orthogonal decomposition
\begin{align*}
    \L^2(\Omega,\CC^n) = \L^2_{\mathfrak{n},\sigma}(\Omega) \oplus \nabla \H^{1,2}(\Omega)
\end{align*}
called the \emph{Helmholtz decomposition}. When it comes to Sobolev and Besov spaces with $p\neq2$, the question of such decomposition becomes way more intricate, especially in domains of low regularity. This actually deeply related to solvability of the Neumann Laplacian in the corresponding function spaces, see \cite[Sec.~9~\&~11]{FabesMendezMitrea1998} and \cite[Prop.~2.16]{MitreaMonniaux2008}, where the optimal results in Lipschitz domains are given. Here is the result, initially obtained in \cite[Prop.~2.16]{MitreaMonniaux2008}, improved to the case of Besov spaces by applying real interpolation to it:
\begin{proposition}\label{prop:HelmholtzDecomp}Let $\Omega$ be a bounded Lipschitz domain of $\RR^n$, $n\geqslant 2$. Let $\varepsilon_{{}_{\Omega}}>0$ be given by Proposition~\ref{prop:NeumannHinfty}, then  for all $1<p<\infty$, $q\in[1,\infty]$, $s\in(-1+1/p,1/p)$ satisfying either
\begin{enumerate}
        \item $s\in(\sfrac{3}{p}-2-\varepsilon_{{}_\Omega},\sfrac{1}{p})$, if $p\in(1,\frac{2}{1+\varepsilon_{{}_\Omega}}]$; or
        \item $s\in(-1+\sfrac{1}{p},\sfrac{1}{p})$, if $p\in[\frac{2}{1+\varepsilon_{{}_\Omega}},\frac{2}{1-\varepsilon_{{}_\Omega}}]$; or
        \item $s\in(-1+\sfrac{1}{p},\sfrac{3}{p}-1+\varepsilon_{{}_\Omega})$, if $p\in[\frac{2}{1-\varepsilon_{{}_\Omega}},\infty)$;
\end{enumerate}
see Figure~\ref{Fig:RegularityNeuLapLipDomains}, the following topological decompositions holds
\begin{align*}
    \H^{s,p}(\Omega,\CC^{n}) &= \H^{s,p}_{0,\sigma}(\Omega)\oplus \nabla \H^{s+1,p}(\Omega),\\
    \text{ and }\quad\B^{s}_{p,q}(\Omega,\CC^{n}) &= \B^{s,\sigma}_{p,q,0}(\Omega)\oplus \nabla \B^{s+1}_{p,q}(\Omega),
\end{align*}
induced by the extension of the Leray projection $\PP_{\Omega}$ as a bounded linear operator
\begin{align*}
    \PP_{\Omega}&\,:\,\H^{s,p}(\Omega,\CC^{n})\longrightarrow\H^{s,p}_{0,\sigma}(\Omega);\\
    &\,:\,\B^{s}_{p,q}(\Omega,\CC^{n}) \longrightarrow \B^{s,\sigma}_{p,q,0}(\Omega).
\end{align*}
In particular, for all $p\in[3/2,3]$ it holds true for $\L^p(\Omega)$ and $\B^{0}_{p,q}(\Omega)$.
\end{proposition}
 This latter proposition, implies that for all $p\in(1,\infty)$, $s\in(-1+1/p,1/p)$, satisfying the same conditions, \textit{i.e.} $(1/p,s)$ lying in the green area of Figure~\ref{Fig:RegularityNeuLapLipDomains}, that one has the duality property
\begin{align}
    \H^{s,p}_{0,\sigma}(\Omega) = (\H^{-s,p'}_{0,\sigma}(\Omega))'.\label{eq:DualitySolenoidal}
\end{align}

 We now introduce the Stokes--Dirichlet operator drawing inspiration from \cite[Sec.~4]{MitreaMonniaux2008}, \cite[Sec.~2]{MonniauxShen2018}. On bounded Lipschitz domains, we endow $\H^{1,2}_{0}(\Omega)$ and $\H^{1,2}_{0,\sigma}(\Omega)$ with the norm $\lVert \nabla \cdot \rVert_{\L^{2}}$. We set
\begin{enumerate}
    \item $\iota\,:\,\L^2_{\mathfrak{n},\sigma}(\Omega)\longrightarrow \L^2(\Omega,\CC^n)$ to be the canonical embedding;
    \item $\PP_{\Omega}\,:\,\L^2(\Omega,\CC^n)\longrightarrow\L^2_{\mathfrak{n},\sigma}(\Omega)$ to be the orthogonal Leray projection, we have that $\iota'=\PP_{\Omega}$ with $\PP_{\Omega}\iota=\mathbf{I}_{\L^2_{\mathfrak{n},\sigma}}$;
    \item $\iota_0\,:\,\H^{1,2}_{0,\sigma}(\Omega)\longrightarrow \H^{1,2}_{0}(\Omega,\CC^n)$ to be the canonical embedding, with the property $\iota_{|_{\H^{1,2}_{0,\sigma}}}=\iota_0$;
    \item $\PP_{1}\,:\,\H^{-1,2}(\Omega,\CC^n)\longrightarrow(\H^{1,2}_{0,\sigma}(\Omega))'$, the projection given by $\PP_1:=\iota_0'$.
\end{enumerate}

\begin{definition}The Stokes--Dirichlet operator is the unbounded $\L^2_{\mathfrak{n},\sigma}$-realization of the bounded operator $\AA_{\mathcal{D}}\,:\,\H^{1,2}_{0,\sigma}(\Omega)\longrightarrow (\H^{1,2}_{0,\sigma}(\Omega))'$, denoted by $(\D_2(\AA_{\mathcal{D}}),\AA_{\mathcal{D}})$, and induced by the sesquilinear form
\begin{align*}
    \mathfrak{a}_{\mathcal{D},\sigma}\,:\,\D(\mathfrak{a}_{\mathcal{D},\sigma})\times \D(\mathfrak{a}_{\mathcal{D},\sigma}) &\longrightarrow \CC\\
    (\uu ,\vv)\quad\qquad&\longmapsto \langle \nabla \iota \uu, \nabla \iota \vv\rangle_{\Omega} = \int_{\Omega} \nabla \uu (x) : \overline{\nabla \vv(x)}\,\d x,
\end{align*}
with form domain $\D(\mathfrak{a}_{\mathcal{D},\sigma}):=\H^{1,2}_{0,\sigma}(\Omega)$.
\end{definition}
In particular, \cite[Thm.~4.7]{MitreaMonniaux2008},  \cite[Prop.~1]{MonniauxShen2018}, we have a suitable standard description in the Hilbert space setting $(\H^{1,2}_{0,\sigma}(\Omega),\L^2_{\mathfrak{n},\sigma}(\Omega), (\H^{1,2}_{0,\sigma}(\Omega))')$ to solve uniquely the so called (resolvent) Stokes equations: provided $\mu\in(0,\pi)$, for any $\lambda\in\Sigma_\mu\cup\{0\}$ and any $\ff \in\L^2_{\mathfrak{n},\sigma}(\Omega)$, there exists a unique $\uu\in\H^{1,2}_{0,\sigma}(\Omega)$, and there exists $\mathfrak{p}\in\L^{2}(\Omega)$,
\begin{equation*}\tag{DS${}_\lambda$}\label{eq:DirStokesSystemBddLip}
    \left\{ \begin{array}{rllr}
         \lambda \uu - \Delta \uu +\nabla \mathfrak{p} &= \ff \text{, }&&\text{ in } \Omega\text{,}\\
        \div \uu &= 0\text{, } &&\text{ in } \Omega\text{,}\\
        \uu_{|_{\partial\Omega}} &=0\text{, } &&\text{ on } \partial\Omega\text{.}
    \end{array}
    \right.
\end{equation*}
called the Stokes--Dirichlet (resolvent) equations. It holds that
\begin{align*}
    \mathbb{A}_\mathcal{D}\uu = -\Delta \uu + \nabla \mathfrak{p},\qquad \text{ in }\H^{-1,2}(\Omega).
\end{align*}
Therefore, one has the description
\begin{align*}
    \D_2(\AA_\mathcal{D})=\{\, \uu\in\W^{1,2}_{0,\sigma}(\Omega)\,|\, \exists\, \mathfrak{p}\in\L^2(\Omega),\,-\Delta\uu+\nabla\mathfrak{p}\in\L^{2}_{\mathfrak{n},\sigma}(\Omega)\,\}.
\end{align*}
Now, in $\L^p$, since $\Omega$ is bounded, we can define
\begin{align*}
    \D_p(\AA_\mathcal{D}):=\{\, \uu\in\D_{2}(\AA_{\mathcal{D}})\cap\L^p_{\mathfrak{n},\sigma}(\Omega)\,|\, \mathbb{A}_\mathcal{D}\uu\in\L^{p}_{\mathfrak{n},\sigma}(\Omega)\,\},\quad \text{ if } p>2,
\end{align*}
and if $p<2$, provided $\mathbb{A}_{\mathcal{D}}$ is closable in $\L^{p}_{\mathfrak{n},\sigma}(\Omega)$, one sets the domain in $\L^p$ as the closure of $\AA_\mathcal{D}$, \textit{i.e.},
\begin{align*}
    \D_p(\AA_\mathcal{D})&:=\{\, \vv\in \L^p_{\mathfrak{n},\sigma}(\Omega)\,|\, \exists\,\ff\in\L^p_{\mathfrak{n},\sigma}(\Omega),\,\exists (\vv_{k})_{k\in\NN}\subset\D_{2}(\AA_{\mathcal{D}}),\, \AA_\mathcal{D}\vv_k\xrightarrow[k\rightarrow+\infty]{}\ff\,\text{ in }\,\L^p_{\mathfrak{n},\sigma}(\Omega)\},\\
    \AA_\mathcal{D}\uu &:= \lim_{k\rightarrow +\infty}  \AA_\mathcal{D}\uu_k  = \ff.
\end{align*}
Note that both definitions are consistent with the ones that we obtain when we are able to extend the resolvent $(\lambda\I+\mathbb{A}_{\mathcal{D}})^{-1}$ from $\L^{2}_{\mathfrak{n},\sigma}\cap \L^{p}_{\mathfrak{n},\sigma}(\Omega)$ to the whole $ \L^{p}_{\mathfrak{n},\sigma}(\Omega)$ by density.

\begin{theorem}\label{thm:StokesLipSobolev} Let $\Omega$ be a bounded Lipschitz domain of $\RR^n$. There exists $\varepsilon_{\Omega}\in(0,(n-1)/2n)$, such that for all $p\in(1,\infty)$ satisfying
\begin{align}\label{eq:CondtnLpStokesLip}
    \left| \frac{1}{p}-\frac{1}{2}\right|<\frac{1}{2n}+\varepsilon_{\Omega}
\end{align}
it holds that the Stokes operator $\mathbb{A}_\mathcal{D}$ is an invertible sectorial of angle $0$ in $\L^p_{\mathfrak{n},\sigma}(\Omega)$ and has bounded $\mathbf{H}^\infty(\Sigma_\theta)$-functional calculus for all $\theta\in(0,\pi)$.
\medbreak

\noindent Additionally, there exists $\delta_{{}_\Omega,p}>0$, such that for all $s\in(-1+1/p,1+\delta_{{}_\Omega,p})$,  $s\neq1/p$, one has an isomorphism
    \begin{align}
        \mathbb{A}_\mathcal{D}^{s/2}\,:\,\H^{s,p}_{0,\sigma}(\Omega)\longrightarrow\L^p_{\mathfrak{n},\sigma}(\Omega).\label{eq:IsomFracPowerStokes}
    \end{align}
In particular, when $n=3$, everything holds true for $p\in[3/2,3]\subset( \frac{3}{2+3\varepsilon_{{}_{\Omega}}},\frac{3}{1-3\varepsilon_{{}_{\Omega}}})$.
\end{theorem}

\begin{proof} \textbf{Step 1:} When $\delta_{{}_\Omega,p}=0$ the full result is already somehow known. Indeed by \cite[Thm.~1.1]{Shen2012} yields sectoriality of angle $0$ and invertibility, then \cite[Thm.~16]{KunstmannWeis2017} proved the bounded $\mathbf{H}^{\infty}$-functional calculus of the same angle in $\L^{p}_{\mathfrak{n},\sigma}(\Omega)$ for all $p$ satisfying \eqref{eq:CondtnLpStokesLip}. So that by invertibility and bounded $\mathbf{H}^{\infty}$-functional calculus, it has Bounded Imaginary Powers, and for all $\alpha\in[0,2]$, by Theorem~\ref{thm:BIP}, it holds
\begin{align}
    [\L^{p}_{\mathfrak{n},\sigma}(\Omega),\D_p(\AA_{\mathcal{D}})]_{\alpha/2} = \D_{p}(\mathbb{A}_{\mathcal{D}}^{\alpha/2}),\label{eq:CompInterpBIpStokes}
\end{align}
with the underlying isomorphism
\begin{align*}
    \mathbb{A}_{\mathcal{D}}^{\alpha/2}\,:\,[\L^{p}_{\mathfrak{n},\sigma}(\Omega),\D_p(\AA_{\mathcal{D}})]_{\alpha/2} \longrightarrow \L^p_{\mathfrak{n},\sigma}(\Omega).
\end{align*}
For the same $p$, it has been proved, \cite[Thm.~1.1]{Tolksdorf2018-1}, that
\begin{align*}
    \mathbb{A}_{\mathcal{D}}^{1/2}\,:\,\W^{1,p}_{0,\sigma}(\Omega) \longrightarrow \L^p_{\mathfrak{n},\sigma}(\Omega) 
\end{align*}
is an isomorphism. Therefore, by complex interpolation and BIP, Theorem~\ref{thm:BIP}, for all $s\in[0,1]$, $s\neq1/p$, and the interpolation results \cite[Thm.~2.12]{MitreaMonniaux2008}, \cite[Prop.~4.23]{BreitGaudin2025},
\begin{align}
    [\L^{p}_{\mathfrak{n},\sigma}(\Omega),\D_p(\AA_{\mathcal{D}})]_{s/2} = [\L^{p}_{\mathfrak{n},\sigma}(\Omega),\D_p(\AA_{\mathcal{D}}^{1/2})]_{s}= [\L^{p}_{\mathfrak{n},\sigma}(\Omega),\W^{1,p}_{0,\sigma}(\Omega)]_{s} = \H^{s,p}_{0,\sigma}(\Omega) \label{eq:CompInterpDomainStokesLip}
\end{align}
with the underlying isomorphism
\begin{align*}
    \mathbb{A}_{\mathcal{D}}^{s/2}\,:\,\H^{s,p}_{0,\sigma}(\Omega) \longrightarrow \L^p_{\mathfrak{n},\sigma}(\Omega).
\end{align*}
By duality, \eqref{eq:DualitySolenoidal}, the isomorphism remains valid for all $s\in(-1+1/p,0]$, thus has been obtained for all $s\in(-1+1/p,1]$.

\textbf{Step 2:} We prove that we can actually go slightly beyond regularity $s=1$ for the identification of fractional powers. By \eqref{eq:CompInterpDomainStokesLip}, for $s\in[0,1]$, $s\neq 1/p$ one has
\begin{align*}
    [\L^{p}_{\mathfrak{n},\sigma}(\Omega),\D_p(\AA_{\mathcal{D}})]_{s/2}  = \H^{s,p}_{0,\sigma}(\Omega) = [\L^{p}_{\mathfrak{n},\sigma}(\Omega),\W^{2,p}_{0,\sigma}(\Omega)]_{s/2},
\end{align*}
with the underlying canonical isomorphism. One has the continuous embedding $\W^{2,p}_{0,\sigma}(\Omega)\hookrightarrow \D_{p}(\AA_{\mathcal{D}})$, \cite[Lem.~2.5]{Tolksdorf2018-1}, hence, the identity map
\begin{itemize}
    \item $\I\,:\,[\L^{p}_{\mathfrak{n},\sigma}(\Omega),\W^{2,p}_{0,\sigma}(\Omega)]_{s/2} \longrightarrow [\L^{p}_{\mathfrak{n},\sigma}(\Omega),\D_p(\AA_{\mathcal{D}})]_{s/2}$ is well-defined and bounded for all $s\in[0,2]$;
    \item $\I\,:\,[\L^{p}_{\mathfrak{n},\sigma}(\Omega),\W^{2,p}_{0,\sigma}(\Omega)]_{s/2} \longrightarrow [\L^{p}_{\mathfrak{n},\sigma}(\Omega),\D_p(\AA_{\mathcal{D}})]_{s/2}$ is an isomorphism for all $s\in[0,1]$.
\end{itemize}
By Sneiberg's stability theorem, see \cite[Thm~1.3.24]{EgertPhDThesis2015}, there exists $\delta_{{}_{\Omega},p}>0$, such that
\begin{align*}
    \I\,:\,[\L^{p}_{\mathfrak{n},\sigma}(\Omega),\W^{2,p}_{0,\sigma}(\Omega)]_{s/2} \longrightarrow [\L^{p}_{\mathfrak{n},\sigma}(\Omega),\D_p(\AA_{\mathcal{D}})]_{s/2}
\end{align*}
remains an isomorphism for all $s\in[0,1+\delta_{{}_{\Omega},p})$. Thus, by \eqref{eq:CompInterpBIpStokes} combined with the fact that  $[\L^{p}_{\mathfrak{n},\sigma},\W^{2,p}_{0,\sigma}]_{s/2}= \H^{s,p}_{0,\sigma}$, we deduce the isomorphism
\begin{align*}
    \I\,:\,\D_p(\AA_{\mathcal{D}}^{s/2}) \longrightarrow \H^{s,p}_{0,\sigma}(\Omega),\,\quad s\in[1,1+\delta_{{}_\Omega,p}).
\end{align*}
The latter translates as \eqref{eq:IsomFracPowerStokes} since $\mathbb{A}_{\mathcal{D}}$ is invertible. This ends the proof.
\end{proof}

\begin{corollary}\label{cor:SrtPropStokesW-1p} Let $\Omega$ be a bounded Lipschitz domain of $\RR^n$. For all $p\in (1,\infty)$ that satisfies \eqref{eq:CondtnLpStokesLip}, the linear operator
\begin{align*}
     \AA_{\mathcal{D}}^{-1/2}\PP_{\Omega}\div\,:\,\L^p(\Omega,\CC^{n\times n})\longrightarrow\L^p_{\mathfrak{n},\sigma}(\Omega)
\end{align*}
is well-defined and bounded.    
\end{corollary}

\begin{proof}By Theorem~\ref{thm:StokesLipSobolev} and Poincaré's inequality, $\nabla\AA_{\mathcal{D}}^{-1/2}$, as a linear operator from $\L^{p'}_{\mathfrak{n},\sigma}(\Omega)$ to $\L^{p'}(\Omega,\CC^{n\times n})$, is well-defined and bounded. Writing $=\nabla\AA_{\mathcal{D}}^{-1/2}=\nabla \iota\AA_{\mathcal{D}}^{-1/2}$ where $\iota'=\PP_{\Omega}$, the result holds by duality. Indeed, writing for all $\uu\in\Ccinftydiv(\Omega)$, all $\mathbf{F}\in\Ccinfty(\Omega,\CC^{n\times n})$
\begin{align*}
    \langle \mathbf{F}, \nabla\AA^{-1/2}_\mathcal{D}\uu\rangle_{\Omega} = \langle \mathbf{F}, \nabla \iota\AA^{-1/2}_\mathcal{D}\uu\rangle_{\Omega} = \langle -\div( \mathbf{F}), \iota\AA^{-1/2}_\mathcal{D}\uu\rangle_{\Omega} = \langle -\AA^{-1/2}_\mathcal{D}\PP_\Omega\div( \mathbf{F}), \uu\rangle_{\Omega}, 
\end{align*}
one can then extends $\AA_{\mathcal{D}}^{-1/2}\PP_{\Omega}\div$ to the whole $\L^p(\Omega,\CC^{n\times n})$, by strong density of $\Ccinfty(\Omega,\CC^{n\times n})$ in $\L^p(\Omega,\CC^{n\times n})$.
\end{proof}

Here is the Besov spaces counterpart of the result Theorem~\ref{thm:StokesLipSobolev}.

\begin{proposition}\label{prop:StokesBesovLip}Let $\Omega$ be a bounded Lipschitz domain of $\RR^n$. Let $\varepsilon_{{}_\Omega},\delta_{{}_\Omega,p}>0$ be given by Theorem~\ref{thm:StokesLipSobolev}. Let $p\in(1,\infty)$, $s\in(-1+1/p,1/p)$ such that \eqref{eq:CondtnLpStokesLip} is satisfied. Let $q\in[1,\infty]$. Then the following assertions hold
\begin{enumerate}
    \item The Stokes operator is an invertible sectorial operator of angle $0$ in $\B^{s,\sigma}_{p,q,0}(\Omega)$, and it is densely defined whenever $q<\infty$;
    \item For all $\theta\in(0,1)$, such that $s+2\theta< 1+{\delta_{{}_\Omega,p}}$, $s+2\theta\neq1/p$, for all $r\in[1,\infty]$, it holds
    \begin{align*}
        (\B^{s,\sigma}_{p,q,0}(\Omega),\D^{s}_{p,q}(\AA_\mathcal{D}))_{\theta,r} = \B^{s+2\theta,\sigma}_{p,r,0}(\Omega)
    \end{align*}
    with equivalence of norms;
    \item Setting $I_p^{\sigma}:=(-1+1/p,1+\delta_{{}_\Omega,p})\setminus\{1/p\}$, for all $\tau,\beta\in I_p^{\sigma}$ such that $\tau+\beta \in I_p^{\sigma}$,
    \begin{align*}
        \AA_\mathcal{D}^{\beta/2}\,:\,\B^{\tau+\beta,\sigma}_{p,q,0}(\Omega)\longrightarrow\B^{\tau,\sigma}_{p,q,0}(\Omega)
    \end{align*}
    is an isomorphism. In particular, one has the isomorphism
    \begin{align}
         \AA_\mathcal{D}^{1/2}\,:\,\B^{1,\sigma}_{p,q,0}(\Omega)\longrightarrow\B^{0,\sigma}_{p,q,0}(\Omega).\label{eq:SQRTStokesBesovLip}
    \end{align}
\end{enumerate}
\end{proposition}

\begin{proof} Everything is a direct consequence of  of  real interpolation applied to Theorem~\ref{thm:StokesLipSobolev}.
\end{proof}

The square-root property \eqref{eq:SQRTStokesBesovLip}, allowed by the existence of $\delta_{{}_\Omega,p}>0$ in Theorem~\ref{thm:StokesLipSobolev}, is the corner stone for our analysis of the Stokes system in Besov spaces over bounded Lipschitz domains, Proposition~\ref{Prop:SolutionOperatorBoussinesqBesov}.

\subsection{The parabolic maximal regularity estimates.}\label{Sec:MaxGeg}

We now, specify the regularity-in-time results we are going to use for the linear Neumann Heat and Stokes evolutionary equations, the latter being subject to no-slip boundary condition.

These evolution problems, are respectively represented by the set of equations
\begin{equation}\tag{NH}\label{eq:NeumanHeatEq}
    \left\{\begin{aligned}
        \partial_t \theta - \kappa\Delta \theta  & = f & \quad & \text{ in $(0,T)\times\Omega$}, \\
        \partial_\mathbf{n}\theta & =0&  &  \text{ on  $(0,T)\times\partial \Omega$}, \\
    \theta (0) & = \theta_0, &  &
    \end{aligned}\right.
\end{equation}
\noindent and
\begin{equation}\tag{SD}\label{eq:StokesDirichlet}
    \left\{\begin{aligned}
        \partial_t \uu - \nu\Delta \uu + \nabla \mathfrak{p}  & =  \ff & \quad & \text{in $(0,T)\times\Omega$}, \\
        \div \uu & =  0 & \quad & \text{in $(0,T)\times\Omega$}, \\
        \uu & =0&  &  \text{on $(0,T)\times\partial \Omega$}, \\
        \uu (0) & = \uu_0. &  &
    \end{aligned}\right.
\end{equation}
are the reference systems that are going to be investigated. We recall for the reader that the norm estimates induced by the following Propositions~\ref{thm:LqMaxRegNeumannBesov}~and~\ref{thm:LqMaxRegStokesBesov}, are actually estimates on the solution operator
\begin{align*}
    (x_0,\,g)\longmapsto\left[e^{-tA}x_0+\int_{0}^{t}e^{-(t-s)A}\,g(s)\,\d s\right] 
\end{align*}
where $A$ is either the Neumann Laplacian $-\kappa\Delta_\mathcal{N}$ or the Stokes operator $\nu\AA_\mathcal{D}$.

\begin{proposition}\label{thm:LqMaxRegNeumannBesov} Let  $\Omega\subset\RR^n$ be a bounded Lipschitz domain. Let $p\in(1,\infty)$, $r\in[1,\infty]$, $q\in(1,\infty)\cup\{r\}$, $s\in(-2+1/p,-1+1/p)$, such that
\begin{align*}
    p> \frac{3}{2-\varepsilon_{{}_{\Omega}}}\quad  \text{ and } \quad s+2<\frac{3}{p}+\varepsilon_{{}_{\Omega}},
\end{align*} 
where $\varepsilon_{{}_{\Omega}}$ is given by Proposition~\ref{prop:NeumannRegularity}.
\medbreak
Let $T\in(0,\infty)$. For all $f\in \L^q(0,T;\B^{s}_{p,r,0}(\Omega))$, all $\theta_0\in \B^{s+2-2/q}_{p,q,0}(\Omega)$, the problem \eqref{eq:NeumanHeatEq} admits a unique mild solution $\theta\in \C^0_{ub}([0,T];\B^{s+2-2/q}_{p,q,0}(\Omega))$ such that $\theta \in \L^q(0,T;\B^{s+2}_{p,r}(\Omega))$ with the estimates
\begin{align}\label{eq:maxRegBspq}
    \kappa^{1-1/q}\lVert \theta \rVert_{\L^\infty(0,T;\B^{s+2 -2/q}_{p,q}(\Omega))} + \kappa\lVert  \theta&\rVert_{\L^q(0,T;\B^{s+2}_{p,r}(\Omega))}\\
    &\lesssim_{p,q,s,\Omega} \lVert f\rVert_{\L^q(0,T;\B^{s}_{p,r}(\Omega))} + \kappa^{1-1/q}\lVert \theta_0\rVert_{\B^{ s+2-2/q}_{p,q}(\Omega)}\nonumber\\&\qquad\qquad\qquad\qquad\qquad\qquad\qquad+  \kappa T^{1/q}|\langle \theta_0, \mathbf{1}\rangle_{\Omega}|\nonumber\text{.}
\end{align}
Furthermore, if  $f\in\L^q(\RR_+,\B^{s}_{p,r,\circ}(\Omega))$ and $\theta_0\in \B^{s+2-2/q}_{p,q,\circ}(\Omega)$ then $\theta\in \C^0_{0}(\overline{\RR_+};\B^{s+2-2/q}_{p,q,\circ}(\Omega))\cap\L^q(\RR_+;\B^{s+2}_{p,r,\circ}(\Omega))$ and \eqref{eq:maxRegBspq} remains valid with a uniform constant with respect to $T=\infty$. Furthermore, there exists $c>0$ such that
\begin{align*}
    \kappa^{1-1/q}\lVert e^{ct} \theta \rVert_{\L^\infty(\RR_+;\B^{s+2 -2/q}_{p,q}(\Omega))} + \kappa\lVert e^{ct} \nabla^2\theta&\rVert_{\L^q(\RR_+;\B^{s}_{p,r}(\Omega))} \\&\lesssim_{p,q,s,\Omega} \lVert  f\rVert_{\L^q(\RR_+;\B^{s}_{p,r}(\Omega))} + \kappa^{1-1/q}\lVert \theta_0\rVert_{\B^{ s+2-2/q}_{p,q}(\Omega)}\text{.}
\end{align*}
\end{proposition}

\begin{proof}For the case $T\in(0,\infty]$ and $\theta_0\in\B^{s+2-2/q}_{p,q,\circ}(\Omega)$, thanks to our assumptions on $p$ and $s$, by Proposition~\ref{prop:NeumannBesov}, one can find $-2+1/p<\tilde{s}<s$ such that
\begin{align*}
    \D_{p,1}^{\tilde{s}}(-\kappa\Delta_{\mathcal{N}}) = \B^{\tilde{s}+2}_{p,1,\circ}(\Omega)
\end{align*}
on $\X = \B^{\tilde{s}}_{p,1,\circ}(\Omega)$,  and $\eta_0\in(0,1)$ close to $0$ such that
\begin{align*}
    \B^s_{p,q,\circ}(\Omega)& = (\B^{\tilde{s}+2}_{p,1,\circ}(\Omega),\D_{p,1}^{\tilde{s}}(-\kappa\Delta_{\mathcal{N}}))_{\eta_0,q},\quad\\&\text{ and }\quad\B^{s+2-2/q}_{p,q,\circ}(\Omega) = (\B^{\tilde{s}+2}_{p,1,\circ}(\Omega),\D_{p,1}^{\tilde{s}}(-\kappa\Delta_{\mathcal{N}}))_{1+\eta_0-1/q,q},
\end{align*}
Therefore, the result is a direct consequence of the Da Prato--Grisvard Theorem, Theorem~\ref{thm:DaPratoGrisvard}, when $\theta_0\in\B^{s+2-2/q}_{p,q,\circ}(\Omega)$. Now, considering the case $T\in(0,\infty)$ and $\theta_0\in\B^{s+2-2/q}_{p,q,0}(\Omega)$, writing $\theta_0 = \P_{\circ}\theta_0 + (\theta_0)_{\Omega}$, by linearity, the complete estimate \eqref{eq:maxRegBspq} holds.
\end{proof}

\begin{proposition}\label{thm:LqMaxRegStokesBesov} Let  $\Omega\subset\RR^n$ be a bounded Lipschitz domain. Let $p\in(1,\infty)$, $r\in[1,\infty]$, $q\in(1,\infty)\cup\{r\}$ and $s\in(-1+1/p,1/p)$ such that $p$ satisfies \eqref{eq:CondtnLpStokesLip} as well as
\begin{align*}
    s+1-\frac{2}{q}<\delta_{{}_{\Omega},p}\quad\text{ and }\quad \frac{2}{q}+\frac{1}{p}\neq s+2,
\end{align*}
where $\delta_{{}_\Omega,p}>0$ is given by Theorem~\ref{thm:StokesLipSobolev}.
\medbreak

Let $T\in(0,\infty]$. For all $\ff\in \L^q(0,T;\B^{s,\sigma}_{p,r,0}(\Omega))$, all $\uu_0\in \B^{s+2-2/q,\sigma}_{p,q,0}(\Omega)$, the problem \eqref{eq:StokesDirichlet} admits a unique mild solution $\uu\in \C^0_{ub}([0,T];\B^{s+2-2/q,\sigma}_{p,q,0}(\Omega))$ such that $\L^q(0,T;\D^{s}_{p,r}(\AA_\mathcal{D},\Omega))$ with the estimates
\begin{align}
    \nu^{1-1/q}\lVert \uu \rVert_{\L^\infty(0,T;\B^{ s+2 -2/q}_{p,q}(\Omega))} + \nu\lVert \mathbb{A}_\mathcal{D} u&\rVert_{\L^q(0,T;\B^{s}_{p,r}(\Omega))}\\&\lesssim_{p,q,s,\Omega} \lVert  \ff\rVert_{\L^q(0,T;\B^{s}_{p,r}(\Omega))} + \nu^{1-1/q}\lVert \uu_0\rVert_{\B^{s+2 -2/q}_{p,q}(\Omega)}\nonumber\text{.}
\end{align}
Furthermore, if $T=\infty$, there exists $c>0$ such that
\begin{align*}
    \nu^{1-1/q}\lVert e^{ct} \uu \rVert_{\L^\infty(\RR_+;\B^{s+2 -2/q}_{p,q}(\Omega))} + \nu\lVert e^{ct}\AA_\mathcal{D}\uu&\rVert_{\L^q(\RR_+;\B^{s}_{p,r}(\Omega))}\\ &\lesssim_{p,q,s,\Omega} \lVert  \ff\rVert_{\L^q(\RR_+;\B^{s}_{p,r}(\Omega))} + \nu^{1-1/q}\lVert \uu_0\rVert_{\B^{s+2 -2/q}_{p,q}(\Omega)}\text{.}
\end{align*}
\end{proposition}

\begin{proof} Thanks to our assumptions on $p$ and $s$, by Proposition~\ref{prop:StokesBesovLip}, one can find $-1+1/p<\tilde{s}<s$ and $\eta_0\in(0,1)$ close to $0$ such that
for  $\X = \B^{\tilde{s},\sigma}_{p,1,0}(\Omega)$, one has
\begin{align*}
    \B^{s,\sigma}_{p,q,0}(\Omega)& = (\B^{\tilde{s},\sigma}_{p,1,0}(\Omega),\D_{p,1}^{\tilde{s}}(\AA_\mathcal{D}))_{\eta_0,q},\quad\\&\text{ and }\quad\B^{s+2-2/q,\sigma}_{p,q,0}(\Omega) = (\B^{\tilde{s},\sigma}_{p,1,0}(\Omega),\D_{p,1}^{\tilde{s}}(\AA_\mathcal{D}))_{1+\eta_0-1/q,q}. 
\end{align*}
Therefore, the result is a direct consequence of the Da Prato--Grisvard Theorem, Theorem~\ref{thm:DaPratoGrisvard}.
\end{proof}

\subsection{(Bi-)Linear estimates for the solutions operators of the Boussinesq system.}\label{Sec:BiLiN}

\subsection*{$\L^2$-in-time estimates}

\begin{notation}\label{def:LpSolvSpacesBSQ}For $p,r\in(1,\infty)$, $\Omega\subset \RR^3$ to be a bounded Lipschitz domain and $T\in(0,\infty]$, we set
\begin{itemize}
    \item The function space of velocity vector fields to be
\begin{align*}
    \mathrm{Z}_p((0,T)\times\Omega) :=  \L^\infty(0,T;\B^{0,\sigma}_{p,2,0}(\Omega))\cap\L^2(0,T;\B^{1,\sigma}_{p,\infty,0}(\Omega))
\end{align*}
endowed with norm
\begin{align*}
    \lVert \vv \rVert_{\mathrm{Z}_p((0,T)\times\Omega)}:=\lVert \vv \rVert_{\L^\infty(0,T;\B^{ 0}_{p,2}(\Omega))} + \nu^{1/2}\lVert  \nabla \vv\rVert_{\L^2(0,T;\B^{0}_{p,\infty}(\Omega))};
\end{align*}

\item The function space of temperatures to be
\begin{align*}
    \mathrm{K}_r((0,T)\times\Omega) := \L^\infty(0,T;\B^{-1}_{r,2,0}(\Omega))\cap\L^2(0,T;\B^{0}_{r,\infty}(\Omega))
\end{align*}
endowed with norm
\begin{align*}
    \lVert \Theta \rVert_{\mathrm{K}_r((0,T)\times\Omega)}:=\lVert \Theta \rVert_{\L^\infty(0,T;\B^{ -1}_{r,2}(\Omega))} + \kappa^{1/2}\lVert  \Theta\rVert_{\L^2(0,T;\B^{0}_{r,\infty}(\Omega))}.
\end{align*}
\end{itemize}
\end{notation}

\begin{proposition}\label{Prop:SolutionOperatorBoussinesqLp} Let $\Omega\subset\RR^3$ be  a bounded Lipschitz domain, and let $T\in(0,\infty]$. Then
\begin{enumerate}
    \item the linear operator $\L_\nu$, as given by \eqref{eq:LOpLin},
is well-defined and bounded from $\mathrm{K}_{3/2}((0,T)\times \Omega)$ to $\mathrm{Z}_{3}((0,T)\times\Omega)$ and its operator norm satisfies
\begin{align*}
    \lVert \L_{\nu}\rVert_{\mathrm{K}_{3/2}\rightarrow \mathrm{Z}_3} \lesssim_{\Omega} \frac{1}{\sqrt{\nu}\sqrt{\kappa}};
\end{align*}
    \item the bilinear operator $\B_{\nu}$, as given by \eqref{eq:BOpBiLin},
is well-defined and bounded from $\mathrm{Z}_{3}((0,T)\times\Omega)\times \mathrm{Z}_{3}((0,T)\times\Omega)$ to $\mathrm{Z}_{3}((0,T)\times\Omega)$ and its operator norm satisfies 
\begin{align*}
    \lVert \B_{\nu}\rVert_{\mathrm{Z}_{3}\times \mathrm{Z}_3\rightarrow\mathrm{Z}_3} \lesssim_{\Omega} \frac{1}{\nu};
\end{align*}
\item if additionally $\Omega$ has $\C^{1,\alpha}$ boundary with $\alpha>1/3$, the bilinear operator $\C_\kappa$, as given by \eqref{eq:COpBiLin},
is well-defined and bounded from $\mathrm{Z}_{3}((0,T)\times\Omega)\times \mathrm{K}_{3/2}((0,T)\times \Omega)$ to $\mathrm{K}_{3/2}((0,T)\times \Omega)$ and its operator norm satisfies the bound
\begin{align*}
    \lVert \C_{\kappa}\rVert_{\mathrm{Z}_{3}\times \mathrm{K}_{3/2}\rightarrow\mathrm{K}_{3/2}} \lesssim_{\Omega} \frac{1}{\nu^{1/4}\kappa^{3/4}};
\end{align*}
\end{enumerate}
\end{proposition}

\begin{proof}We perform the analysis for $T=\infty$. For $T$ to be finite, it admits the same proof. 

\textbf{Step 1:} By the $\L^2$-maximal regularity for the Stokes operator in $\AA_\mathcal{D}^{1/2}\B^{0,\sigma}_{3,\infty,0}(\Omega)$, Proposition~\ref{thm:LqMaxRegStokesBesov}, and by the boundedness of $\PP_{\Omega}$ in $\B^{0}_{3/2,\infty}(\Omega)^3$, Proposition~\ref{prop:HelmholtzDecomp}, it holds:
\begin{align*}
    \lVert \L_\nu ( \Theta) \rVert_{\mathrm{Z}_{3}(\RR_+\times\Omega)} &\lesssim_{\Omega} \frac{1}{\sqrt{\nu}} \lVert \AA_\mathcal{D}^{-1/2}\PP_{\Omega}(\Theta\mathfrak{e}_3)\rVert_{\L^2(\RR_+;\B^0_{3,\infty}(\Omega))}\\
    &\lesssim_{\Omega} \frac{1}{\sqrt{\nu}}  \lVert \Theta\rVert_{\L^2(\RR_+;\B^0_{3/2,\infty}(\Omega))}\\
    & \lesssim_{\Omega} \frac{1}{\sqrt{\nu}\sqrt{\kappa}} \lVert\Theta \rVert_{{\mathrm{K}_{3/2}}(\RR_+\times\Omega)}.
\end{align*}

\textbf{Step 2:} By the $\L^2$-maximal regularity for the Stokes operator in $\AA_\mathcal{D}^{1/2}\B^{0,\sigma}_{3,\infty,0}(\Omega)$, Proposition~\ref{thm:LqMaxRegStokesBesov}, the embedding $\L^3(\Omega)\hookrightarrow\B^{0}_{3,\infty}(\Omega)$ and Corollary~\ref{cor:SrtPropStokesW-1p}, it holds:
\begin{align}
    \lVert \B_\nu(\vv,\ww) \rVert_{\mathrm{Z}_{3}(\RR_+\times\Omega)} &\lesssim_{\Omega} \frac{1}{\sqrt{\nu}}  \lVert  \AA_\mathcal{D}^{-1/2}\PP_{\Omega}(\div (\vv\otimes\ww))\rVert_{\L^2(\RR_+;\L^3(\Omega))}\nonumber\\
    &\lesssim_{\Omega} \frac{1}{\sqrt{\nu}}  \lVert  \vv\otimes\ww\rVert_{\L^2(\RR_+;\L^3(\Omega))} \nonumber \\
    & \lesssim_{\Omega} \frac{1}{\sqrt{\nu}}   \lVert \vv \rVert_{\L^4(\RR_+;\L^{6}(\Omega))}\lVert \ww \rVert_{\L^4(\RR_+;\L^{6}(\Omega))}\nonumber\\
    & \lesssim_{\Omega} \frac{1}{\sqrt{\nu}}   \lVert \vv \rVert_{\L^4(\RR_+;\H^{1/2,3}(\Omega))}\lVert \ww \rVert_{\L^4(\RR_+;\H^{1/2,3}(\Omega))}\label{eq:ProofNSEpartL4W1/23Est}
\end{align}
We did took advantage of the time-space H\"{o}lder inequality $\L^4(\L^6)\cdot\L^4(\L^6) \hookrightarrow\L^2(\L^3)$, and of the Sobolev embedding $\H^{1/2,3}(\Omega)\hookrightarrow\L^{6}(\Omega)$. By taking advantage of the interpolation inequality
\begin{align*}
     \lVert u \rVert_{\L^4(\RR_+;\H^{1/2,3}(\Omega))}\lesssim_{\Omega} \Big\lVert \lVert u\rVert_{\B^0_{3,2}(\Omega)}^\frac{1}{2} \lVert \nabla u\rVert_{\B^0_{3,\infty}(\Omega)}^\frac{1}{2} \Big\rVert_{\L^4(\RR_+)}\lesssim_{\Omega}  \lVert u\rVert_{\L^\infty(\RR_+;\B^0_{3,2}(\Omega))}^\frac{1}{2} \lVert \nabla u\rVert_{\L^2(\RR_+;\B^0_{3,\infty}(\Omega))}^\frac{1}{2}, 
\end{align*}
the estimate \eqref{eq:ProofNSEpartL4W1/23Est} becomes
\begin{align}
    \lVert &\B_\nu(\vv,\ww) \rVert_{\mathrm{Z}_{3}(\RR_+\times\Omega)}\nonumber\\
    & \lesssim_{\Omega} \frac{1}{\sqrt{\nu}}  \lVert \vv\rVert_{\L^\infty(\RR_+;\B^{0}_{3,2}(\Omega))}^\frac{1}{2} \lVert \nabla \vv\rVert_{\L^2(\RR_+;\B^0_{3,\infty}(\Omega))}^\frac{1}{2}\lVert \ww\rVert_{\L^\infty(\RR_+;\B^{0}_{3,2}(\Omega))}^\frac{1}{2} \lVert \nabla \ww\rVert_{\L^2(\RR_+;\B^0_{3,\infty}(\Omega))}^\frac{1}{2}.\label{eq:ProofBSQBiLinNSEEstLp}
\end{align}
Finally, by Young's inequality $ab \leqslant \frac{a^2+b^2}{2}$, yields
\begin{align*}
    \lVert \B_\nu(\vv,\ww) \rVert_{\mathrm{Z}_{3}(\RR_+\times\Omega)}&\lesssim_{\Omega} \frac{1}{{\nu}}  \lVert \vv \rVert_{\mathrm{Z}_{3}(\RR_+\times\Omega)}\lVert \ww \rVert_{\mathrm{Z}_{3}(\RR_+\times\Omega)}.
\end{align*}

\noindent\textbf{Step 3:} By the $\L^2$-maximal regularity for the Neumann Laplacian in $(-\Delta_\mathcal{N})^{1/2}\B^{-1}_{3/2,\infty,\circ}(\Omega)$, Proposition~\ref{thm:LqMaxRegNeumannBesov}, and the Sobolev embedding $\B^{-1/2}_{6/5,\infty}(\Omega)\hookrightarrow \B^{-1}_{3/2,\infty}(\Omega)$, it holds:
\begin{align*}
    \lVert \C_\kappa(\vv,\Theta) \rVert_{\mathrm{K}_{3/2}(\RR_+\times\Omega)} &\lesssim_{\Omega} \frac{1}{\sqrt{\kappa}}  \lVert (-\Delta_\mathcal{N})^{-1/2} \div (\vv\Theta)\rVert_{\L^2(\RR_+;\B^{-1}_{3/2,\infty}(\Omega))} \\ 
    &\lesssim_{\Omega}\frac{1}{\sqrt{\kappa}} \lVert (-\Delta_\mathcal{N})^{-1/2} \div (\vv\Theta)\rVert_{\L^2(\RR_+;\B^{-1/2}_{6/5,\infty}(\Omega))}\\ 
    &\lesssim_{\Omega} \frac{1}{\sqrt{\kappa}}\lVert \vv\Theta\rVert_{\L^2(\RR_+;\B^{-1/2}_{6/5,\infty}(\Omega))}.
\end{align*}
The last line was obtained by means of \eqref{NeumannRieszTransf}\footnote{This is the only place where regularity beyond than Lipschitz for the boundary is necessary. If instead, one consumes more and more negative regularity in the Sobolev embedding, one gets closer and closer to a pathological $\L^{1}$-type space. Getting closer could possibly allow to reduce the need of boundary regularity up to being arbitrarily close to $\C^{1}$--boundary regularity, while having some space $\B^{-\delta}_{1+\varepsilon,\infty}$ for the product to be estimated. Yet, the integrability indices decomposition to identify in the latter, in order to make the product rules to work, becomes less identifiable.}. Therefore, by the product rule from Corollary~\ref{prop:ProductRullBesovDim3-2}, we deduce
\begin{align}
    \lVert \C_\kappa(\vv,\Theta) \rVert_{\mathrm{K}_{3/2}(\RR_+\times\Omega)} &\lesssim_{\Omega} \frac{1}{\sqrt{\kappa}}\lVert \vv\Theta\rVert_{\L^2(\RR_+;\B^{-1/2}_{6/5,\infty}(\Omega))}\nonumber\\
    &\lesssim_{\Omega}\frac{1}{\sqrt{\kappa}} \Big\lVert \lVert\vv\rVert_{\B^{1/2}_{3,1}(\Omega)}\lVert\Theta\rVert_{\B^{-1/2}_{3/2,1}(\Omega)}\Big\rVert_{\L^2(\RR_+)}.\label{eq:Est:C(v,O)L2Lp}
\end{align}
We recall that we have the interpolation inequalities
\begin{align*}
    \lVert \vv\rVert_{\B^{1/2}_{3,1}(\Omega)} \lesssim_{\Omega}\lVert \vv\rVert_{\B^{0}_{3,2}(\Omega)}^\frac{1}{2} \lVert \vv\rVert_{\B^{1}_{3,\infty}(\Omega)}^\frac{1}{2} ,\qquad \text{ and }\qquad\lVert \Theta\rVert_{\B^{-1/2}_{3/2,1}(\Omega)} \lesssim_{\Omega}\lVert \Theta\rVert_{\B^{-1}_{3/2,2}(\Omega)}^\frac{1}{2} \lVert\Theta\rVert_{\B^0_{3/2,\infty}(\Omega)}^\frac{1}{2}.
\end{align*}
Therefore, by H\"{o}lder's inequality, \eqref{eq:Est:C(v,O)L2Lp} becomes
\begin{align}
    \lVert \C_\kappa(\vv&,\Theta) \rVert_{\mathrm{K}_{3/2}(\RR_+\times\Omega)}\nonumber\\
    &\lesssim_{\Omega} \frac{1}{\sqrt{\kappa}}\Big\lVert \lVert \vv\rVert_{\B^{0}_{3,2}(\Omega)}^\frac{1}{2} \lVert \vv\rVert_{\B^{1}_{3,\infty}(\Omega)}^\frac{1}{2}\Big\rVert_{\L^4(\RR_+)}\Big\lVert \lVert \Theta\rVert_{\B^{-1}_{3/2,2}(\Omega)}^\frac{1}{2} \lVert\Theta\rVert_{\B^0_{3/2,\infty}(\Omega)}^\frac{1}{2}\Big\rVert_{\L^4(\RR_+)}\nonumber\\
    &\lesssim_{\Omega} \frac{1}{\sqrt{\kappa}} \lVert \vv\rVert_{\L^{\infty}(\RR_+;\B^{0}_{3,2}(\Omega))}^\frac{1}{2} \lVert \vv\rVert_{\L^{2}(\RR_+;\B^{1}_{3,\infty}(\Omega))}^\frac{1}{2} \lVert \Theta\rVert_{\L^{\infty}(\RR_+;\B^{-1}_{3/2,2}(\Omega))}^\frac{1}{2} \lVert\Theta\rVert_{\L^{2}(\RR_+;\B^0_{3/2,\infty}(\Omega))}^\frac{1}{2}\label{eq:ProofBSQBiLinTempEstLp}.
\end{align}
By virtue of Young's inequality,
\begin{align*}
    \lVert \C_\kappa(\vv,\Theta) \rVert_{\mathrm{K}_{3/2}(\RR_+\times\Omega)} \lesssim_{\Omega} \frac{1}{\nu^{1/4}\kappa^{3/4}}\lVert \Theta \rVert_{\mathrm{K}_{3/2}(\RR_+\times\Omega)}\lVert \vv \rVert_{\mathrm{Z}_{3}(\RR_+\times\Omega)}.
\end{align*}
This was the remaining claimed estimate, so that the proofs ends.
\end{proof}

\subsection*{$\L^1$-in-time estimates}

\begin{notation}\label{def:BesovSolvSpacesBSQ}For $3/2\leqslant r<2<p\leqslant 3$, $\Omega\subset \RR^3$ to be a bounded Lipschitz domain and $T\in(0,\infty]$, we set
\begin{itemize}
    \item The function space of velocity vector fields to be
\begin{align*}
   \mathrm{U}_{p}((0,T)\times\Omega) := \L^\infty(0,T;\B^{3/p-1,\sigma}_{p,1,0}(\Omega))\cap \L^1(0,T;\D_{p,1}^{3/p-1}(\AA_\mathcal{D},\Omega)).
\end{align*}
endowed with norm
\begin{align*}
    \lVert \vv \rVert_{\mathrm{U}_p((0,T)\times\Omega)}:=\lVert \vv \rVert_{\L^\infty(0,T;\B^{ 0}_{p,1}(\Omega))} + \nu\lVert  \mathbb{A}_{\mathcal{D}}\vv\rVert_{\L^1(0,T;\B^{ 0}_{p,1}(\Omega))};
\end{align*}

\item The function space of temperatures to be
\begin{align*}
     \T_{r}((0,T)\times\Omega) &:=\L^\infty(0,T;\B^{3/r-3}_{r,1,0}(\Omega))\cap\L^1(0,T;\B^{3/r-1}_{r,1}(\Omega)).
\end{align*}
endowed with norm
\begin{align*}
    \lVert \Theta \rVert_{\mathrm{K}_r((0,T)\times\Omega)}:=\lVert \Theta \rVert_{\L^\infty(0,T;\B^{ 3/r -3}_{r,1}(\Omega))} + \kappa\lVert  \Theta\rVert_{\L^1(0,T;\B^{3/r-1}_{r,1}(\Omega))}.
\end{align*}
\end{itemize}
\end{notation}

\begin{proposition}\label{Prop:SolutionOperatorBoussinesqBesov} Let $\Omega\subset\RR^3$ be  a bounded Lipschitz domain, and let $T\in(0,\infty]$. Then
\begin{enumerate}
    \item the linear operator $\L_\nu$, as given by \eqref{eq:LOpLin},
is well-defined and bounded from $\mathrm{T}_{3/2}((0,T)\times \Omega)$ to $\mathrm{U}_{3}((0,T)\times\Omega)$ and its operator norm satisfies
\begin{align*}
    \lVert \L_{\nu}\rVert_{\mathrm{T}_{3/2}\rightarrow \mathrm{U}_3} \lesssim_{\Omega} \frac{1}{\kappa};
\end{align*}
    \item the bilinear operator $\B_{\nu}$, as given by \eqref{eq:BOpBiLin},
is well-defined and bounded from $\mathrm{U}_{3}((0,T)\times\Omega)\times \mathrm{U}_{3}((0,T)\times\Omega)$ to $\mathrm{U}_{3}((0,T)\times\Omega)$ and its operator norm satisfies 
\begin{align*}
    \lVert \B_{\nu}\rVert_{\mathrm{U}_{3}\times \mathrm{U}_3\rightarrow\mathrm{U}_3} \lesssim_{\Omega} \frac{1}{\nu};
\end{align*}
\item  the bilinear operator $\C_\kappa$, as given by \eqref{eq:COpBiLin},
is well-defined and bounded from $\mathrm{U}_{3}((0,T)\times\Omega)\times \mathrm{T}_{3/2}((0,T)\times \Omega)$ to $\mathrm{T}_{3/2}((0,T)\times \Omega)$ and its operator norm satisfies the bound
\begin{align*}
    \lVert \C_{\kappa}\rVert_{\mathrm{U}_{3}\times \mathrm{T}_{3/2}\rightarrow\mathrm{T}_{3/2}} \lesssim_{\Omega} \frac{1}{\sqrt{\kappa}\sqrt{\nu}};
\end{align*}
\end{enumerate}
\end{proposition}

\begin{proof} We perform the analysis for $T=+\infty$, the case $T$ to be finite admits the same proof.

\textbf{Step 1:} By the Da Prato-Grisvard $\L^1$-maximal regularity for the Stokes operator Proposition~\ref{thm:LqMaxRegNeumannBesov}, and by the boundedness of $\PP_{\Omega}$ in $\B^{0}_{3,1}(\Omega)^3$, Proposition~\ref{prop:HelmholtzDecomp}, it holds:
\begin{align*}
    \lVert \L_{\nu} ( \Theta) \rVert_{\mathrm{U}_{3}(\RR_+\times\Omega)} &\lesssim_{\Omega} \lVert \PP_{\Omega}(\Theta\mathfrak{e}_3)\rVert_{\L^1(\RR_+;\B^{0}_{3,1}(\Omega))}\\ &\lesssim_{\Omega} \lVert \Theta\rVert_{\L^1(\RR_+;\B^{0}_{3,1}(\Omega))} \lesssim_{\Omega} \lVert \Theta\rVert_{\L^1(\RR_+;\B^{1}_{3/2,1}(\Omega))} \\
    & \lesssim_{\Omega} \frac{1}{\kappa} \lVert\Theta \rVert_{\T_{{3/2}}(\RR_+\times\Omega)}
\end{align*}
\textbf{Step 2:} By the Da Prato-Grisvard $\L^1$-maximal regularity for the Stokes operator Proposition~\ref{thm:LqMaxRegStokesBesov}, and by the boundedness of $\PP_{\Omega}$ in $\B^{0}_{3,1}(\Omega)^3$, Proposition~\ref{prop:HelmholtzDecomp}, it holds:
\begin{align*}
    \lVert \B_\nu(\vv,\ww) \rVert_{\mathrm{U}_{3}(\RR_+\times\Omega)} &\lesssim_{\Omega} \lVert \PP_{\Omega}(\div (\vv\otimes\ww))\rVert_{\L^1(\RR_+;\B^{0}_{3,1}(\Omega))}\\ &\lesssim_{\Omega} \lVert \vv\otimes\ww\rVert_{\L^1(\RR_+;\B^{1}_{3,1}(\Omega))}  \\
    & \lesssim_{\Omega} \lVert \vv \rVert_{\L^2(\RR_+;\B^{1}_{3,1}(\Omega))}\lVert \ww \rVert_{\L^2(\RR_+;\B^{1}_{3,1}(\Omega))}
\end{align*}
due to the fact that $\B^{1}_{3,1}(\Omega)$ is an algebra, since we are in dimension $3$. Finally, by the square-root property for the Stokes--Dirichlet operator, \eqref{eq:SQRTStokesBesovLip}, it holds
\begin{align}\label{eq:ProofBesovEstiBilLinStokes}
    \lVert \B_\nu(\vv,\ww) \rVert_{\mathrm{U}_{3}(\RR_+\times\Omega)} &\lesssim_{\Omega} \lVert \vv \rVert_{\L^2(\RR_+;\B^{1}_{3,1}(\Omega))}\lVert \ww \rVert_{\L^2(\RR_+;\B^{1}_{3,1}(\Omega))}\nonumber\\
    &\lesssim_{\Omega} \lVert \AA_{\mathcal{D}}^{1/2}\vv \rVert_{\L^2(\RR_+;\B^{0}_{3,1}(\Omega))}\lVert \AA_{\mathcal{D}}^{1/2}\ww \rVert_{\L^2(\RR_+;\B^{0}_{3,1}(\Omega))}.
\end{align}
We recall that we have interpolation inequality:
\begin{align*}
    \lVert A^{1/2}u \rVert_{\L^2(\RR_+;\X)}\lesssim_{A,\X} \bigg(\int_{0}^{+\infty} \lVert u(t)\rVert_{\X}\lVert A u(t)\rVert_{\X} \d t\bigg)^{1/2} \lesssim_{A,\X}\lVert u \rVert_{\L^\infty(\RR_+;\X)}^\frac{1}{2}\lVert A u \rVert_{\L^1(\RR_+;\X)}^\frac{1}{2},
\end{align*}
valid for any sectorial operator $A$ in a Banach space $\X$, \cite[Prop.~6.6.4]{bookHaase2006}. In particular, one obtains
\begin{align}\label{eq:ProofInterpLinftyL1(A)}
    \lVert \AA_{\mathcal{D}}^{1/2} u \rVert_{\L^2(\RR_+;\B^{0}_{3,1}(\Omega))} \lesssim_{\Omega}\lVert u \rVert_{\L^\infty(\RR_+;\B^{0}_{3,1}(\Omega))}^\frac{1}{2}\lVert  \AA_{\mathcal{D}} u \rVert_{\L^1(\RR_+;\B^{0}_{3,1}(\Omega))}^\frac{1}{2},\quad u\in\{\vv,\ww\}.
\end{align}
which, plugged in \eqref{eq:ProofBesovEstiBilLinStokes}, yields
\begin{align}
    \lVert &\B_\nu(\vv,\ww) \rVert_{\mathrm{U}_{3}(\RR_+\times\Omega)}\nonumber\\
    &\lesssim_{\Omega} \lVert \vv \rVert_{\L^\infty(\RR_+;\B^{0}_{3,1}(\Omega))}^\frac{1}{2}\lVert  \AA_{\mathcal{D}} \vv \rVert_{\L^1(\RR_+;\B^{0}_{3,1}(\Omega))}^\frac{1}{2}\lVert \ww \rVert_{\L^\infty(\RR_+;\B^{0}_{3,1}(\Omega))}^\frac{1}{2}\lVert  \AA_{\mathcal{D}} \ww \rVert_{\L^1(\RR_+;\B^{0}_{3,1}(\Omega))}^\frac{1}{2}.\label{eq:ProofBiLinestStokes}
\end{align}

\noindent\textbf{Step 3:} By the Da Prato-Grisvard $\L^1$-maximal regularity for the Neumann Laplacian, Proposition~\ref{thm:LqMaxRegNeumannBesov}, it holds:
\begin{align*}
    \lVert \C_\kappa(\vv,\Theta) \rVert_{\mathrm{T}_{3/2}(\RR_+\times\Omega)} &\lesssim_{\Omega} \lVert \div (\vv\Theta)\rVert_{\L^1(\RR_+;\B^{-1}_{3/2,1}(\Omega))}  \lesssim_{\Omega} \lVert \vv\Theta\rVert_{\L^1(\RR_+;\B^{0}_{3/2,1}(\Omega))}.
\end{align*}
Applying the product estimate Corollary~\ref{cor:ProductRullBesovDim3-1}, we deduce
\begin{align*}
    \lVert \C_\kappa(\vv,\Theta) \rVert_{\mathrm{T}_{3/2}(\RR_+\times\Omega)} &\lesssim_{\Omega} \lVert \vv\Theta\rVert_{\L^1(\RR_+;\B^{0}_{3/2,1}(\Omega))}\\
    &\lesssim_{\Omega} \lVert \vv\rVert_{\L^2(\RR_+;\B^{1}_{3,1}(\Omega))}\lVert \Theta\rVert_{\L^2(\RR_+;\B^{0}_{3/2,1}(\Omega))}.
\end{align*}
However, taking advantage of \eqref{eq:ProofInterpLinftyL1(A)} and of the following interpolation inequality
\begin{align*}
    \lVert \Theta\rVert_{\L^2(\RR_+;\B^{0}_{3/2,1}(\Omega))} &\lesssim_{\Omega} \Big\lVert \lVert \Theta\rVert_{\B^{-1}_{3/2,1}(\Omega)}^\frac{1}{2}\lVert \Theta\rVert_{\B^{1}_{3/2,1}(\Omega)}^\frac{1}{2}\Big\rVert_{\L^2(\RR_+)}\\&\lesssim_{\Omega} \lVert \Theta\rVert_{\L^\infty(\RR_+;\B^{-1}_{3/2,1}(\Omega))}^\frac{1}{2}\lVert \Theta\rVert_{\L^1(\RR_+;\B^{1}_{3/2,1}(\Omega))}^\frac{1}{2},
\end{align*}
one obtains
\begin{align}
   \lVert &\C_\kappa(\vv,\Theta) \rVert_{\mathrm{T}_{3/2}(\RR_+\times\Omega)} \nonumber\\ & \lesssim_{\Omega}\lVert \vv \rVert_{\L^\infty(\RR_+;\B^{0}_{3,1}(\Omega))}^\frac{1}{2}\lVert  \AA_{\mathcal{D}} \vv \rVert_{\L^1(\RR_+;\B^{0}_{3,1}(\Omega))}^\frac{1}{2}\lVert \Theta\rVert_{\L^\infty(\RR_+;\B^{-1}_{3/2,1}(\Omega))}^\frac{1}{2}\lVert \Theta\rVert_{\L^1(\RR_+;\B^{1}_{3/2,1}(\Omega))}^\frac{1}{2}.\label{eq:ProofBiLinestBoussinesTransp}
\end{align}
Young's inequality $ab\leqslant\frac{a^2+b^2}{2}$ allows to conclude the proof with
\begin{align*}
    \lVert \C_\kappa(\vv,\Theta) \rVert_{\mathrm{T}_{3/2}(\RR_+\times\Omega)} \lesssim_{\Omega} \frac{1}{\sqrt{\kappa}\sqrt{\nu}}\lVert \vv \rVert_{\mathrm{U}_{3}(\RR_+\times\Omega)}\lVert \Theta \rVert_{\mathrm{T}_{3/2}(\RR_+\times\Omega)}.
\end{align*}
This last estimate ends the proof.
\end{proof}

%----------------------------------------------------
%-------------- Free Divergence Spaces --------------
%----------------------------------------------------
\section{Well-posedness and asymptotics for the Boussinesq system}\label{Sec:WPLebesgueSobolev}

From now on and throughout the remainder of this work, all function spaces taking values in a finite-dimensional vector space are assumed to be real-valued. We introduce two notions of solutions, depending on the underlying functional framework.

\begin{definition}We say that $(\uu,\theta)\in\L^1_\textrm{loc}([0,T);\L^{1}(\Omega,\RR^3)\times\B^0_{1,\infty}(\Omega))$ is a solution to the Boussinesq system \eqref{eq:BoussinesqSystemIntro} on $(0,T)$ \textit{in Lebesgue-Sobolev spaces}  if the following is satisfied
\begin{enumerate}
    \item $(\uu,\theta)\in\L^\infty(0,T;\B^{0}_{3,2}(\Omega,\RR^3)\times\B^{-1}_{3/2,2,0}(\Omega))\cap\L^2(0,T;\B^{1}_{3,\infty}(\Omega,\RR^3)\times\B^0_{3/2,\infty}(\Omega))$;
    \item For almost every $t>0$,
    \begin{align*}
        \div \uu(t) =0,\quad \text{ in }\Omega,\quad\text{ and }\quad \uu(t)_{|_{\partial\Omega}} =0 \text{ on }\partial\Omega,
    \end{align*}
    respectively in the distributional sense, and in the trace sense.
    \item For all $\boldsymbol{\varphi}\in\Ccinfty([0,T);\Ccinftydiv(\Omega))$, all $\eta\in\Ccinfty([0,T);\D_{3,1}^0(\Delta_\mathcal{N}^\infty))$, it holds
    \begin{align}
        \int_{(0,T)\times\Omega} -\uu\cdot \partial_t \boldsymbol{\varphi} +\nu\nabla \uu: \nabla \boldsymbol{\varphi} - (\uu\otimes\uu) : \nabla \boldsymbol{\varphi} \,\d x \,\d t    &= \int_{(0,T)\times\Omega} \theta \boldsymbol{\varphi}_3\, \d x \d t  + \int_{\Omega} \uu_0\cdot\boldsymbol{\varphi}(0) \,\d x \label{eq:WeakFormBSQLpVelo}\\
        \text{ and }\qquad \int_{(0,T)\times\Omega} -\theta\cdot \partial_t \eta -\kappa \,\theta\,  \Delta \eta -  \theta \, \uu \cdot \nabla \eta \,\d x \,\d t    &=  \int_{\Omega} \theta_0\,\eta(0)\, \d x.\label{eq:WeakFormBSQLpTemp}
    \end{align}
\end{enumerate}
\end{definition}

\begin{definition}We say that $(\uu,\theta)\in\L^1_\textrm{loc}([0,T);\L^{1}(\Omega,\RR^3)\times\L^1(\Omega))$ is a solution to the Boussinesq system \eqref{eq:BoussinesqSystemIntro} on $(0,T)$ \textit{in Lebesgue-Besov spaces}  if the following is satisfied
\begin{enumerate}
    \item $(\uu,\theta)\in\L^\infty(0,T;\B^{0}_{3,1}(\Omega,\RR^3)\times\B^{-1}_{3/2,1,0}(\Omega))\cap\L^1(0,T;\D^{0}_{3,1}(\mathbb{A}_\mathcal{D},\Omega)\times\B^{1}_{3/2,1}(\Omega))$;
    \item For almost every $t>0$,
    \begin{align*}
        \div \uu(t) =0,\quad \text{ in }\Omega,\quad\text{ and }\quad \uu(t)_{|_{\partial\Omega}} =0 \text{ on }\partial\Omega,
    \end{align*}
    respectively in the distributional sense, and in the trace sense.
    \item For all $\boldsymbol{\varphi}\in\Ccinfty([0,T);\Ccinftydiv(\Omega))$, all $\eta\in\Ccinfty([0,T);\C^{\infty}(\overline{\Omega}))$, it holds
    \begin{align}
        \int_{(0,T)\times\Omega} -\uu\cdot \partial_t \boldsymbol{\varphi} +\nu\nabla \uu: \nabla \boldsymbol{\varphi} - (\uu\otimes\uu) : \nabla \boldsymbol{\varphi} \,\d x \,\d t    &= \int_{(0,T)\times\Omega} \theta \boldsymbol{\varphi}_3\, \d x \d t  + \int_{\Omega} \uu_0\cdot\boldsymbol{\varphi}(0) \,\d x \label{eq:WeakFormBSQBesovVelo}\\
        \text{ and }\qquad \int_{(0,T)\times\Omega} -\theta\cdot \partial_t \eta +\kappa \nabla \theta\cdot  \nabla\eta -  \theta \, \uu \cdot \nabla \eta \,\d x \,\d t    &=  \int_{\Omega} \theta_0\,\eta(0)\, \d x.\label{eq:WeakFormBSQBesovTemp}
    \end{align}
\end{enumerate}
\end{definition}

One should notice the difference between the two weak formulations concerning the choice of test functions for the temperature equation.

\subsection{Existence results}\label{Sec:WPBSQExistence}

\begin{theorem}[ Small times, arbitrary initial data  in $\B^{s}_{p,2}$]\label{thm:ExistenceSmallTimesLp} Let $\Omega\subset\RR^3$ be a bounded $\C^{1,\alpha}$ domain with $\alpha>1/3$. For any $(\uu_0,\theta_0)\in\B^{0,\sigma}_{3,2,0}(\Omega)\times\B^{-1}_{3/2,2,0}(\Omega)$, there exists $T>0$, depending on $\uu_0,\theta_0$, $\Omega$, $\nu$ and $\kappa$, such that
then the system \eqref{eq:BoussinesqSystemIntro} admits a (mild) solution on $(0,T)$:
\begin{align*}
    (\uu,\theta)\in\L^2(0,T;\B^{1,\sigma}_{3,\infty,0}(\Omega)\times \B^{0}_{3/2,\infty}(\Omega))\cap \C_{ub}^0([0,T];&\B^{0,\sigma}_{3,2,0}(\Omega)\times \B^{-1}_{3/2,2,0}(\Omega)).
\end{align*}
\end{theorem}

\begin{proof}Let $T\in(0,\infty)$, we consider the bilinear form $\mathcal{B}$ given by \eqref{eq:ultimateBilinearmap}, acting on $\mathrm{Z}_{3,T}\times\mathrm{K}_{3/2,T} := \mathrm{Z}_{3}((0,T)\times\Omega)\times\mathrm{K}_{3/2}((0,T)\times\Omega)$, see Notation~\ref{def:LpSolvSpacesBSQ}, thanks to Proposition~\ref{Prop:SolutionOperatorBoussinesqLp}. We set $\Tilde{\uu} := e^{\nu(\cdot)\AA_\mathcal{D}}\uu_0$ and $\Tilde{\Theta}:= e^{\kappa(\cdot)\Delta_\mathcal{N}}\theta_0$, it remains to analyze to following problem
\begin{align*}
    \left(\begin{array}{c}
         \vv \\
         \zeta
    \end{array}\right) =   \mathcal{B}\left(\left(\begin{array}{c}
         \Tilde{\uu} \\
         \Tilde{\Theta}
    \end{array}\right) ,\left(\begin{array}{c}
         \Tilde{\uu} \\
         \Tilde{\Theta}
    \end{array}\right) \right)+ 2 \mathcal{B}\left(\left(\begin{array}{c}
         \Tilde{\uu} \\
         \Tilde{\Theta}
    \end{array}\right),\left(\begin{array}{c}
         \vv \\
         \zeta
    \end{array}\right)\right)+ \mathcal{B}\left(\left(\begin{array}{c}
         \vv \\
         \zeta
    \end{array}\right),\left(\begin{array}{c}
         \vv \\
         \zeta
    \end{array}\right)\right).
\end{align*}
This is a problem in the form
\begin{align*}
    \mathbf{x} = \Tilde{\mathbf{x}}_0 + \mathcal{L}(\mathbf{x}) +\mathcal{B}(\mathbf{x},\mathbf{x}), 
\end{align*}
with 
\begin{align}\label{eq:LinPartBiLinprob}
    \mathcal{L}\,:\,\left(\begin{array}{c}
         \vv \\
         \eta
    \end{array}\right)\longmapsto
    2\mathcal{B}\left(\left(\begin{array}{c}
         \Tilde{\uu} \\
         \Tilde{\Theta}
    \end{array}\right),\left(\begin{array}{c}
         \vv \\
         \eta
    \end{array}\right)\right)
\end{align}
So we aim to apply the fixed point result Proposition~\ref{prop:FixedPoint2Bil+Lin}, hence we have to find $T>0$ such that $\mathcal{L}$ has norm strictly less than $1$ on $\mathrm{Z}_{3,T}\times\mathrm{K}_{3/2,T}$. Thanks to Proposition~\ref{Prop:SolutionOperatorBoussinesqLp}, \eqref{eq:ultimateBilinearmap}, the task is reduced to the analysis of the linear operators
\begin{align*}
    \vv\mapsto \C_{\kappa}(\Tilde{\Theta},\vv),\quad \eta\mapsto \C_{\kappa}(\eta,\Tilde{\uu}) ,\quad\text{ and }\quad \vv\longmapsto\B_\nu(\Tilde{\uu},\vv),
\end{align*}
for which it is sufficient to prove one can find $T>0$ such that each does have small enough norm.

\textbf{Step 1:} First for $\B_\nu$, following the proof of Proposition~\ref{Prop:SolutionOperatorBoussinesqLp}, Step~2, up to \eqref{eq:ProofBSQBiLinNSEEstLp}, one obtains
\begin{align*}
    \lVert &\B_\nu(\Tilde{\uu},\vv) \rVert_{\mathrm{Z}_{3,T}}\\
    & \lesssim_{\Omega} \frac{1}{\sqrt{\nu}}  \lVert \Tilde{\uu}\rVert_{\L^\infty(0,T;\B^{0}_{3,2}(\Omega))}^\frac{1}{2} \lVert \nabla \Tilde{\uu}\rVert_{\L^2(0,T;\B^{0}_{3,\infty}(\Omega))}^\frac{1}{2}\lVert \vv\rVert_{\L^\infty(0,T;\B^{0}_{3,2}(\Omega))}^\frac{1}{2} \lVert \nabla \vv\rVert_{\L^2(0,T;\B^{0}_{3,\infty}(\Omega))}^\frac{1}{2}\\
    & \lesssim_{\Omega} \frac{1}{\nu}  \lVert {\uu}_0\rVert_{\B^{0}_{3,2}(\Omega)}^\frac{1}{2} ( \sqrt{\nu}\lVert \nabla \Tilde{\uu}\rVert_{\L^2(0,T;\B^{0}_{3,\infty}(\Omega))})^\frac{1}{2}\lVert \vv\rVert_{\L^\infty(0,T;\B^{0}_{3,2}(\Omega))}^\frac{1}{2} (\sqrt{\nu}\lVert \nabla \vv\rVert_{\L^2(0,T;\B^{0}_{3,\infty}(\Omega))})^\frac{1}{2}.
\end{align*}
Where the last line was obtained  by uniform boundedness of the semigroup $(e^{-\nu t\AA_\mathcal{D}})_{t\geqslant 0}$ on $\B^{0}_{3,2}(\Omega)$. Hence, by Young's inequality and the dominated convergence theorem, since $\nabla \Tilde{\uu}\in\L^2(\RR_+;\B^{0}_{3,\infty}(\Omega))$, we deduce
\begin{align}\label{eq:SmallnormBnuTStokes}
    \lVert \B_\nu(\Tilde{\uu},\vv) \rVert_{\mathrm{Z}_{3,T}}& \lesssim_{\Omega} \frac{1}{\nu}  \lVert {\uu}_0\rVert_{\B^{0}_{3,2}(\Omega)}^\frac{1}{2} ( \sqrt{\nu}\lVert \nabla \Tilde{\uu}\rVert_{\L^2(0,T;\B^{0}_{3,\infty}(\Omega))})^\frac{1}{2} \lVert \vv\rVert_{\mathrm{Z}_{3,T}} \xrightarrow[T\longrightarrow0]{}0.
\end{align}

\textbf{Step 2:} Second, for $\C_\kappa$, similarly, one proceeds  to follow the proof of Proposition~\ref{Prop:SolutionOperatorBoussinesqLp}, Step~3, up to \eqref{eq:ProofBSQBiLinTempEstLp}, in order to obtain
\begin{align*}
    \lVert \C_\kappa(\vv&,\tilde{\Theta}) \rVert_{\mathrm{K}_{3/2,T}}\\
    &\lesssim_{\Omega} \frac{1}{\sqrt{\kappa}} \lVert \vv\rVert_{\L^{\infty}(0,T;\B^{0}_{3,2}(\Omega))}^\frac{1}{2} \lVert \vv\rVert_{\L^{2}(0,T;\B^{1}_{3,\infty}(\Omega))}^\frac{1}{2} \lVert \tilde{\Theta}\rVert_{\L^{\infty}(0,T;\B^{-1}_{3/2,2}(\Omega))}^\frac{1}{2} \lVert \tilde{\Theta}\rVert_{\L^{2}(0,T;\B^{0}_{3/2,\infty}(\Omega))}^\frac{1}{2}.
    \end{align*}
Now, then applying Young's inequality,  yield
\begin{align*}
    \lVert \C_\kappa(\vv&,\tilde{\Theta}) \rVert_{\mathrm{K}_{3/2,T}}\\
    &\lesssim_{\Omega} \frac{1}{\sqrt{\kappa}\nu^{1/4}} \lVert \tilde{\Theta}\rVert_{\L^{\infty}(0,T;\B^{-1}_{3/2,2}(\Omega))}^\frac{1}{2} \lVert \tilde{\Theta}\rVert_{\L^{2}(0,T;\B^{0}_{3/2,\infty}(\Omega))}^\frac{1}{2} \lVert \vv\rVert_{\mathrm{Z}_{3,T}} .
    \end{align*}
    Finally, uniform boundedness of the Neumann semigroup $(e^{\kappa t\Delta_\mathcal{N}})_{t\geqslant 0}$ on $\B^{-1}_{3/2,2,0}(\Omega)$, and since $\tilde{\Theta}\in\L^{2}_{\text{loc}}([0,\infty);\B^{0}_{3/2,\infty}(\Omega))$, the dominated convergence theorem yields
    \begin{align}
        \lVert \C_\kappa(\vv&,\tilde{\Theta}) \rVert_{\mathrm{K}_{3/2,T}}\nonumber\\
    &\lesssim_{\Omega,\nu,\kappa}\lVert \theta_0\rVert_{\B^{-1}_{3/2,2}(\Omega)}^\frac{1}{2} \lVert \tilde{\Theta}\rVert_{\L^{2}(0,T;\B^{0}_{3/2,\infty}(\Omega))}^\frac{1}{2} \lVert \vv\rVert_{\mathrm{Z}_{3,T}}  \xrightarrow[T\longrightarrow0]{}0.\label{eq:SmallnormCkappaTTheta0}
    \end{align}
This gave the result for $\vv\longmapsto \C_\kappa (\vv,\tilde{\Theta})$. The exacts same arguments apply to $\eta\longmapsto \C_\kappa (\tilde{\uu},\eta)$, so that we obtain
    \begin{align}
        \lVert \C_\kappa(\tilde{\uu}&,\eta) \rVert_{\mathrm{K}_{3/2,T}}\nonumber\\
    &\lesssim_{\Omega,\nu,\kappa}\lVert \uu_0\rVert_{\B^{0}_{3,2}(\Omega)}^\frac{1}{2} \lVert \nabla \tilde{\uu}\rVert_{\L^{2}(0,T;\B^{0}_{3,\infty}(\Omega))}^\frac{1}{2} \lVert \eta\rVert_{\mathrm{K}_{3/2,T}}  \xrightarrow[T\longrightarrow0]{}0\label{eq:SmallnormCkappaTu0}.
    \end{align}
\textbf{Step 3:} We finally deduce the claimed result. By \eqref{eq:SmallnormBnuTStokes}, \eqref{eq:SmallnormCkappaTTheta0} and \eqref{eq:SmallnormCkappaTu0}, for all $\varepsilon>0$, there exists $T>0$, depending on $\uu_0$ and $\theta_0$, such that
\begin{align*}
    \lVert\B_\nu(\Tilde{\uu},\cdot) \rVert_{\mathrm{Z}_{3,T}\rightarrow\mathrm{Z}_{3,T}},\quad  \lVert\C_\kappa( \cdot ,\tilde{\Theta}) \rVert_{\mathrm{Z}_{3,T}\rightarrow\mathrm{K}_{3/2,T}},\quad  \lVert\C_\kappa( \tilde{\uu}, \cdot) \rVert_{\mathrm{K}_{3/2,T}\rightarrow\mathrm{K}_{3/2,T}}< \varepsilon.
\end{align*}
Due to \eqref{eq:LinPartBiLinprob} and \eqref{eq:ultimateBilinearmap}, one obtains
\begin{align*}
    \lVert\mathcal{L}\rVert_{ \mathrm{Z}_{3,T}\times\mathrm{K}_{3/2,T}\rightarrow \mathrm{Z}_{3,T}\times\mathrm{K}_{3/2,T}} 
    < 2\varepsilon (3+2\lVert \L_{\nu}\rVert_{\mathrm{K}_{3/2,\infty}\rightarrow\mathrm{Z}_{3,\infty}}).
\end{align*}
Hence, choosing $T>0$, such that $\varepsilon<\frac{1}{(3+2\lVert \L_{\nu}\rVert_{\mathrm{K}_{3/2,\infty}\rightarrow\mathrm{Z}_{3,\infty}})}$, there exists $T>0$ such that  $$\lVert\mathcal{L}\rVert_{ \mathrm{Z}_{3,T}\times\mathrm{K}_{3/2,T}\rightarrow \mathrm{Z}_{3,T}\times\mathrm{K}_{3/2,T}} <1.$$
For the exact same reasons, with the same arguments as those present in Steps 1 and 2, one can apply above procedures to show that $\mathcal{B}\left(\left(\begin{array}{c}
         \Tilde{\uu} \\
         \Tilde{\Theta}
    \end{array}\right) ,\left(\begin{array}{c}
         \Tilde{\uu} \\
         \Tilde{\Theta}
    \end{array}\right) \right)$ has arbitrarily small norm for a well chosen $T>0$, with the appropriate bound, so that by Proposition~\ref{prop:FixedPoint2Bil+Lin} we deduce the result of existence.
\end{proof}

\begin{theorem}[ Small times, arbitrary initial data in $\B^{s}_{p,1}$] Let $\Omega\subset\RR^3$ be a bounded Lipschitz domain. For any $(\uu_0,\theta_0)\in\B^{0,\sigma}_{3,1,0}(\Omega)\times\B^{-1}_{3/2,1,0}(\Omega)$, there exists $T>0$, depending on $\uu_0,\theta_0$, $\Omega$, $\nu$ and $\kappa$, such that
then the system \eqref{eq:BoussinesqSystemIntro} admits a (mild) solution on $(0,T)$:
\begin{align*}
    (\uu,\theta)\in\C_{ub}^0([0,T];&\B^{0}_{3,1}(\Omega,\RR^3)\times \B^{-1}_{3/2,1,0}(\Omega))\cap\L^1(0,T;\D^{0}_{3,1}(\mathbb{A}_\mathcal{D},\Omega)\times \B^{1}_{3/2,1}(\Omega)).
\end{align*}
\end{theorem}

\begin{proof} The proof follows exactly the same lines as the proof of Theorem~\ref{thm:ExistenceSmallTimesLp}, taking advantage of \eqref{eq:ProofBiLinestStokes} and \eqref{eq:ProofBiLinestBoussinesTransp} from the proof of Proposition~\ref{Prop:SolutionOperatorBoussinesqBesov}.
\end{proof}

Now, we dive in the global-in-time existence results.

\begin{theorem}[ Large times, small initial data in $\B^s_{p,2}$
]\label{thm:LargeTimesExistenceLp} Let $\Omega\subset\RR^3$ be a bounded $\C^{1,\alpha}$ domain. Let $(\uu_0,\theta_0)\in\B^{0,\sigma}_{3,2,0}(\Omega)\times\B^{-1}_{3/2,2,0}(\Omega)$. We set $\Tilde{\theta}_0=\frac{1}{|\Omega|}\int_{\Omega}\theta_0\in\RR$.  There exists a constant $c>0$, depending only on $\Omega$, $\nu$ and $\kappa$, such that if
\begin{align*}
    \lVert \uu_0\rVert_{\B^{0}_{3,2}(\Omega)}+\lVert \theta_0-\Tilde{\theta}_0\rVert_{\B^{-1}_{3/2,2}(\Omega)} \leqslant c
\end{align*}
then the system \eqref{eq:BoussinesqSystemIntro} admits a global-in-time (mild) solution such that
\begin{align*}
    (\uu,\theta-\tilde{\theta}_0)\in\L^2(\RR_+;\B^{1,\sigma}_{3,\infty,0}(\Omega)\times \B^{0}_{3/2,\infty,\circ}(\Omega))\cap\C^0_{ub}(\overline{\RR_+};\B^{0,\sigma}_{3,2,0}(\Omega)\times \B^{-1}_{3/2,2,\circ}(\Omega)).
\end{align*}
with the estimate
\begin{align*}
    \lVert   \uu  \rVert_{\L^\infty(\RR_+;\B^{0}_{3,2}(\Omega))} +\nu^\frac{1}{2}\lVert & \nabla\uu\rVert_{\L^2(\RR_+,\B^{0}_{3,\infty}(\Omega))}\\
    &+ \lVert   \theta -\Tilde{\theta}_0 \rVert_{\L^\infty(\RR_+,\B^{-1}_{3/2,2}(\Omega))} + \kappa^\frac{1}{2}\lVert  \theta -\Tilde{\theta}_0 \rVert_{\L^2(\RR_+,\B^{0}_{3/2,\infty}(\Omega))}  \lesssim_{\Omega} c.
\end{align*}
\end{theorem}

\begin{proof} \textbf{Step 1:} The case $\theta_0\in\B^{-1}_{3/2,2,\circ}(\Omega)\subset\B^{-1}_{3/2,2,0}(\Omega)$, \textit{i.e.} $\tilde{\theta}_0=0$.

Then the result follows directly from the application of Proposition~\ref{prop:FixedPointBilin} to the abstract problem
\begin{align*}
    \left(\begin{array}{c}
         \uu \\
         \theta
    \end{array}\right) =  \left(\begin{array}{c}
         e^{-(\cdot) \nu\AA_\mathcal{D}}\uu_0 + \L_{\nu}( e^{(\cdot) \kappa\Delta_\mathcal{N}}\theta_0)\\
         e^{(\cdot) \kappa\Delta_\mathcal{N}}\theta_0
    \end{array}\right)+ \mathcal{B}\left(\left(\begin{array}{c}
         \uu \\
         \theta
    \end{array}\right),\left(\begin{array}{c}
         \uu \\
         \theta
    \end{array}\right)\right).
\end{align*}
where   $\mathcal{B}$ is the bilinear map given by \eqref{eq:ultimateBilinearmap}, well-defined and bounded on $\Z_3(\RR_+\times\Omega)\times \K_{3/2}(\RR_+\times\Omega)$, see Notations~\ref{def:LpSolvSpacesBSQ}, by Proposition~\ref{Prop:SolutionOperatorBoussinesqLp}.

\textbf{Step 2:} Now, assume that $\theta_0\in\B^{-1}_{3/2,2,0}(\Omega)$, and consider the decomposition $\theta_0 = \theta_0^{\circ}+ \tilde{\theta}_0 \mathds{1}_{\Omega}$, where $\tilde{\theta}_0= |\Omega|^{-1}\langle \tilde{\theta}_0, \mathbf{1}\rangle_{\Omega}$. We identify $\tilde{\theta}_0 \mathds{1}_{\Omega}$ with $\tilde{\theta}_0\in\RR$. Thus $\theta_0^{\circ}$ is mean free. By the previous Step, there exists as solution $(\uu,\theta^\circ)$  to 
\begin{align*}
    \left(\begin{array}{c}
         \uu \\
         \theta^{\circ}
    \end{array}\right) =  \left(\begin{array}{c}
         e^{-(\cdot) \nu\AA_\mathcal{D}}\uu_0 + \L_{\nu}( e^{(\cdot) \kappa\Delta_\mathcal{N}}\theta_0^{\circ})\\
         e^{(\cdot) \kappa\Delta_\mathcal{N}}\theta_0^{\circ}
    \end{array}\right)+ \mathcal{B}\left(\left(\begin{array}{c}
         \uu \\
         \theta^{\circ}
    \end{array}\right),\left(\begin{array}{c}
         \uu \\
         \theta^{\circ}
    \end{array}\right)\right).
\end{align*}
Now, we make a remark: for all $c\in\RR$, one has $\PP_{\Omega}(c\mathfrak{e}_n) = c\PP_{\Omega}(\nabla x_n)=0$. Thus, by the definition of $\mathcal{B}$, \eqref{eq:ultimateBilinearmap}, subordinated to $\B_\nu$, \eqref{eq:BOpBiLin}, and $\L_\nu$, \eqref{eq:LOpLin}, it holds that $\mathcal{B}((0,c), \cdot) =0$ and $\L_\nu(c) =0$, and we obtain
\begin{align*}
    \left(\begin{array}{c}
         \uu \\
         \theta^{\circ} +\tilde{\theta}_0
    \end{array}\right) &=  \left(\begin{array}{c}
         e^{-(\cdot) \nu\AA_\mathcal{D}}\uu_0 + \L_{\nu}( e^{(\cdot) \kappa\Delta_\mathcal{N}}\theta_0^{\circ})\\
         e^{(\cdot) \kappa\Delta_\mathcal{N}}\theta_0^{\circ} +\tilde{\theta}_0
    \end{array}\right)+ \mathcal{B}\left(\left(\begin{array}{c}
         \uu \\
         \theta^{\circ}
    \end{array}\right),\left(\begin{array}{c}
         \uu \\
         \theta^{\circ}
    \end{array}\right)\right)\\
    &=  \left(\begin{array}{c}
         e^{-(\cdot) \nu\AA_\mathcal{D}}\uu_0 + \L_{\nu}( e^{(\cdot) \kappa\Delta_\mathcal{N}}\theta_0^{\circ} + \tilde{\theta}_0)\\
         e^{(\cdot) \kappa\Delta_\mathcal{N}}\theta_0^{\circ} +\tilde{\theta}_0
    \end{array}\right)+ \mathcal{B}\left(\left(\begin{array}{c}
         \uu \\
         \theta^{\circ}+\tilde{\theta}_0
    \end{array}\right),\left(\begin{array}{c}
         \uu \\
         \theta^{\circ}+\tilde{\theta}_0
    \end{array}\right)\right).
\end{align*}
From this point, we now use the fact that $e^{(\cdot)\kappa\Delta_{N}}c =c$ for all $c\in\RR$, and we set $\theta := \theta^{\circ}+\tilde{\theta}_0$, so that we are recovering
\begin{align*}
    \left(\begin{array}{c}
         \uu \\
         \theta
    \end{array}\right) &=  \left(\begin{array}{c}
         e^{-(\cdot) \nu\AA_\mathcal{D}}\uu_0 + \L_{\nu}( e^{(\cdot) \kappa\Delta_\mathcal{N}}\theta_0)\\
         e^{(\cdot) \kappa\Delta_\mathcal{N}}\theta_0 
    \end{array}\right)+ \mathcal{B}\left(\left(\begin{array}{c}
         \uu \\
         \theta
    \end{array}\right),\left(\begin{array}{c}
         \uu \\
         \theta
    \end{array}\right)\right).
\end{align*}
Consequently, we have obtained a solution with the expected properties, which ends the proof.
\end{proof}

\begin{theorem}[ Large times, small initial data in $\B^{s}_{p,1}$]\label{thm:LargeTimesExistenceBesov} Let $\Omega\subset\RR^3$ be a bounded Lipschitz domain. Let $(\uu_0,\theta_0)\in\B^{0,\sigma}_{3,1,0}(\Omega)\times\B^{-1}_{3/2,1,0}(\Omega)$. We set $\Tilde{\theta}_0=\frac{1}{|\Omega|}\int_{\Omega}\theta_0\in\RR$. There exists a constant $c>0$, depending only on $\Omega$, $\nu$ and $\kappa$, such that if  $$\lVert \uu_0\rVert_{\B^{0}_{3,1}(\Omega)}+\lVert \theta_0-\Tilde{\theta}_0\rVert_{\B^{-1}_{3/2,1}(\Omega)} \leqslant c$$
then the system \eqref{eq:BoussinesqSystemIntro} admits a global-in-time (mild) solution
\begin{align*}
    (\uu,\theta)\in\C^0_{ub}(\overline{\RR_+};\B^{0,\sigma}_{3,1,0}(\Omega)\times \B^{-1}_{3/2,1}(\Omega))\cap\L^1_{\text{loc}}(\overline{\RR_+};\D^{0}_{3,1}(\mathbb{A}_\mathcal{D},\Omega)\times \B^{1}_{3/2,1}(\Omega)),
\end{align*}
and it is such that
\begin{align*}
    (\uu,\theta-\tilde{\theta}_0)\in\C^0_{0}(\overline{\RR_+};\B^{0,\sigma}_{3,1,0}(\Omega)\times \B^{-1}_{3/2,1}(\Omega))\cap\L^1(\overline{\RR_+};\D^{0}_{3,1}(\mathbb{A}_\mathcal{D},\Omega)\times \B^{1}_{3/2,1}(\Omega)),
\end{align*}
Furthermore, one has the estimates
\begin{align*}
    \lVert \uu  \rVert_{\L^\infty(\RR_+,\B^{0}_{3,1}(\Omega))} + &\lVert \nu\AA_\mathcal{D}\uu  \rVert_{\L^1(\RR_+,\B^{0}_{3,1}(\Omega))}\\
    &+ \lVert \theta - \tilde{\theta_0}  \rVert_{\L^\infty(\RR_+,\B^{-1}_{3/2,1}(\Omega))} +\lVert  \kappa\nabla^2\theta  \rVert_{\L^1(\RR_+,\B^{-1}_{3/2,1}(\Omega))} \lesssim_{\Omega} c.
\end{align*}
\end{theorem}

\begin{proof} The proof is the same as the one of Theorem~\ref{thm:LargeTimesExistenceLp}, taking now advantage of Proposition~\ref{Prop:SolutionOperatorBoussinesqBesov}.
\end{proof}

\subsection{Uniqueness results}\label{Sec:WPBSQUniqueness}

We now prove uniqueness. Note that, for the uniqueness of weak solutions, the boundedness of the domain $\Omega$ is crucial for our argument. The goal is to use information from the linear Hilbertian $\L^2$-theory, in order to prove that the difference of two weak solutions we consider can be written as mild solution of an appropriate problem. Thus, we critically rely on the embedding $\L^3(\Omega)\hookrightarrow\L^2(\Omega)$, only valid for bounded domains. We will also need a slight amount of boundary regularity in order to close our arguments. Without boundedness or boundary regularity requirements for the domain-- say if we are in the half-space (using then homogeneous Sobolev and Besov spaces), or if the domain is bounded and only has Lipschitz boundary-- only the second parts in the proofs of Theorem~\ref{thm:UniquenessWeakLp}~and~\ref{thm:UniquenessWeakBesov} below are applicable. In this case, one  recovers then uniqueness only in the sole class of mild solutions, see for instance Proposition~\ref{prop:UniqMild}.

\begin{theorem}\label{thm:UniquenessWeakLp} Let $\Omega\subset\RR^3$ be a bounded $\C^{1,\alpha}$ domain with $\alpha>1/3$ and  $T\in(0,\infty]$. Let $(\uu_0,\theta_0)\in\B^{0,\sigma}_{3,2,0}(\Omega)\times\B^{-1}_{3/2,2,0}(\Omega)$, and let
\begin{align*}
    (\uu_1,\theta_1),\quad (\uu_2,\theta_2)\in\L^\infty(0,T;&\,\B^{0,\sigma}_{3,2,0}(\Omega)\times \B^{-1}_{3/2,2,0}(\Omega))\cap\L^2(0,T;\B^{1,\sigma}_{3,\infty,0}(\Omega)\times \B^{0}_{3/2,\infty}(\Omega))
\end{align*}
such that both satisfies the system \eqref{eq:BoussinesqSystemIntro} on $(0,T)$ with initial value $(\uu_0,\theta_0)$ in the sense of \eqref{eq:WeakFormBSQLpVelo}--\eqref{eq:WeakFormBSQLpTemp}. Then, it holds
\begin{align*}
   \forall s\in[0,T),\quad (\uu_1,\theta_1)(s) = (\uu_2,\theta_2)(s).
\end{align*}
\end{theorem}

\begin{proof} We set $\ww = \uu_1-\uu_2$ and $\Theta = \theta_1-\theta_2$, for all $\boldsymbol{\varphi}\in\Ccinfty([0,T);\Ccinftydiv(\Omega))$, all $\eta\in\Ccinfty([0,T);\D_{3,1}^0(\Delta_\mathcal{N}^\infty))$, it holds
\begin{align}
        \int_{(0,T)\times\Omega} -\ww\cdot \partial_t \boldsymbol{\varphi} +\nu\nabla \ww: \nabla \boldsymbol{\varphi} - (\ww\otimes[\uu_1+\uu_2]) : \nabla \boldsymbol{\varphi} \,\d x \,\d t    &= \int_{(0,T)\times\Omega} \Theta \boldsymbol{\varphi}_3\, \d x \d t \label{eq:WeakFormulProofUniqLpVelo} \\
        \text{ and }\qquad \int_{(0,T)\times\Omega} -\Theta\cdot \partial_t \eta -\kappa \,\Theta\,  \Delta \eta -  \theta_1 \, \ww \cdot \nabla \eta-  \Theta \, \uu_2 \cdot \nabla \eta \,\d x \,\d t    &= 0. \label{eq:WeakFormulProofUniqLpTemp}
\end{align}
\textbf{Step 1:} We show that  $\ww$ is a mild solution to a linear Stokes evolution problem. By linearity, and since $\Omega$ is bounded, for $\varepsilon>0$ small enough it holds that $$\ww\in\L^\infty(0,T;\,\B^{0,\sigma}_{3,2,0}(\Omega))\cap\L^2(0,T;\B^{1,\sigma}_{3,\infty,0}(\Omega))\hookrightarrow\L^\infty(0,T;\,\L^{2}_{\mathfrak{n,\sigma}}(\Omega))\cap\L^2(0,T;\W^{1-\varepsilon,2}_{0,\sigma}(\Omega)).$$
Therefore, one obtains 
\begin{align*}
    \AA_{\mathcal{D}}^{-\frac{\varepsilon}{2}}\ww \in \L^\infty(0,T;\,\L^{2}_{\mathfrak{n,\sigma}}(\Omega))\cap\L^2(0,T;\W^{1,2}_{0,\sigma}(\Omega)).
\end{align*}

On the other hand, as in the proof of Proposition~\ref{Prop:SolutionOperatorBoussinesqLp}, \eqref{eq:ProofNSEpartL4W1/23Est}--\eqref{eq:ProofBSQBiLinNSEEstLp}, one obtains $\ww\otimes[\uu_1+\uu_2]\in\L^2(0,T;\L^3(\Omega)^9)\subset \L^2(0,T;\L^2(\Omega)^9)$ and $\Theta {\mathfrak{e}_3}\in\L^2(0,T;\B^0_{3/2,\infty,\circ}(\Omega))\hookrightarrow \L^2(0,T;\W^{-1,2}_\circ(\Omega))$, where the mean free condition is obtained by testing the temperature equation with  $\eta\,:\,(t,x)\mapsto \eta(t)$ in \eqref{eq:WeakFormulProofUniqLpTemp}. Therefore, there exists $F\in\L^2(0,T;\L^2(\Omega)^9)$, such that $\div (F)=\Theta\mathfrak{e}_3$, and it holds that $F-\ww\otimes[\uu_1+\uu_2]\in\L^2(0,T;\L^2(\Omega)^9)$, since $\ww\otimes[\uu_1+\uu_2]\in\L^2(0,T;\L^2(\Omega)^9)$ as in \eqref{eq:ProofNSEpartL4W1/23Est}. By \cite[Thm.~2.4.1]{SohrBook2001}, up to apply $\AA_{\mathcal{D}}^{-\frac{\varepsilon}{2}}$ since it is an isomorphism, for almost every $t\in[0,T)$, one can write
\begin{align}
    \AA_{\mathcal{D}}^{-\frac{\varepsilon}{2}}\ww(t) &= \int_{0}^t e^{-\nu(t-s)\AA_\mathcal{D}}\AA_{\mathcal{D}}^{-\frac{\varepsilon}{2}}\PP_{\Omega}\div( F(s) - \ww(s)\otimes[\uu_1(s)+\uu_2(s)])\,\d s\nonumber\\ &= \int_{0}^t e^{-\nu(t-s)\AA_\mathcal{D}}\AA_{\mathcal{D}}^{-\frac{\varepsilon}{2}}\PP_{\Omega}[\Theta(s)\mathfrak{e}_3]\,\d s- \int_{0}^t e^{-\nu(t-s)\AA_\mathcal{D}}\AA_{\mathcal{D}}^{-\frac{\varepsilon}{2}}\PP_{\Omega}\div(\ww(s)\otimes[\uu_1(s)+\uu_2(s)])\,\d s\nonumber\\
    &= \AA_{\mathcal{D}}^{-\frac{\varepsilon}{2}}\L_\nu(\Theta)(t)+\AA_{\mathcal{D}}^{-\frac{\varepsilon}{2}}\B_\nu(\ww,\uu_1+\uu_2)(t),\label{eq:ProofThetaUniquenessVelo}
\end{align}
and it holds that $\AA_{\mathcal{D}}^{-\frac{\varepsilon}{2}}\ww\in\C^{0}([0,T);\L^{2}_{\mathfrak{n},\sigma}(\Omega))\cap\W^{1,2}(0,T;\AA_\mathcal{D}^{1/2}\L^{2}_{\mathfrak{n},\sigma}(\Omega))$. Furthermore, by Propositions~\ref{thm:LqMaxRegStokesBesov}~and~\ref{Prop:SolutionOperatorBoussinesqBesov}, it holds that
\begin{align*}
    \ww\in\C^0([0,T);\,\B^{0,\sigma}_{3,2,0}(\Omega))\cap\L^2(0,T;\B^{1,\sigma}_{3,\infty,0}(\Omega)).
\end{align*}

\textbf{Step 2:} We show that $\Theta$ is a mild solution. Now, by the isomorphism properties Propositions~\ref{prop:NeumannBesov}~and~\ref{prop:NeumannRegularity}, and Sobolev embeddings, for $\varepsilon>0$ small enough, it holds
\begin{align}
    \Theta \in &\L^\infty(0,T;\, \B^{-1}_{3/2,2,\circ}(\Omega))\cap\L^2(0,T;\B^{0}_{3/2,\infty,\circ}(\Omega))\nonumber\\
    &= \L^\infty(0,T;\, (-\Delta_\mathcal{N})^{(1+\varepsilon)/4}\B^{-1/2}_{3/2,2,\circ}(\Omega))\cap\L^2(0,T;(-\Delta_\mathcal{N})^{(1+\varepsilon)/4}\H^{1/2,3/2}_{\circ}(\Omega))\nonumber \\
    &\hookrightarrow \L^\infty(0,T;\, (-\Delta_\mathcal{N})^{(1+\varepsilon)/4}\H^{-1,2}_{\circ}(\Omega))\cap\L^2(0,T;(-\Delta_\mathcal{N})^{(1+\varepsilon)/4}\L^{2}_\circ(\Omega)) \nonumber\\
    &= \L^\infty(0,T;\, (-\Delta_\mathcal{N})^{(3+\varepsilon)/4}\L^{2}_{\circ}(\Omega))\cap\L^2(0,T;(-\Delta_\mathcal{N})^{(3+\varepsilon)/4}\W^{1,2}_\circ(\Omega)) \label{eq:UniqunessL2MildSpaceTheta}.
\end{align}
Moreover, it holds that $(-\Delta_\mathcal{N})^{-(3+\varepsilon)/4}[\div (\uu_1\Theta) + \div(\ww \theta_2)]\in\L^2(0,T;\W^{1,2}_{\circ}(\Omega))$: indeed by the isomorphisms properties for the Neumann Laplacian, Proposition~\ref{prop:NeumannBesov}, the boundedness of $(-\Delta_\mathcal{N})^{-\varepsilon/4}$ and by Sobolev embeddings it holds
\begin{align}
    \lVert (-\Delta_\mathcal{N})^{-(3+\varepsilon)/4}&[\div (\uu_2\Theta) + \div(\ww \theta_1)]\rVert_{\L^2(0,T;\W^{-1,2}(\Omega))}\nonumber\\
    &\lesssim_{\Omega}\lVert (-\Delta_\mathcal{N})^{-1/2}[\div (\uu_2\Theta) + \div(\ww \theta_1)]\rVert_{\L^2(0,T;\B^{-1}_{3/2,\infty}(\Omega))}\label{eq:membershipForcingTermThetaProofUniqueLp}
\end{align}
and therefore, one can conclude following the arguments of \eqref{eq:Est:C(v,O)L2Lp}--\eqref{eq:ProofBSQBiLinTempEstLp}. Thus, we have obtained
\begin{itemize}
    \item $\Theta\in\L^\infty(0,T;\, (-\Delta_\mathcal{N})^{(3+\varepsilon)/4}\L^{2}_{\circ}(\Omega))\cap\L^2(0,T;(-\Delta_\mathcal{N})^{(3+\varepsilon)/4}\W^{1,2}_\circ(\Omega))$;
    \item $f:=-[\div (\uu_2\Theta) + \div(\ww \theta_1)]\in\L^2(0,T;(-\Delta_\mathcal{N})^{(3+\varepsilon)/4}\W^{-1,2}_\circ(\Omega))$.
\end{itemize}
Testing the equation \eqref{eq:WeakFormulProofUniqLpTemp} with elements $\eta\in\Ccinfty(0,T;\D_{3,1}^0(\Delta_\mathcal{N}^\infty)) \hookrightarrow \Ccinfty(0,T;\D_{2}(\Delta_\mathcal{N}^\infty))$, we deduce
\begin{align*}
    \partial_t\Theta\in\L^2(0,T;(-\Delta_\mathcal{N})^{(3+\varepsilon)/4}\W^{-1,2}_\circ(\Omega)),
\end{align*}
and therefore $\Theta\in\C^0([0,T);(-\Delta_\mathcal{N})^{(3+\varepsilon)/4}\L^{2}_{\circ}(\Omega))$. This implies that $\Theta$ is a mild solution and can be written for all $t\in[0,T)$ as
\begin{align}
    \Theta(t)&=-\int_{0}^t e^{\kappa (t-s)\Delta_\mathcal{N}}[\div (\uu_2(s)\Theta(s)) + \div(\ww(s) \theta_1(s))]\,\d s\nonumber\\
    &= \C_\kappa(\uu_2,\Theta)(t) + \C_\kappa(\ww,\theta_1)(t)\label{eq:ProofThetaUniqueness}.
\end{align}
Thus, by Propositions~\ref{thm:LqMaxRegNeumannBesov}~and~\ref{Prop:SolutionOperatorBoussinesqBesov}, it holds that
\begin{align*}
    \Theta\in\C^0([0,T); \B^{-1}_{3/2,2}(\Omega))\cap\L^2(0,T;\B^{0}_{3/2,\infty}(\Omega))
\end{align*}

\textbf{Step 3:} We show that $\ww=0$ and $\Theta=0$. First, notice that the set
\begin{align*}
    I=\{ t\in[0,T)\,|\, (\ww(t),\Theta(t))=(0,0) \text{ in } \B^{0}_{3,2}(\Omega,\RR^3)\times\B^{-1}_{3/2,2,0}(\Omega)\}
\end{align*}
is closed and non empty. We show that it is open. By a maximality argument and connectedness it suffices to show that $[0,\tau)\subset I$ for some $\tau>0$.

We recall that plugging \eqref{eq:ProofThetaUniqueness} in \eqref{eq:ProofThetaUniquenessVelo}, one can write
\begin{align*}
    \ww &= \B_\nu(\ww,\uu_1+\uu_2) + \L_\nu(\C_\kappa(\uu_2,\Theta)) + \L_\nu(\C_\kappa(\ww,\theta_1))\\
    \Theta&= \C_\kappa(\uu_2,\Theta) + \C_\kappa(\ww,\theta_1)
\end{align*}

Let $\tau>0$, following the steps that lead respectively to \eqref{eq:SmallnormBnuTStokes}, \eqref{eq:SmallnormCkappaTTheta0} and \eqref{eq:SmallnormCkappaTu0} in the proof of Theorem~\ref{thm:ExistenceSmallTimesLp}, it holds that
\begin{align*}
    &\lVert \ww \rVert_{\mathrm{Z}_3((0,\tau)\times\Omega)}=\lVert \ww \rVert_{\L^\infty(0,\tau;\B^{0}_{3,2}(\Omega))} + \sqrt{\nu}\lVert \ww \rVert_{\L^2(0,\tau;\B^{1}_{3,\infty}(\Omega))}\\
    &\qquad\qquad\lesssim_{\Omega,\nu,\kappa}\lVert {\uu}_1 + \uu_2\rVert_{\L^\infty(0,\tau;\B^{0}_{3,2}(\Omega))}^\frac{1}{2}\lVert \nabla( \uu_1 + \uu_2)\rVert_{\L^2(0,\tau;\B^{0}_{3,\infty}(\Omega))}^\frac{1}{2} \lVert \ww \rVert_{\mathrm{Z}_3((0,\tau)\times\Omega)}\\
    &\qquad\qquad\qquad\qquad  +\lVert \theta_1\rVert_{\L^{\infty}(0,\tau;\B^{-1}_{3/2,2}(\Omega))}^\frac{1}{2} \lVert \theta_1\rVert_{\L^{2}(0,\tau;\B^{0}_{3/2,\infty}(\Omega))}^\frac{1}{2} \lVert \ww \rVert_{\mathrm{Z}_3((0,\tau)\times\Omega)}\\
    &\qquad\qquad\qquad\qquad  +\lVert \uu_2\rVert_{\L^{\infty}(0,\tau;\B^{0}_{3,2}(\Omega))}^\frac{1}{2} \lVert \nabla \uu_2\rVert_{\L^{2}(0,\tau;\B^{0}_{3,\infty}(\Omega))}^\frac{1}{2}  \lVert \Theta \rVert_{\mathrm{K}_{3/2}((0,\tau)\times\Omega)},
\end{align*}
and similarly
\begin{align*}
    & \lVert \Theta \rVert_{\mathrm{K}_{3/2}((0,\tau)\times\Omega)}=\lVert \Theta \rVert_{\L^\infty(0,\tau;\B^{-1}_{3/2,2}(\Omega))} + \sqrt{\kappa}\lVert \Theta\rVert_{\L^2(0,\tau;\B^{0}_{3/2,\infty}(\Omega))}\\
    &\qquad\qquad\lesssim_{\Omega,\nu,\kappa}\lVert \theta_1\rVert_{\L^{\infty}(0,\tau;\B^{-1}_{3/2,2}(\Omega))}^\frac{1}{2} \lVert \theta_1\rVert_{\L^{2}(0,\tau;\B^{0}_{3/2,\infty}(\Omega))}^\frac{1}{2} \lVert \ww \rVert_{\mathrm{Z}_3((0,\tau)\times\Omega)}\\
    &\qquad\qquad\qquad\qquad  +\lVert \uu_2\rVert_{\L^{\infty}(0,\tau;\B^{0}_{3,2}(\Omega))}^\frac{1}{2} \lVert \nabla \uu_2\rVert_{\L^{2}(0,\tau;\B^{0}_{3,\infty}(\Omega))}^\frac{1}{2}  \lVert \Theta \rVert_{\mathrm{K}_{3/2}((0,\tau)\times\Omega)}.
\end{align*}
Summing up both, one obtains
\begin{align}
    &\lVert \ww \rVert_{\mathrm{Z}_3((0,\tau)\times\Omega)}+\lVert \Theta \rVert_{\mathrm{K}_{3/2}((0,\tau)\times\Omega)}\nonumber\\
    &\lesssim_{\Omega,\nu,\kappa}\lVert {\uu}_1 + \uu_2\rVert_{\L^\infty(0,\tau;\B^{0}_{3,2}(\Omega))}^\frac{1}{2}\lVert \nabla( \uu_1 + \uu_2)\rVert_{\L^2(0,\tau;\B^{0}_{3,\infty}(\Omega))}^\frac{1}{2} \lVert \ww \rVert_{\mathrm{Z}_3((0,\tau)\times\Omega)}\nonumber\\
    &\qquad\qquad  +\lVert \theta_1\rVert_{\L^{\infty}(0,\tau;\B^{-1}_{3/2,2}(\Omega))}^\frac{1}{2} \lVert \theta_1\rVert_{\L^{2}(0,\tau;\B^{0}_{3/2,\infty}(\Omega))}^\frac{1}{2} \lVert \ww \rVert_{\mathrm{Z}_3((0,\tau)\times\Omega)}\label{eq:FinalEstimateuniquenessLp}\\
    &\qquad\qquad  +\lVert \uu_2\rVert_{\L^{\infty}(0,\tau;\B^{0}_{3,2}(\Omega))}^\frac{1}{2} \lVert \nabla \uu_2\rVert_{\L^{2}(0,\tau;\B^{0}_{3,\infty}(\Omega))}^\frac{1}{2}  \lVert \Theta \rVert_{\mathrm{K}_{3/2}((0,\tau)\times\Omega)},\nonumber
\end{align}
Since $(\uu_1,\theta_1), (\uu_2,\theta_2)\in\L^\infty(0,T;\,\B^{0,\sigma}_{3,2,0}(\Omega)\times \B^{-1}_{3/2,2,0}(\Omega))\cap\L^2(0,T;\B^{1,\sigma}_{3,\infty,0}(\Omega)\times \B^{0}_{3/2,\infty}(\Omega))$, by the dominated convergence theorem,
\begin{align*}
    \left.\begin{array}{r}
         \lVert {\uu}_1 + \uu_2\rVert_{\L^\infty(0,\tau;\B^{0}_{3,2}(\Omega))}^\frac{1}{2}\lVert \nabla( \uu_1 + \uu_2)\rVert_{\L^2(0,\tau;\B^{0}_{3,\infty}(\Omega))}^\frac{1}{2}\\
         \lVert \theta_1\rVert_{\L^{\infty}(0,\tau;\B^{-1}_{3/2,2}(\Omega))}^\frac{1}{2} \lVert \theta_1\rVert_{\L^{2}(0,\tau;\B^{0}_{3/2,\infty}(\Omega))}^\frac{1}{2}\\
         \lVert \uu_2\rVert_{\L^{\infty}(0,\tau;\B^{0}_{3,2}(\Omega))}^\frac{1}{2} \lVert \nabla \uu_2\rVert_{\L^{2}(0,\tau;\B^{0}_{3,\infty}(\Omega))}^\frac{1}{2} 
    \end{array}\right\}\xrightarrow[\tau\longrightarrow 0]{}0.
\end{align*}
Therefore, there exists $\tau>0$ such that \eqref{eq:FinalEstimateuniquenessLp} becomes
\begin{align*}
    \lVert \ww \rVert_{\mathrm{Z}_3((0,\tau)\times\Omega)}+\lVert \Theta \rVert_{\mathrm{K}_{3/2}((0,\tau)\times\Omega)}<\frac{\lVert \ww \rVert_{\mathrm{Z}_3((0,\tau)\times\Omega)}+\lVert \Theta \rVert_{\mathrm{K}_{3/2}((0,\tau)\times\Omega)}}{2},
\end{align*}
which implies $(\ww,\Theta)=0$ in $[0,\tau)$. This ends the proof for uniqueness.
\end{proof}

The next uniqueness result simplifies significantly due to the $\L^1$-in-time maximal regularity.

\begin{theorem}\label{thm:UniquenessWeakBesov} Let $\Omega\subset\RR^3$ be a bounded $\C^{1,\alpha}$ domain with $\alpha>1/3$ and $T\in(0,\infty]$. Let $(\uu_0,\theta_0)\in\B^{0,\sigma}_{3,1,0}(\Omega)\times\B^{-1}_{3/2,1,0}(\Omega)$, and let
\begin{align*}
    (\uu_1,\theta_1),\quad (\uu_2,\theta_2)\in\L^\infty(0,T;&\B^{0,\sigma}_{3,1,0}(\Omega)\times \B^{-1}_{3/2,1,0}(\Omega))\cap\L^1(0,T;\D^{0}_{3,1}(\mathbb{A}_\mathcal{D},\Omega)\times \B^{1}_{3/2,1}(\Omega))
\end{align*}
such that both satisfies the system \eqref{eq:BoussinesqSystemIntro} on $(0,T)$ with initial value $(\uu_0,\theta_0)$ in the sense of \eqref{eq:WeakFormBSQBesovVelo}--\eqref{eq:WeakFormBSQBesovTemp}. Then, it holds
\begin{align*}
   \forall s\in[0,T),\quad (\uu_1,\theta_1)(s) = (\uu_2,\theta_2)(s).
\end{align*}
\end{theorem}

\begin{proof} \textbf{Step 1:} The proofs begin similarly to the one of Theorem~\ref{thm:UniquenessWeakLp}: with the same notations, we show that $(\ww,\Theta)=(\uu_1-\uu_2,\theta_1-\theta_2)$ is a mild solution of an appropriate problem. 

First notice that by linearity, it holds
\begin{align*}
    (\ww,\Theta)\in\L^\infty(0,T;\B^{0,\sigma}_{3,1,0}(\Omega)\times \B^{-1}_{3/2,1,0}(\Omega))\cap\L^1(0,T;\D^{0}_{3,1}(\mathbb{A}_\mathcal{D},\Omega)\times \B^{1}_{3/2,1}(\Omega)).
\end{align*}
It also holds, as in the proof Theorem~\ref{thm:UniquenessWeakLp}, that
\begin{align*}
    \langle \Theta,\mathbf{1}\rangle_{\Omega}=0. 
\end{align*}

By the isomorphism property \eqref{eq:SQRTStokesBesovLip},  the interpolation inequality \eqref{eq:ProofInterpLinftyL1(A)} and the boundedness of $\Omega$, it holds that $\ww\in\L^{2}(0,T;\B^{1,\sigma}_{3,1,0}(\Omega))\hookrightarrow\L^{2}(0,T;\W^{1,2}_{0,\sigma}(\Omega))$. Moreover, and since $\Omega$ is bounded, it holds $\B^{0,\sigma}_{3,1,0}(\Omega)\hookrightarrow\L^{2}_{\mathfrak{n},\sigma}(\Omega)$. Altogether, we have obtained
\begin{align*}
    \ww\in \L^{\infty}(0,T;\L^{2}_{\mathfrak{n},\sigma}(\Omega))\cap\L^{2}(0,T;\W^{1,2}_{0,\sigma}(\Omega)).
\end{align*}
Additionally, following the argument from the proof of Proposition~\ref{Prop:SolutionOperatorBoussinesqBesov}, Step 2, since  $\uu_1+\uu_2, \ww \in \L^{2}(0,T;\B^{1,\sigma}_{3,1,0}(\Omega))$, we deduce
\begin{align*}
    \lVert \div([\uu_1+\uu_2]\otimes \ww ])\rVert_{\L^1(0,T;\L^2(\Omega))} &\lesssim_{\Omega}\lVert \div([\uu_1+\uu_2]\otimes \ww )\rVert_{\L^1(0,T;\B^{0}_{3,1}(\Omega))} \\
    &\lesssim_{\Omega}\lVert [\uu_1+\uu_2]\otimes \ww \rVert_{\L^1(0,T;\B^{1}_{3,1}(\Omega))} \\&\lesssim_{\Omega}\lVert [\uu_1+\uu_2] \rVert_{\L^2(0,T;\B^{1}_{3,1}(\Omega))} \lVert \ww\rVert_{\L^2(0,T;\B^{1}_{3,1}(\Omega))} ,
\end{align*}
due to the fact that $\B^{1}_{3,1}$ is an algebra since we are in dimension $3$. Additionally, by the  Sobolev embedding $\B^{0}_{3/2,1}\hookrightarrow\H^{-1,2}$ and interpolation inequalities it holds
\begin{align*}
    \lVert \Theta\rVert_{\L^2(0,T;\W^{-1,2}(\Omega))} &\lesssim_{\Omega}\lVert \Theta\rVert_{\L^2(0,T;\B^{0}_{3/2,1}(\Omega))} 
    \lesssim_{\Omega}\lVert \Theta\rVert_{\L^\infty(0,T;\B^{-1}_{3/2,1}(\Omega))}^\frac{1}{2}\lVert \Theta\rVert_{\L^1(0,T;\B^{1}_{3/2,1}(\Omega))}^\frac{1}{2}.
\end{align*}
In particular, we isolate the fact that
\begin{align}\label{eq:InterpIneqTemp}
    \L^1(0,T;\B^{1}_{3/2,1,\circ}(\Omega))\cap \L^\infty(0,T;\B^{-1}_{3/2,1,\circ}(\Omega))\hookrightarrow\L^2(0,T;\B^{0}_{3/2,1,\circ}(\Omega))
\end{align}
\medbreak
\noindent Since $\Theta \mathfrak{e}_3\in \L^2(0,T;\W^{-1,2}_\circ(\Omega,\RR^3))$, there exists $F\in\L^{2}(0,T;\L^2(\Omega)^9)$ such that $$\div (F) = \Theta.$$ Now, we summarize our current knowledge concerning the Stokes part of the problem:
\begin{itemize}
    \item $\ww\in \L^{\infty}(0,T;\L^{2}_{\mathfrak{n},\sigma}(\Omega))\cap\L^{2}(0,T;\W^{1,2}_{0,\sigma}(\Omega));$
    \item $\div([\uu_1+\uu_2]\otimes \ww ) \in\L^{1}(0,T;\L^2(\Omega,\RR^3))$ and $F\in\L^2(0,T;\L^2(\Omega)^9)$.
\end{itemize}
Therefore, by \cite[Thm~2.4.1]{SohrBook2001}, $\ww$ is the \textbf{mild} solution to the linear Stokes evolution problem with forcing term $-\div([\uu_1+\uu_2]\otimes \ww )+\div(F)=-\div([\uu_1+\uu_2]\otimes \ww )+\Theta\mathfrak{e}_3$ and initial value $0$. The identity \eqref{eq:ProofThetaUniquenessVelo} remains valid in our framework. Additionally, by Propositions~\ref{thm:LqMaxRegStokesBesov}~and~\ref{Prop:SolutionOperatorBoussinesqBesov}, one obtains
\begin{align*}
    \ww\in\C^{0}([0,T);\B^{0,\sigma}_{3,1,0}(\Omega))\cap\L^1(0,T;\D^{0}_{3,1}(\mathbb{A}_\mathcal{D},\Omega)).
\end{align*}
Notice that by \eqref{eq:InterpIneqTemp}, that $\Theta\in\L^2(0,T;\, \B^{0}_{3/2,1,\circ}(\Omega))\hookrightarrow\L^2(0,T;\, \L^{3/2}_\circ(\Omega))$ and that we also have $\Theta\in\L^\infty(0,T;\, \B^{-1}_{3/2,1,\circ}(\Omega))\hookrightarrow\L^\infty(0,T;\, \B^{-1}_{3/2,2,\circ}(\Omega))$. Therefore, the analysis that lead to \eqref{eq:UniqunessL2MildSpaceTheta} in the proof of Theorem~\ref{thm:UniquenessWeakLp} remains valid, and it holds
\begin{align*}
    \Theta\in\L^\infty(0,T;\, (-\Delta_\mathcal{N})^{3/4}\L^{2}_{\circ}(\Omega))\cap\L^2(0,T;(-\Delta_\mathcal{N})^{3/4}\W^{1,2}_\circ(\Omega)).
\end{align*}
Starting from \eqref{eq:membershipForcingTermThetaProofUniqueLp} (this is where comes from the $\C^{1,\alpha}$-assumption on the boundary regularity), we have that \begin{align}\label{eq:ClaimProofUniqForcingTermThetaBesov}
    (-\Delta_\mathcal{N})^{-3/4}[\div (\uu_2\Theta) + \div(\ww \theta_1)]\in\L^{2}(0,T;\W^{-1,2}_\circ(\Omega)).
\end{align}
Indeed, by the Sobolev embedding and the product rule from Corollary~\ref{prop:ProductRullBesovDim3-2}, one obtains 
\begin{align*}
    &\lVert (-\Delta_\mathcal{N})^{-3/4}[\div (\uu_2\Theta) + \div(\ww \theta_1)]\rVert_{\L^2(0,T;\W^{-1,2}(\Omega))}\\
    &\lesssim_{\Omega,\kappa} \lVert \uu_2 \Theta\rVert_{\L^2(\RR_+;\B^{-1/2}_{6/5,\infty}(\Omega))}+\lVert \ww \theta_1\rVert_{\L^2(\RR_+;\B^{-1/2}_{6/5,\infty}(\Omega))}\nonumber\\
    &\lesssim_{\Omega,\kappa} \Big\lVert \lVert\uu_2\rVert_{\B^{1/2}_{3,1}(\Omega)}\lVert\Theta\rVert_{\B^{-1/2}_{3/2,1}(\Omega)}\Big\rVert_{\L^2(\RR_+)} + \Big\lVert \lVert\ww\rVert_{\B^{1/2}_{3,1}(\Omega)}\lVert\theta_1\rVert_{\B^{-1/2}_{3/2,1}(\Omega)}\Big\rVert_{\L^2(\RR_+)}.
\end{align*}
However, for $(\vv,\eta)\in \{(\uu_2,\Theta),(\ww,\theta_1)\}$, by interpolation inequalities, and H\"{o}lder's inequality, it holds
\begin{align*}
    &\Big\lVert \lVert\vv\rVert_{\B^{1/2}_{3,1}(\Omega)}\lVert\eta\rVert_{\B^{-1/2}_{3/2,1}(\Omega)}\Big\rVert_{\L^2(\RR_+)}\\ &\lesssim_{\Omega} \Big\lVert \lVert\vv\rVert_{\B^{0}_{3,1}(\Omega)}^\frac{1}{2}\lVert\vv\rVert_{\B^{1}_{3,1}(\Omega)}^\frac{1}{2}\lVert\eta\rVert_{\B^{-1}_{3/2,1}(\Omega)}^\frac{3}{4}\lVert\eta\rVert_{\B^{1}_{3/2,1}(\Omega)}^\frac{1}{4}\Big\rVert_{\L^2(\RR_+)} \\
    &\lesssim_{\Omega} \Big\lVert \lVert\vv\rVert_{\B^{0}_{3,1}(\Omega)}^\frac{1}{2}\lVert\vv\rVert_{\B^{1}_{3,1}(\Omega)}^\frac{1}{2}\Big\rVert_{\L^4(\RR_+)} \Big\lVert \lVert\eta\rVert_{\B^{-1}_{3/2,1}(\Omega)}^\frac{3}{4}\lVert\eta\rVert_{\B^{1}_{3/2,1}(\Omega)}^\frac{1}{4}\Big\rVert_{\L^4(\RR_+)}\\
    &\lesssim_{\Omega} \lVert\vv\rVert_{\L^\infty(0,T;\B^{0}_{3,1}(\Omega))}^\frac{1}{2}\lVert\vv\rVert_{\L^2(0,T;\B^{1}_{3,1}(\Omega))}^\frac{1}{2} \lVert\eta\rVert_{\L^\infty(0,T;\B^{-1}_{3/2,1}(\Omega))}^\frac{3}{4}\lVert\eta\rVert_{\L^1(0,T;\B^{1}_{3/2,1}(\Omega))}^\frac{1}{4}<\infty.
\end{align*}
Thus, the claim \eqref{eq:ClaimProofUniqForcingTermThetaBesov} is valid. Consequently, for the difference of temperatures $\Theta$, we have obtained
\begin{itemize}
    \item $\Theta\in\L^\infty(0,T;\, (-\Delta_\mathcal{N})^{3/4}\L^{2}_{\circ}(\Omega))\cap\L^2(0,T;(-\Delta_\mathcal{N})^{3/4}\W^{1,2}_\circ(\Omega))$;
    \item $g:=\div (\uu_2\Theta) + \div(\ww \theta_1)\in\L^2(0,T;(-\Delta_\mathcal{N})^{3/4}\W^{-1,2}_\circ(\Omega))$
\end{itemize}
so that $\Theta$ is a weak variational solution to the Heat Neumann equation with forcing term $g$ and initial value $0$. Hence, $\Theta$ is the  \textbf{mild} solution belonging to $\C^{0}([0,T);\, (-\Delta_\mathcal{N})^{3/4}\L^{2}_{\circ}(\Omega))\cap\L^2(0,T;(-\Delta_\mathcal{N})^{3/4}\W^{1,2}_\circ(\Omega))$.

\textbf{Step 2:} We show that $(\ww,\Theta)=(0,0)$ on $(0,T)$. Since we have proved that $(\ww,\Theta)$ are appropriate mild solutions, we can write
\begin{align*}
    \ww &= \B_\nu(\ww,\uu_1+\uu_2) + \L_\nu(\Theta)\\
    \Theta&= \C_\kappa(\uu_2,\Theta) + \C_\kappa(\ww,\theta_1).
\end{align*}
Let $t\in[0,T)$. By $\L^1$-maximal regularity, Proposition~\ref{thm:LqMaxRegStokesBesov}, it holds that
\begin{align}
     \lVert \ww \rVert_{\mathrm{U}_{3}((0,t)\times\Omega)} &\lesssim_{\Omega} \int_{0}^{t} \lVert \PP_{\Omega}\div([\uu_1(s)+\uu_2(s)]\otimes \ww (s))\rVert_{\B^{0}_{3,1}(\Omega)} \d s +\int_{0}^{t} \lVert \PP_{\Omega}[\Theta(s) \mathfrak{e}_3]\rVert_{\B^{0}_{3,1}(\Omega)} \d s\nonumber\\
     &\lesssim_{\Omega} \int_{0}^{t} \lVert [\uu_1(s)+\uu_2(s)]\otimes \ww (s)\rVert_{\B^{1}_{3,1}(\Omega)} \d s +\int_{0}^{t} \lVert \Theta(s) \rVert_{\B^{0}_{3,1}(\Omega)} \d s\nonumber\\
     &\lesssim_{\Omega} \int_{0}^{t} \lVert [\uu_1(s)+\uu_2(s)]\rVert_{\B^{1}_{3,1}(\Omega)}\lVert\ww (s)\rVert_{\B^{1}_{3,1}(\Omega)} \d s +\int_{0}^{t} \lVert \Theta(s) \rVert_{\B^{1}_{3/2,1}(\Omega)} \d s\label{eq:RefGronwallL1P1Velo}.
\end{align}
The last line is obtained by the Sobolev embedding $\B^{1}_{3/2,1}\hookrightarrow\B^{0}_{3,1}$. By the isomorphism property \eqref{eq:SQRTStokesBesovLip}, the moment inequality $\lVert A^{1/2} x\rVert_{\X}\lesssim \lVert  x\rVert_{\X}^\frac{1}{2}\lVert A x\rVert_{\X}^\frac{1}{2}$, then Young's Inequality $ab \leqslant \frac{a^2}{2\varepsilon}+\frac{\varepsilon b^2}{2}$, one deduces
\begin{align}
     &\int_{0}^{t} \lVert [\uu_1(s)+\uu_2(s)]\rVert_{\B^{1}_{3,1}(\Omega)}\lVert\ww (s)\rVert_{\B^{1}_{3,1}(\Omega)} \d s\nonumber\nonumber\\
     &\lesssim_{\Omega} \int_{0}^{t} \lVert [\uu_1(s)+\uu_2(s)]\rVert_{\B^{1}_{3,1}(\Omega)}\lVert \AA_\mathcal{D}^{1/2}\ww (s)\rVert_{\B^{0}_{3,1}(\Omega)} \d s \nonumber\\
     &\lesssim_{\Omega} \int_{0}^{t} \lVert [\uu_1(s)+\uu_2(s)]\rVert_{\B^{1}_{3,1}(\Omega)}\lVert \ww (s)\rVert_{\B^{0}_{3,1}(\Omega)}^\frac{1}{2}\lVert \AA_\mathcal{D} \ww (s)\rVert_{\B^{0}_{3,1}(\Omega)}^\frac{1}{2} \d s \nonumber\\
     &\lesssim_{\Omega} \frac{1}{\varepsilon}\int_{0}^{t} \lVert [\uu_1(s)+\uu_2(s)]\rVert_{\B^{1}_{3,1}(\Omega)}^2   \lVert \ww (s)\rVert_{\B^{0}_{3,1}(\Omega)} \d s + {\varepsilon}\lVert \AA_\mathcal{D} \ww (s)\rVert_{\L^1(0,t;\B^{0}_{3,1}(\Omega))} \nonumber \\
     &\lesssim_{\Omega,\nu} \frac{1}{\varepsilon}\int_{0}^{t} \lVert [\uu_1(s)+\uu_2(s)]\rVert_{\B^{1}_{3,1}(\Omega)}^2   \lVert \ww (s)\rVert_{\B^{0}_{3,1}(\Omega)} \d s + {\varepsilon} \lVert \ww \rVert_{\mathrm{U}_{3}((0,t)\times\Omega)},\qquad\forall\varepsilon>0\label{eq:RefGronwallL1P2Velo}.
\end{align}
Plugging \eqref{eq:RefGronwallL1P2Velo} into \eqref{eq:RefGronwallL1P1Velo}, then choosing $\varepsilon$ small enough, it yields
\begin{align}
    \lVert \ww \rVert_{\mathrm{U}_{3}((0,t)\times\Omega)} \lesssim_{\Omega,\nu}\int_{0}^{t} \lVert [\uu_1(s)+\uu_2(s)]\rVert_{\B^{1}_{3,1}(\Omega)}^2   \lVert \ww (s)\rVert_{\B^{0}_{3,1}(\Omega)} \d s +\int_{0}^{t} \lVert \Theta(s) \rVert_{\B^{1}_{3/2,1}(\Omega)} \d s.\label{eq:GronwallL1P1} 
\end{align}
On the other hand, by $\L^1$-maximal regularity for the Neumann Laplacian, Proposition~\ref{thm:LqMaxRegNeumannBesov}, and Corollary~\ref{cor:ProductRullBesovDim3-1},  it holds
\begin{align*}
    &\lVert \Theta \rVert_{\mathrm{T}_{3/2}((0,t)\times\Omega)} \\&\lesssim_{\Omega} \int_{0}^t \lVert \uu_2(s) \Theta(s)\rVert_{\B^{0}_{3/2,1}(\Omega)} \d s +\int_{0}^t \lVert \ww(s) \theta_1(s)\rVert_{\B^{0}_{3/2,1}(\Omega)} \d s\\
    &\lesssim_{\Omega} \int_{0}^t \lVert \uu_2(s) \rVert_{\B^{1}_{3,1}(\Omega)}\lVert \Theta(s)\rVert_{\B^{0}_{3/2,1}(\Omega)} \d s +\int_{0}^t \lVert \ww(s) \rVert_{\B^{1}_{3,1}(\Omega)} \lVert  \theta_1(s)\rVert_{\B^{0}_{3/2,1}(\Omega)} \d s.
\end{align*}
Consequently, one can proceed exactly as for \eqref{eq:RefGronwallL1P2Velo}, in order to obtain
\begin{align*}
    &\lVert \Theta \rVert_{\mathrm{T}_{3/2}((0,t)\times\Omega)} \\
    &\lesssim_{\Omega,\kappa, \nu} \frac{1}{\varepsilon'}\int_{0}^t \lVert \uu_2(s) \rVert_{\B^{1}_{3,1}(\Omega)}^2\lVert \Theta(s)\rVert_{\B^{-1}_{3/2,1}(\Omega)} \d s  + \varepsilon'\lVert \Theta\rVert_{\L^{1}(0,t;\B^{-1}_{3/2,1}(\Omega))}\\
    &\qquad\qquad+ \frac{1}{\varepsilon}\int_{0}^t \lVert \ww(s) \rVert_{\B^{0}_{3,1}(\Omega)} \lVert  \theta_1(s)\rVert_{\B^{0}_{3/2,1}(\Omega)}^2 \d s + \varepsilon \lVert \AA_\mathcal{D} \ww\rVert_{\L^{1}(0,t;\B^0_{3,1}(\Omega))},\\
    &\lesssim_{\Omega,\kappa, \nu} \frac{1}{\varepsilon'}\int_{0}^t \lVert \uu_2(s) \rVert_{\B^{1}_{3,1}(\Omega)}^2\lVert \Theta(s)\rVert_{\B^{-1}_{3/2,1}(\Omega)} \d s  + \varepsilon'\lVert \Theta \rVert_{\mathrm{T}_{3/2}((0,t)\times\Omega)}\\
    &\qquad\qquad+ \frac{1}{\varepsilon}\int_{0}^t \lVert \ww(s) \rVert_{\B^{0}_{3,1}(\Omega)} \lVert  \theta_1(s)\rVert_{\B^{0}_{3/2,1}(\Omega)}^2 \d s + \varepsilon \lVert \ww \rVert_{\mathrm{U}_{3}((0,t)\times\Omega)} ,\qquad \forall \varepsilon,\varepsilon'>0.
\end{align*}
Choosing $\varepsilon'$ sufficiently small, we are left with the inequality
\begin{align}
     &\lVert \Theta \rVert_{\mathrm{T}_{3/2}((0,t)\times\Omega)}\nonumber\\ 
    &\lesssim_{\Omega,\kappa, \nu} \int_{0}^t \lVert \uu_2(s) \rVert_{\B^{1}_{3,1}(\Omega)}^2\lVert \Theta(s)\rVert_{\B^{-1}_{3/2,1}(\Omega)} \d s \label{eq:GronwallL1P2} \\
    &\qquad\qquad+ \frac{1}{\varepsilon}\int_{0}^t \lVert \ww(s) \rVert_{\B^{0}_{3,1}(\Omega)} \lVert  \theta_1(s)\rVert_{\B^{0}_{3/2,1}(\Omega)}^2 \d s + \varepsilon \lVert \ww \rVert_{\mathrm{U}_{3}((0,t)\times\Omega)} ,\qquad \forall \varepsilon>0.\nonumber
\end{align}
Now, we take the sum of \eqref{eq:GronwallL1P2} with \eqref{eq:GronwallL1P1}, so that choosing $\varepsilon$ small enough again, we reach the estimate,
\begin{align*}
     \lVert \Theta \rVert_{\mathrm{T}_{3/2}((0,t)\times\Omega)} +  \lVert \ww \rVert_{\mathrm{U}_{3}((0,t)\times\Omega)} \lesssim_{\Omega,\kappa,\nu} \int_{0}^{t} \varphi(s)[\lVert \ww(s) \rVert_{\B^{0}_{3,1}(\Omega)}+ \lVert \ww(s) \rVert_{\B^{-1}_{3/2,1}(\Omega)}] \d s
\end{align*}
where $\varphi(s):=1+\lVert  \theta_1(s)\rVert_{\B^{0}_{3/2,1}(\Omega)}^2+\lVert \uu_2(s) \rVert_{\B^{1}_{3,1}(\Omega)}^2+\lVert \uu_1(s)+\uu_2(s) \rVert_{\B^{1}_{3,1}(\Omega)}^2$. Finally, we have obtained
\begin{align*}
    \lVert \ww(t) \rVert_{\B^{0}_{3,1}(\Omega)}+ \lVert \Theta(t) \rVert_{\B^{-1}_{3/2,1}(\Omega)} \lesssim_{\Omega,\kappa,\nu} \int_{0}^{t} \varphi(s)[\lVert \ww(s) \rVert_{\B^{0}_{3,1}(\Omega)}+ \lVert \Theta(s) \rVert_{\B^{-1}_{3/2,1}(\Omega)}] \d s, \quad \forall t\in[0,T).
\end{align*}
This, by Gr\"{o}nwall's inequality, implies $\lVert \ww(t) \rVert_{\B^{0}_{3,1}(\Omega)}+ \lVert \Theta(t) \rVert_{\B^{-1}_{3/2,1}(\Omega)} =0$ for all $t\in[0,T)$, which ends the proof.
\end{proof}

As the second part, Step 2, of the proof of Theorem~\ref{thm:UniquenessWeakBesov} only relies on the fact that solutions are mild solutions, we obtain the following result:
\begin{proposition}\label{prop:UniqMild}Let $\Omega\subset\RR^3$ be a bounded Lipschitz domain and $T\in(0,\infty]$. Let $(\uu_0,\theta_0)\in\B^{0,\sigma}_{3,1,0}(\Omega)\times\B^{-1}_{3/2,1,0}(\Omega)$, and let
\begin{align*}
    (\uu_1,\theta_1),\quad (\uu_2,\theta_2)\in\L^\infty(0,T;&\B^{0,\sigma}_{3,1,0}(\Omega)\times \B^{-1}_{3/2,1,0}(\Omega))\cap\L^1(0,T;\D^{0}_{3,1}(\mathbb{A}_\mathcal{D},\Omega)\times \B^{1}_{3/2,1,0}(\Omega))
\end{align*}
such that both are mild solutions to the system \eqref{eq:BoussinesqSystemIntro} on $(0,T)$ with initial value $(\uu_0,\theta_0)$ in the sense of \eqref{eq:ModifAbstractBoussinesq}. Then, it holds that
\begin{align*}
   \forall s\in[0,T),\quad (\uu_1,\theta_1)(s) = (\uu_2,\theta_2)(s).
\end{align*}
\end{proposition}

\subsection{Exponential stabilization towards the equilibrium.}\label{Sec:Asymptotics}

\begin{proposition}Let $\Omega\subset\RR^3$ be a bounded $\C^{1,\alpha}$ domain with $\alpha>1/3$. Let $(\uu_0,\theta_0)\in\B_{3,2,0}^{0,\sigma}(\Omega)\times \B^{-1}_{3/2,2,0}(\Omega)$ such that $(\uu_0,\theta_0-\tilde{\theta}_0)$ has small norm, where $\tilde{\theta}_0 = \frac{1}{|\Omega|} \int_{\Omega} \theta_0$. There exits $\lambda>0$, independent of the initial data, such that the unique mild solution $(\uu,\theta)\in\C^{0}_{ub}(\overline{\RR_+},\B_{3,2,0}^{0,\sigma}(\Omega)\times \B^{-1}_{3/2,2,0}(\Omega))$, given by Theorems~\ref{thm:LargeTimesExistenceLp}~and~\ref{thm:UniquenessWeakLp}, satisfies
\begin{align*}
    e^{\lambda t}\lVert (\uu(t),\theta(t))- (0,\Tilde{\theta}_0) \rVert_{\B^{0}_{3,2}(\Omega)\times \B^{-1}_{3/2,2}(\Omega)} \xrightarrow[t\rightarrow+\infty]{} 0.
\end{align*}
\end{proposition}

\begin{proof} By the uniqueness result Theorem~\ref{thm:UniquenessWeakLp}, and by the analysis performed in the proof of Theorem~\ref{thm:LargeTimesExistenceLp}, it holds that $(\uu,\theta)$ is the mild solution with initial data $(\uu_0,\theta_0)$ if and only if $(\uu,\theta-\tilde{\theta}_0)$ is the mild solution with initial data $(\uu_0,\theta_0-\tilde{\theta}_0)$. Since $\theta_0-\tilde{\theta}_0$ is mean free, we can apply the exponential-weighted maximal regularity estimates from Propositions~\ref{thm:LqMaxRegNeumannBesov}~and~\ref{thm:LqMaxRegStokesBesov}, and we deduce that there exists $c>0$, such that
\begin{align*}
    e^{ct}\lVert (\uu(t),\theta(t))- (0,\Tilde{\theta}_0) \rVert_{\B^{0}_{3,2}(\Omega)\times \B^{-1}_{3/2,2}(\Omega)} \lesssim_{\nu,\kappa,\Omega} 1, \qquad\forall t\geqslant 0.
\end{align*}
Choosing $\lambda := c/2$ yields the result.
\end{proof}

\begin{proposition}Let $\Omega\subset\RR^3$ be a bounded Lipschitz domain. Let $(\uu_0,\theta_0)\in\B_{3,1,0}^{0,\sigma}(\Omega)\times \B^{-1}_{3/2,1,0}(\Omega)$ such that $(\uu_0,\theta_0-\tilde{\theta}_0)$ has small norm, where $\tilde{\theta}_0 = \frac{1}{|\Omega|} \int_{\Omega} \theta_0$. There exits $\lambda>0$, independent of the initial data, such that the unique mild solution $(\uu,\theta)\in\C^{0}_{ub}(\overline{\RR_+},\B_{3,1,0}^{0,\sigma}(\Omega)\times \B^{-1}_{3/2,1,0}(\Omega))$, given by Theorem~\ref{thm:LargeTimesExistenceBesov}~and~Proposition~\ref{prop:UniqMild}, satisfies
\begin{align*}
    e^{\lambda t}\lVert (\uu(t),\theta(t))- (0,\Tilde{\theta}_0) \rVert_{\B^{0}_{3,1}(\Omega)\times \B^{-1}_{3/2,1}(\Omega)} \xrightarrow[t\rightarrow+\infty]{} 0.
\end{align*}
\end{proposition}

\begin{proof} The proof is exactly the same as the one of the previous Proposition, but instead we take advantage of  Proposition~~\ref{Prop:SolutionOperatorBoussinesqBesov},  Theorem~\ref{thm:LargeTimesExistenceBesov}, and the exponential weighted estimate from Propositions~\ref{thm:LqMaxRegNeumannBesov}~and~\ref{thm:LqMaxRegStokesBesov}.
\end{proof}

%----------------------------------------------------
%--------------------- Appendix ---------------------
%----------------------------------------------------
%----------------------------------------------------
%------------------- Appendices ----------------------
%----------------------------------------------------

\appendix

\section{Appendix: Product rules in Besov spaces}\label{App:ProductRules}

\begin{proposition}[ {\cite[Appendix]{DanchinHieberMuchaTolk2020}} ]\label{prop:Productq=1Besov} Let $\Omega$ be a bounded Lipschitz domain of $\RR^n$. Let $(s_1,s_2)\in\RR^2$ and $(p,q,r)\in[1,\infty]^3$ satisfy
$$s_2\leqslant  s_1\leqslant  \min\biggl(\frac np,\frac nq\biggr),\quad
s_1+s_2> 0,\quad\frac1p+\frac1q\leqslant 1\quad\&\quad
\frac1r=\frac1p+\frac1q-\frac{s_1}n\cdotp$$
Then
\begin{equation}\label{eq:prod1} \B^{s_1}_{p , 1}(\Omega) \cdot \B^{s_2}_{q , 1}(\Omega) \subset \B^{s_2}_{r , 1}(\Omega).\end{equation}
If in the limit case $s_1+s_2=0$, 
then
\begin{equation}\label{eq:prod2} \B^{s_1}_{p , 1}(\Omega) \cdot \B^{s_2}_{q , 1}(\Omega) \subset  \B^{s_2}_{r , \infty}(\Omega).\end{equation}
\end{proposition}

\begin{corollary}\label{cor:ProductRullBesovDim3-1} Let $\Omega$ be a Lipschitz domain of $\RR^3$. One has the product rule
\begin{equation*} \B^{1}_{3, 1}(\Omega) \cdot \B^{0}_{3/2, 1}(\Omega) \subset \B^{0}_{3/2, 1}(\Omega),\end{equation*}
\end{corollary}

\begin{proof}Just set the right parameters in Proposition~\ref{prop:Productq=1Besov}, as a special case of \eqref{eq:prod1}.
\end{proof}

\begin{proposition}\label{prop:ProductRullBesovDim3-2}  Let $\Omega$ be a Lipschitz domain of $\RR^n$, $p\in[1,\infty]$ and $s\in[0,n/p]$. One has the product rule
\begin{equation}\label{eq:MainProductRule} \B^{s}_{p, 1}(\Omega) \cdot \B^{-s}_{p', 1}(\Omega) \subset \B^{-s}_{(\frac{n}{s})' ,\infty}(\Omega).\end{equation}
\end{proposition}

\begin{proof} Just set the right parameters in Proposition~\ref{prop:Productq=1Besov}, as a special case of \eqref{eq:prod1}.

% We provide the proof for $\Omega=\RR^n$, the case of domains will follow from the definition of function spaces by restriction.

% First, since $\B^{n/p}_{p,1}(\RR^n)$ is an algebra, one has the product rule
% \begin{align}\label{eq:FirstProductRuleProof}
%     \B^{n/p}_{p,1}(\RR^n)\cdot \B^{n/p}_{p,1}(\RR^n)\hookrightarrow \B^{n/p}_{p,1}(\RR^n)\hookrightarrow \mathcal{B}^{n/p}_{p,\infty}(\RR^n),
% \end{align}
% and
% \begin{align}\label{eq:SecondProductRuleProof}
%     \B^{0}_{p,1}(\RR^n)\cdot \B^{0}_{\infty,1}(\RR^n)\hookrightarrow \L^p(\RR^n)\cdot \C^0_{ub}(\RR^n)\hookrightarrow \L^p(\RR^n)\hookrightarrow\mathcal{B}^{0}_{p,\infty}(\RR^n).
% \end{align}
% By bilinear complex interpolation between \eqref{eq:FirstProductRuleProof} and \eqref{eq:SecondProductRuleProof}, setting $\theta\frac{n}{p}=:s$ where $\theta\in[0,1]$, one obtains
% \begin{align*}
%     \B^{s}_{p,1}(\RR^n)\cdot \B^{s}_{\frac{n}{s},1}(\RR^n) \hookrightarrow \mathcal{B}^{s}_{p,\infty}(\RR^n).
% \end{align*}
% Now, by duality, \eqref{eq:MainProductRule} holds.
\end{proof}

\section{Appendix: Some recalls on operator theory}\label{App:OperatorTheory}

We gather here some notions and results related to operator theory such as Bounded Imaginary Powers, (bounded) $\mathbf{H}^\infty$-functional calculus, and their links to the standard $\L^q$-maximal regularity theory.

We consider $(\D(A),A)$ a sectorial operator of angle $0<\omega<\pi$ on a Banach space $\X$, following the knowledge from Section~\ref{Sec:SectOp}, \textit{i.e.} satisfying \eqref{eq:SectCondtn}. We denote by $\D(A^\infty)$ the intersection of the family $(\D(A^k))_{k\in\NN}$. When $\omega<\pi/2$, note that for all $t>0$, $e^{-tA}\X\subset \D(A^\infty)$. If $\N(A)$ is the null space of $A$, one has $e^{-tA}{}_{|_{\N(A)}} = \I_{\N(A)}$ for all $t\geqslant 0$.

One of the main features of sectorial operators concerns the possibility of defining the Dunford functional calculus, that is, making sense of the quantity
\begin{align}\label{eq:DunfordintegralFuncCalc}
    f(A)x = \frac{1}{2 i \pi}\int_{\partial \Sigma_{\theta}} f(z)(z\I-A)^{-1}x\,\mathrm{d}z\text{,}
\end{align}
for all $x\in \X$, provided $\theta,\mu \in(\omega,\pi)$ and
$f\in\L^1(\Sigma_{\theta}, \mathrm{d}|z|/|z| )\cap\mathbf{H}^\infty({\Sigma}_{\mu})$, and $\theta<\mu$. Here, $\mathbf{H}^\infty(\Sigma_\theta)$ stands for the set of holomorphic functions on $\Sigma_\theta$ which are uniformly bounded. In particular, for the following class of decaying holomorphic functions
\begin{align*}
    \mathbf{H}^\infty_0(\Sigma_\mu):=\Big\{\,\psi\in\mathbf{H}^\infty({\Sigma}_\mu)\,:\exists\varepsilon_0>0,\,\forall z \in\Sigma_\mu,\,|\psi(z)|\lesssim_{\mu,\psi,\varepsilon_0}\frac{|z|^{\varepsilon_0}}{(1+|z|)^{2\varepsilon_0}}\,\Big\},
\end{align*}
it is easy to check that \eqref{eq:DunfordintegralFuncCalc} makes sense for any $\omega$-sectorial operator $(\D(A),A)$ on $\X$ and for any $x\in\X$, and yields a bounded linear operator. One says that $A$ admits a \textit{\textbf{bounded}} (or $\mathbf{H}^\infty(\Sigma_\mu)$-)\textit{\textbf{holomorphic functional calculus}} on $\X$ (of angle $\mu\in(\omega,\pi)$) if, there exists a constant $K_\mu$ such that for all $f\in\mathbf{H}^\infty(\Sigma_\mu)$,
\begin{align*}
    \lVert f(A)\rVert_{\X\rightarrow \X}\leqslant K_\mu \lVert f \rVert_{\L^\infty(\Sigma_\mu)}\text{.}
\end{align*}
In that case to be properly define we need $A$ to be injective, and the knowledge of the inequality above for all elements of $\mathbf{H}^\infty_0(\Sigma_\mu)$ is sufficient. If $0\in\sigma(A)$ and $A$ has a non-trivial null space, in order to be able to consider the functional calculus, we restrict the action of $A$ to the closed subspace $\overline{\R(A)}$ of $\X$.

Note that this property implies in particular that for all $s\in\RR$, one can set $f(z)=z^{is}$, so that the fractional imaginary power of $A$, $A^{is}$, is a well-defined linear operator on $\X$ (or on $\overline{\R(A)}$ if $A$ is not injective). This is called the property of \emph{Bounded Imaginary Powers} (BIP).

Furthermore, if $A$ has bounded functional calculus of angle $\mu\in(\omega,\pi)$, one has
\begin{align*}
      \theta_A :=\inf \left\{ \nu\geqslant 0\,:\, \sup_{s\in\RR} e^{-\nu|s|}\lVert A^{is}\rVert_{\X \rightarrow \X}<\infty\right\} \leqslant \mu,
\end{align*}
and $\theta_A$ is called the \emph{type} of $A$.

Bounded Imaginary Powers have two very important properties that are of interest for us:

\medbreak

\noindent \textbf{First}, it implies that the family if domains of fractional powers $(\D(A^\alpha))_\alpha$ forms a complex interpolation scale:

\medbreak

\begin{proposition}[ {\cite[Theorem~6.6.9]{bookHaase2006}} ]\label{thm:BIP}
Let $(\D(A),A)$ be a sectorial operator on $\X$ such that it has BIP. Then for all $\theta\in(0,1)$ and $\zeta:=(1-\theta)\alpha+\theta\beta$, it holds that
\begin{align*}
    [\D(A^\alpha),\D(A^\beta)]_{\theta} = \D(A^{\zeta})
\end{align*}
for all $0\leqslant \Re \alpha\leqslant \Re \beta$, with either $\Re \alpha>0$ or $\alpha =0$.
\end{proposition}

\medbreak

\noindent \textbf{Second}, it yields a simple way to reach \hyperref[eq:MaxRegLq]{(MR${}_{q}^T$)} for $0<T\leqslant \infty$ on \emph{UMD} Banach spaces (we will not go into details here, but the reader has to be aware that typically, for $1<p,q<\infty$, finite-dimensional-valued $\L^p$, $\H^{s,p}$, and $\B^{s}_{p,q}$ spaces satisfy such a property):

\medbreak

\begin{theorem}\label{thm:LqMaxRegUMD}
Let $(\D(A),A)$ be a sectorial operator of angle $\omega\in [0,\frac{\pi}{2})$ on a UMD Banach space $\X$ such that it has BIP of type $\theta_A\in[0,\frac{\pi}{2})$. Let $q\in(1,\infty)$.

Let $T\in(0,\infty)$. For all $f\in\L^q(0,T;\X)$ and all $u_0\in (\X,\D(A))_{1-1/q,q}$, the problem \eqref{ACP} admits a unique mild solution
$u\in \mathrm{C}^0_{ub}([0,T];(\X,\D(A))_{1-1/q,q})$
such that $\partial_t u$ and $Au$ belong to $\L^q(0,T;\X)$, with the estimates
\begin{align}
    \lVert u &\rVert_{\L^\infty(0,T;(\X,\D(A))_{1-1/q,q})}\nonumber\\
    &\quad\quad\quad\quad\quad\lesssim_{A,\X,q,T}
    \lVert (u,\partial_t u, Au)\rVert_{\L^q(0,T;\X)}\label{eq:BoundLqMaxReghomogeneous}\\
    &\quad\quad\quad\quad\quad\quad\quad\quad\quad\quad\quad\quad
    \lesssim_{A,\X,q,T}
    \lVert f\rVert_{\L^q(0,T;\X)} + \lVert u_0\rVert_{(\X,\D(A))_{1-1/q,q}}\nonumber\text{.}
\end{align}
Furthermore, if additionally $0\in \rho(A)$, \textit{i.e.} if $A$ is invertible, then the implicit constants are uniform with respect to $T$, and the result remains valid for $T=\infty$. In this latter case, there exists even $c>0$ such that the following estimate holds:
\begin{align*}
    \lVert e^{ct} u &\rVert_{\L^\infty(\RR_+;(\X,\D(A))_{1-1/q,q})}\\ 
    &\quad\quad\quad\quad\lesssim_{A,\X,q}
    \lVert e^{ct}(u,\partial_t u, Au)\rVert_{\L^q(\RR_+;\X)}\\
    &\quad\quad\quad\quad\quad\quad\quad\quad\quad\quad\quad
    \lesssim_{A,\X,q}
    \lVert f\rVert_{\L^q(\RR_+;\X)} + \lVert u_0\rVert_{(\X,\D(A))_{1-1/q,q}}\text{.}
\end{align*}
\end{theorem}

For more details about the functional calculus of sectorial operators and their link with maximal regularity, we refer to the general references
\cite{bookHaase2006,KunstmannWeis2004,bookDenkHieberPruss2003,PrussSimonett2016,HytonenNeervenVeraarWeisbookVolIII2023}.
The above theorem is very well known, but difficult to find stated in this form; see also the variation given by the author in \cite[Thm.~2.2]{Gaudin2023}.
In \eqref{eq:BoundLqMaxReghomogeneous}, the implicit constant is known to have polynomial growth.

% \medbreak

% Finally, last but not least, on real interpolation spaces built upon the domain of a sectorial operators, not only maximal regularity comes for free, as stated in Theorem~\ref{thm:DaPratoGrisvard}, but also the bounded holomorphic functional calculus and Boundedness of Imaginary Powers:

\smallskip
\par\noindent
{\bf Acknowledgement}.  The author is grateful to Raphaël Danchin for suggesting this interesting problem and for valuable discussions related to it. He is also indebted to him for pointing out a mistake in an earlier version of this work. The author also thanks Dominic Breit for helpful suggestions regarding the introduction.

\section*{Compliance with Ethical Standards}\label{conflicts}
\smallskip
\par\noindent
{\bf Conflict of Interest}. The author declares that he has no conflict of interest.

\smallskip
\par\noindent
{\bf Data Availability}. Data sharing is not applicable to this article as no datasets were generated or analysed during the current study.

%----------------------------------------------------
%------------------ Bibliography --------------------
%----------------------------------------------------
\typeout{}                                
\bibliographystyle{alpha}
{\footnotesize
\bibliography{Bibliography.bib}}

\end{document}